\documentclass[a4paper,12pt]{amsart}

\usepackage{amsmath,amssymb,amsthm,stmaryrd}
\usepackage{hyperref} 
\usepackage[T1]{fontenc}
\usepackage{latexsym}
\usepackage{enumerate}
\usepackage{fancyhdr}
\usepackage{latexsym}
\usepackage{dsfont}				
\usepackage{mathrsfs}			
\usepackage[active]{srcltx}		
\usepackage{esint}
\usepackage{pdfsync}

\usepackage{mathrsfs}
\usepackage{graphicx}

\usepackage[usenames,dvipsnames]{xcolor}

\usepackage[nospace,noadjust]{cite}

\newtheorem{Def}{Definition}[section]
\newtheorem{Lemma}[Def]{Lemma}

\newtheorem{Cor}[Def]{Corollary}
\newtheorem{Theorem}[Def]{Theorem}
\newtheorem{Prop}[Def]{Proposition}
\newtheorem{Remark}[Def]{Remark}

\DeclareMathOperator{\supp}{supp}
\DeclareMathOperator{\dist}{dist}

\DeclareMathOperator{\loc}{loc}

\DeclareMathOperator{\sgn}{sgn}

\numberwithin{equation}{section}

\let\Re=\relax
\DeclareMathOperator{\Re}{Re}

\let\div=\relax
\DeclareMathOperator{\div}{div}

\newcommand{\ovB}{\ensuremath{\overline{B}}}

\newcommand{\R}{\ensuremath{\mathbb{R}}}
\newcommand{\C}{\ensuremath{\mathbb{C}}}
\newcommand{\N}{\ensuremath{\mathbb{N}}}
\newcommand{\Z}{\ensuremath{\mathbb{Z}}}

\newcommand{\Eins}{\ensuremath{\mathds{1}}}

\newcommand{\Pb}{\ensuremath{\mathbb{P}}}
\newcommand{\Qb}{\ensuremath{\mathbb{Q}}}

\newcommand{\calD}{\ensuremath{\mathcal{D}}}

\newcommand{\calM}{\ensuremath{\mathcal{M}}}
\newcommand{\calN}{\ensuremath{\mathcal{N}}}

\newcommand{\calR}{\ensuremath{\mathcal{R}}}
\newcommand{\calQ}{\ensuremath{\mathcal{Q}}}

\newcommand{\calL}{\ensuremath{\mathcal{L}}}

\newcommand{\skp}[1]{\langle #1 \rangle}
\newcommand{\eps}{\ensuremath{\varepsilon}}

\usepackage{comment}

\usepackage{color}

\definecolor{gr}{rgb}   {0.,   0.8,   0. } 
\definecolor{bl}{rgb}   {0.,   0.5,   1. } 
\definecolor{mg}{rgb}   {0.7,  0.,    0.7}

\title[Conical square function estimates for Dirac operators]{Conical square function estimates and functional calculi for perturbed Hodge-Dirac operators in $L^p$ }
\author{Dorothee Frey}
\author{Alan McIntosh}
\author{Pierre Portal}

\begin{document}

\begin{abstract}
Perturbed Hodge-Dirac operators and their holomorphic functional calculi, as investigated in the papers by Axelsson, Keith and the second author, provided insight into the solution of the Kato square-root problem for elliptic operators in $L^2$ spaces, and allowed for an extension of these estimates to other systems with applications to non-smooth boundary value problems.  In this paper, we determine conditions under which such operators satisfy conical square function estimates in a range of $L^p$ spaces, thus allowing us to apply the theory of Hardy spaces associated with an operator, to prove that they have a bounded holomorphic functional calculus in those $L^p$ spaces.  We also obtain functional calculi results for restrictions to certain subspaces, for a larger range of $p$.   This provides  a framework for obtaining $L^p$ results on perturbed Hodge Laplacians, generalising known Riesz transform bounds for an elliptic operator $L$  with bounded measurable coefficients, one Sobolev exponent below the Hodge exponent, and  $L^p$ bounds on the square-root of $L$ by the gradient, two Sobolev exponents below the Hodge exponent.  
Our proof shows that the heart of the harmonic analysis in $L^2$  extends to $L^p$ for all $p \in (1,\infty)$, while the restrictions in $p$ come from the operator-theoretic part of the $L^2$ proof.  In the course of our work, we obtain some results of independent interest about singular integral operators on tent spaces, and about the relationship between  conical and vertical square functions.\\

{\bf Mathematics Subject Classification (2010):} 47A60, 47F05, 42B30, 42B37\\
\end{abstract}

\date{ February 27, 2015 }

\maketitle

\tableofcontents

\section{Introduction}

In \cite{AKM}, Axelsson, Keith, and the second author introduced a general framework to study various harmonic analytic problems, such as boundedness of Riesz transforms or the construction of solutions to boundary value problems, through the holomorphic functional calculus of certain first order differential operators that generalise the Hodge-Dirac operator $d+d^{*}$ (where $d$ is the exterior derivative) of Riemannian geometry. By proving that such Hodge-Dirac operators have a bounded holomorphic functional calculus in $L^2$, they recover, in particular, the solution of Kato's square root problem 
obtained by Auscher, Hofmann, Lacey, McIntosh and Tchamitchian in \cite{ahlmt}. Their results also provide the harmonic analytic foundation to new approaches to problems in PDE (see e.g. \cite{AAMc, aah,aa}) and geometry (see e.g. \cite{AMR}).
\\

The main result in \cite{AKM} is of a perturbative nature. Informally speaking, it states that the functional calculus of the standard Hodge-Dirac operator in $L^2$ is stable under perturbation by rough coefficients. It is natural, and important in applications, to know whether or not such a result also holds in $L^p$ for $p \in (1,\infty)$.
There are two main approaches to this question. The first one uses the extrapolation method pioneered by Blunck and Kunstmann in \cite{bk}, and developed by Auscher in \cite{AuscherMemoirs} to show that the relevant $L^2$ bounds remain valid in certain intervals $(p_{-},p_{+})$ about $2$ which depend on the operator involved. This approach has been mostly developed to study second order differential operators, but has also been adapted to first order operators by Ajiev \cite{aj} and by Auscher and Stahlhut in  \cite{as1,as2}.  
The other approach to $L^p$ estimates for the holomorphic functional calculus of Hodge-Dirac operators consists in adapting the entire machinery of \cite{AKM} to $L^p$. This was done in the series of papers \cite{HMcP1,HMcP2,HMc} by the second and third authors, together with Hyt\"onen, using ideas from (UMD) Banach space valued harmonic analysis. \\

At the technical level, all these results are fundamentally perturbation results for square function estimates. In $L^2$, the heart of \cite{AKM} is an estimate of the form
$$
(\int \limits _{0} ^{\infty} \int \limits _{\R^{n}} |t\Pi_{B}(I+t^{2}{\Pi_{B}}^2)^{-1}u(x)|^{2} \frac{dx dt}{t})^{\frac{1}{2}} 
\lesssim (\int \limits _{0} ^{\infty} \int \limits _{\R^{n}} |t\Pi(I+t^{2}\Pi^2)^{-1}u(x)|^{2} \frac{dx dt}{t})^{\frac{1}{2}} \quad \forall u \in \calR(\Gamma),
$$ 
where $\Pi = \Gamma + \Gamma^{*}$ is a first order differential (Hodge-Dirac) operator with constant coefficients, and $\Pi_{B} =  \Gamma + B_{1}\Gamma^{*}B_{2}$ is a perturbation by $L^\infty$ coefficients $B_{1},B_{2}$. See Section \ref{sec:prelim} for precise definitions.
In $L^p$, the papers \cite{HMcP1,HMcP2,HMc} establish analogues of the form
$$
\|(\int \limits _{0} ^{\infty} |t\Pi_{B}(I+t^{2}{\Pi_{B}}^2)^{-1}u(.)|^{2} \frac{dt}{t})^{\frac{1}{2}}\|_{p} 
\lesssim \|(\int \limits _{0} ^{\infty} |t\Pi(I+t^{2}\Pi^2)^{-1}u(.)|^{2} \frac{dt}{t})^{\frac{1}{2}}\|_{p}  \quad \forall u \in \calR(\Gamma).
$$ 
While these (vertical) $L^p$ square function estimates are traditionally used to establish the boundedness of the holomorphic functional calculus (see e.g. \cite{CDMcY}),
the same result could also be obtained using the
conical $L^p$ square function estimates:
$$
\|(t,x)\mapsto (t\Pi_{B}(I+t^{2}{\Pi_{B}}^2)^{-1})^{M}u(x)\|_{T^{p,2}}
\lesssim \|(t,x)\mapsto (t\Pi(I+t^{2}\Pi^2)^{-1})^{M}u(x)\|_{T^{p,2}} \quad \forall u \in \calR(\Gamma),
$$ 
where $M$ is a suitably large integer and $T^{p,2}$ is one of Coifman-Meyer-Stein's tent spaces (see \cite{cms} and Section \ref{sec:prelim} for precise definitions).
This fact has been noticed in the development of a Hardy space theory associated with bisectorial operators (starting with \cite{AMR,DY,hm}, see also \cite[Theorem 7.10]{HvNP}).\\

In this paper, we prove such conical $L^p$ square function estimates for the Hodge-Dirac operators introduced in \cite{AKM}. 
This allows us to strengthen the results from \cite{HMcP1,HMcP2,HMc} (in the scalar-valued setting) by eliminating the R-boundedness assumptions.  Instead of relying on  probabilistic/dyadic methods, we use the more flexible theory of Hardy spaces associated with operators, and
recent results about integral operators on tent spaces. 
Our proof then exhibits an interesting phenomenon. As in \cite{AKM} and other papers on functional calculus of Hodge-Dirac operators or Kato square root estimates,
we consider separately the ``high frequency" part of the estimate (involving $\|(t,x)\mapsto (t\Pi_{B}(I+t^{2}{\Pi_{B}}^2)^{-1})^{M}(I+t^{2}\Pi^2)^{-M})u(x)\|_{T^{p,2}}$),
and the ``low frequency" part (involving $\|(t,x)\mapsto (t\Pi_{B}(I+t^{2}{\Pi_{B}}^2)^{-1})^{M}(I-(I+t^{2}\Pi^2)^{-M})u(x)\|_{T^{p,2}}$).
In $L^2$, the proof of the high frequency estimate is purely operator theoretic, while the low frequency requires the techniques from real analysis used in the solution of the Kato square root problem.
In the approach to the $L^p$ case given in \cite{HMcP1,HMcP2,HMc}, the same is true, but both the high and the low frequency estimate use an extra assumption: the R-bisectoriality of $\Pi_{B}$ in $L^p$. With the approach through conical square function given here, we obtain the low frequency estimate for all $p \in (1,\infty)$ without any assumption on the $L^p$ behaviour of the operator $\Pi_{B}$. Restrictions in $p$, and appropriate assumptions (which are necessary, as can be seen in \cite{AuscherMemoirs}), are needed for the high frequency part. We believe that this will be helpful in future projects, as the theory moves away from the Euclidean setting (see e.g. the work of Morris \cite{Mor}, Bandara and the second author \cite{lashi}). Dealing with a specific Hodge-Dirac operator in a geometric context, one can hope to prove sharp high frequency estimates using methods specific to the context at hand, and combine them with the harmonic analytic machinery developed here to get the full square function estimates, and hence the functional calculus result.\\

Another feature of the approach given here is that we obtain, from $L^p$ assumptions, not just functional calculus results in $L^p$, but also functional calculus results on some subspaces of $L^q$ for certain $q <p$. In particular, we obtain Riesz transform estimates for $q \in (p_{*},2]$, and reverse Riesz transform estimates for $q \in (p_{**},2]$. Here $p_\ast$ and $p_{\ast\ast}$ denote the first and second Sobolev exponents below $p$.
This can also be  relevant in geometric settings, where one expects the results to depend not only on  the geometry, but on the  different levels of forms.\\

The paper is organised as follows. In Section \ref{sec:prelim}, we give the relevant definitions and recall the main results from the theories that this paper builds upon.
In Section \ref{sec:res}, we state our main results - relevant high and low frequency square function estimates - and establish their functional calculus consequences as corollaries in Section \ref{sec:cons}. In Section \ref{sec:low}, we prove  low frequency estimates by developing $L^p$ conical square function versions of the tools used in \cite{AKM}.
In Section \ref{sec:high1}, we prove  high frequency estimates for $p \in (\max(1,2_{*}),2]$. In this range, the proof is straightforward, and does not require any $L^p$ assumption. In dimensions $1$ and $2$ this already gives the result for all $p \in (1,\infty)$.
In Section \ref{sec:high2} we establish the relevant $L^p$-$L^2$ off-diagonal bounds for the resolvents of our Hodge-Dirac operator. 
 In Section \ref{sec:tech2}, we use  them to bound a conical square function by a vertical square function 
related to the functional calculus of our Hodge-Dirac operator.
In Section \ref{sec:high}, we use these off-diagonal bounds to prove the high frequency estimates. This uses singular integral operator theory on tent spaces, and, in particular,  Schur-type extrapolation results established in Section \ref{sec:tech1}.  We believe the latter results are of independent interest. \\

\subsection{Acknowledgments}
All three authors gratefully acknowledge support from the Australian Research Council 
through the Discovery Project  DP120103692. This work is a key outcome of DP120103692. Frey and McIntosh also acknowledge support from ARC DP110102488.
Portal is further supported by the ARC through the Future Fellowship FT130100607.
 The authors  thank  Pascal Auscher and Sylvie Monniaux for stimulating discussions, and Pascal Auscher in particular  for keeping us aware of the progress of his student Sebastian Stahlhut on related questions. There is a connection between the results in \cite{as2} by Auscher and Stahlhut and our results, though the approaches are rather different because Auscher and Stahlhut rely on the results from \cite{HMcP1,HMcP2,HMc} through \cite{as1}, while one of our aims is to give an alternative approach to \cite{HMcP1,HMcP2,HMc}. We remark that Auscher and Stahlhult  apply their results to develop an extensive theory of a priori estimates for related non-smooth boundary value problems.
  We are also grateful to the anonymous referee for giving us the opportunity to correct a mistake in the original manuscript. 

\section{Preliminaries}
\label{sec:prelim}

\subsection{Notation}
Throughout the paper $n$ and $N$ denote two fixed positive natural numbers.
We express inequalities ``up to a constant" between two positive quantities $a,b$ with the notation $a \lesssim b$. By this we mean that there exists a constant $C>0$, independent of all relevant quantities in the statement, such that $a \leq Cb$. If $a\lesssim b$ and $b \lesssim a$, we write $a \approx b$.\\
We denote $\R^{*}=\R\setminus \{0\}$.
For a Banach space $X$, we write $\calL(X)$ for the set of all  bounded linear operators on $X$. \\
 For $p \in (1,\infty)$ and an unbounded linear operator $A$ on $L^p(\R^n;\C^N)$, we denote by $\calD_p(A),\calR_p(A), \calN_p(A)$ its domain, range and null space, respectively.\\
We use upper and lower stars to denote Sobolev exponents: For $p \in [1,\infty)$, we denote $p_{*} = \frac{np}{n+p}$ and $p^{*}=\frac{np}{n-p}$, with the convention $p^{*}=\infty$ for $p\geq n$.

For a ball (resp. cube) $B \subseteq \R^n$ with radius (resp. side length) $r>0$ and given $\alpha>0$, we write $\alpha B$ for the ball (resp. the cube) with the same centre and radius (resp. side length) $\alpha r$. 
We define dyadic shells by $S_1(B):=4B$ and $S_j(B):=2^{j+1}B \setminus 2^jB$ for $j \geq 2$.

\subsection{Holomorphic functional calculus}

This paper deals with the holomorphic functional calculus of certain bisectorial first order differential operators.
The fundamental results concerning this calculus have been developed in \cite{m,CDMcY, AMcN,kw}.  References for this theory include the lecture notes \cite{ADM} and \cite{kuw}, and the book \cite{markusbook}.

\begin{Def}
Let $0 \leq \omega <\mu<\frac{\pi}{2}$. Define  {\em closed and open sectors and double sectors} in the complex plane by
\begin{align*}
	S_{\omega +} &:= \{z \in \C\,:\, |\arg z| \leq \omega \} \cup \{0\}, \qquad 
	S_{\omega -} := - S_{\omega +}, \\
	S_{\mu +}^o &:= \{z \in \C \,:\, z \neq 0, \, |\arg z| <\mu\}, \qquad 
	S_{\mu -}^o := - S_{\mu +}^o, \\ 
	S_\omega &:= S_{\omega +} \cup S_{\omega -}, \qquad
	S_\mu^o := S_{\mu +}^o \cup S_{\mu -}^o.	
\end{align*}
\end{Def}

Denote by $H(S_\mu^o)$ the space of all holomorphic functions on $S_\mu^o$. Let further
\begin{align*}
	H^\infty(S_\mu^o) & := \{\psi \in H(S_\mu^o) \,:\, \|\psi\|_{L^\infty(S_\mu^o)} <\infty\},\\
	\Psi_\alpha^\beta(S_\mu^o) &:= \{\psi \in H(S_\mu^o) \,:\, \exists C>0: |\psi(z)| \leq C |z|^\alpha(1+|z|^{\alpha+\beta})^{-1} \ \forall z \in S_\mu^o\}
\end{align*}
for every $\alpha, \beta>0$, and set $\Psi(S_\mu^o):=\bigcup_{\alpha,\beta>0}\Psi_\alpha^\beta(S_\mu^o)$. We say that $\psi \in \Psi(S_\mu^o)$ is \emph{non-degenerate} if neither of the restrictions $\psi|_{S_{\mu \pm}^o}$ vanishes identically.\\

\begin{Def} Let $0 \leq \omega <\frac{\pi}{2}$.
A closed operator $D$ acting on a Banach space $X$ is called $\omega${\em-bisectorial}  if $\sigma(D) \subset S_{\omega}$, and for all $\theta \in (\omega,\frac{\pi}{2})$ there exists $C_\theta>0$ such that
$$
\|\lambda(\lambda I-D)^{-1}\|_{\calL(X)} \leq C_\theta \quad \forall \lambda \in \C \setminus S_{\theta}.
$$
\end{Def}

We  say that $D$ is {\em bisectorial} if it is $\omega$-bisectorial for some $\omega \in [0, \frac{\pi}{2})$.

For $D$ bisectorial with angle $\omega \in [0,\frac{\pi}{2})$ and $\psi \in \Psi(S_\mu^o)$ for $\mu \in (\omega,\frac{\pi}{2})$, we  define $\psi(D)$ through the Cauchy integral
$$
\psi(D) = \frac{1}{2\pi i} \int \limits _{\gamma} \psi(z)(zI-D)^{-1} dz,
$$
where $\gamma$ denotes the boundary of $S_{\theta}$ for some $\theta \in (\omega,\mu)$, oriented counter-clockwise.

\begin{Def} Let $0 \leq \omega <\frac{\pi}{2}$ and $\mu \in (\omega, \frac{\pi}{2})$.
An  $\omega$-bisectorial  operator $D$, acting on a Banach space $X$, is said to have a {\em bounded $H^{\infty}$ functional calculus with angle} $\mu$ if
there exists $C>0$ such that for all $\psi \in \Psi(S_{\mu}^{0})$
$$
\|\psi(D)\|_{\calL(X)} \leq C \|\psi\|_{\infty}.
$$
\end{Def}

For such an operator, the functional calculus extends to a bounded algebra homomorphism from $H^{\infty}(S^o_{\mu})$ to $\calL(X)$. More precisely, for all bounded functions $f: S_\mu^o \cup \{0\} \to \C$ which are holomorphic on $S_\mu^o$, one can define a bounded operator $f(D)$ by
\[
	f(D)u = f(0) \Pb_{\calN(D)} u + \lim_{n \to \infty} \psi_n (D) u, \qquad u \in X,
\]
where $\Pb_{\calN(D)}$ denotes the bounded projection onto $\calN(D)$ with null space $\overline{\calR(D)}$, and the functions $\psi_n \in \Psi(S_\mu^o)$ are uniformly bounded and tend locally uniformly to $f$ on $S_\mu^o$; see \cite{ADM,CDMcY}.
The definition is independent of the choice of the approximating sequence $(\psi_n)_{n \in \N}$.

\subsection{Off-diagonal bounds}
The operator theoretic property that captures the relevant aspect of the differential nature of our operators, 
 is the following notion of off-diagonal bounds. This notion plays a central role in many current developments of singular integral operator theory. We refer to \cite{AuscherMemoirs} for more information and references.

\begin{Def}
Let $p \in [1,2]$.
A family of operators $\{U_t \;;\; t \in \R^{*}\} \subset \calL(L^{2}(\R^{n};\C^{N}))$ is said to have $L^{p}$-$L^{2}$ off-diagonal bounds of order $M>0$ if
there exists $C_M>0$ such that for all $t\in \R^{*}$, all Borel sets $E,F \subseteq \R^n$ and all $u \in L^p(\R^n;\C^N)$ with $\supp u \subseteq F$, we have
\begin{equation} \label{eq:off-diag-est}
	\|U_t u\|_{L^2(E)} \leq C_M |t|^{-n(\frac{1}{p}-\frac{1}{2})}\left(1+\frac{\dist(E,F)}{|t|}\right)^{-M} \|u\|_{L^p(F)},
\end{equation}
where $\dist(E,F)=\inf\{|x-y|; \ x \in E, y \in F\}$.
\end{Def}

We also use the following variant, where,  given a ball $B$,  $$\dot W^{1,p}_{\overline B}(\R^n;\C^N)=\{u \in \dot W^{1,p}(\R^n;\C^N) ; \supp(u) \subset \ovB\}  $$ with norm $\|\nabla u\|_p$.

\begin{Def}
Let $p \in [1,2]$.
A family of operators $\{U_t \;;\; t \in \R^{*}\} \subset \calL(L^{2}(\R^{n};\C^{N}))$ is said to have $\dot{W}^{1,p}$-$L^{2}$ off-diagonal bounds of order $M>0$ on balls, if
there exists $C_M>0$ such that for all $t \in \R^{*}$, all balls $B$ of radius $|t|$, all $j \in \N$, and all $u \in {W}_{\ovB} ^{1,p}(\R^n;\C^N)$ we have
\begin{equation} \label{eq:off-diag-est}
	\|U_t u\|_{L^2(S_{j}(B))} \leq C_M |t|^{-n(\frac{1}{p}-\frac{1}{2})}2^{-jM} \|\nabla u\|_{L^p},
\end{equation}
\end{Def}

The following properties of off-diagonal bounds with respect to composition and interpolation are essentially known (see \cite{am2}).
We nonetheless include some proofs.

\begin{Lemma} 
\label{lem:ODcomp} 
Let $p \in (1,2]$.
Let $\{T_t \;;\; t \in \R^{*}\} \subset \calL(L^{2}(\R^{n};\C^{N}))$ have $L^{2}$-$L^{2}$ off-diagonal bounds of every order,
$\{V_t \;;\; t \in \R^{*}\} \subset \calL(L^{2}(\R^{n};\C^{N}))$ have $L^{2}$-$L^{2}$ off-diagonal bounds of order M>0, 
  $\{U_t \;;\; t \in \R^{*}\} \subset \calL(L^{2}(\R^{n};\C^{N}))$ have  $\dot{W}^{1,2}$-$L^{2}$ off-diagonal bounds  of every order on balls, 
    $\{Z_t \;;\; t \in \R^{*}\} \subset \calL(L^{2}(\R^{n};\C^{N}))$ have $\dot{W}^{1,p}$-$L^{2}$ off-diagonal bounds  of every order on balls, 
and $\{S_t \;;\; t \in \R^{*}\} \subset \calL(L^{2}(\R^{n};\C^{N}))$ have $L^{p}$-$L^{2}$ off-diagonal bounds of order $M$,
Then
\begin{enumerate}
\item If $\underset{t \in \R^{*}}{\sup} \,\||t|^{n(\frac{1}{p}-\frac{1}{2})}T_{t}\|_{\calL(L^{p},L^2)}<\infty$, then for all $q \in (p,2]$,
$\{T_t \;;\; t \in \R^{*}\} $ has $L^{q}$-$L^{2}$ off-diagonal bounds of every order. 
\item
 
If $\underset{t \in \R^{*}}{\sup} \ \underset{B=B(x,|t|)}{\sup} \,\||t|^{n(\frac{1}{p}-\frac{1}{2})}U_{t}\|_{\calL(\dot{W}^{1,p}_{\overline{B}},L^2)}<\infty$,  then for all $q \in (p,2]$,
$\{U_t \;;\; t \in \R^{*}\} $ has $\dot{W}^{1,q}$-$L^{2}$ off-diagonal bounds of every order on balls. 

\item If $\{T_t \;;\; t \in \R^{*}\} $ has $L^{p}$-$L^{q}$ off-diagonal bounds of every order for some $q \in [p,2]$, then $\underset{t \in \R^{*}}{\sup}\,\|T_{t}\|_{\calL(L^{p})}<\infty$.
\item  $\{V_t S_{t}\;;\; t \in \R^{*}\}$ has $L^{p}$-$L^{2}$ off-diagonal bounds of order $M$,  and $\{T_t Z_{t}\;;\; t \in \R^{*}\}$ has $\dot{W}^{1,p}$-$L^{2}$ off-diagonal bounds of every order on balls. 
\item For all $t \in \R^{*}$, $T_t$ extends to an operator $T_t: L^\infty(\R^n;\C^N) \to L^2_{\loc}(\R^n;\C^N)$ with
$$
	\|T_t u\|_{L^2(B(x_0,t))} \lesssim |t|^{\frac{n}{2}} \|u\|_{\infty} \quad \forall u \in L^\infty(\R^n;\C^N), \; x_0 \in \R^n.
$$
\end{enumerate}
\end{Lemma}

\begin{proof}
For (1),  use Stein's interpolation \cite[Theorem 1]{Stein} for the analytic family of operators
$\{|t|^{n(\frac{1}{p}-\frac{1}{2})z}(1+\frac{\dist(E,F)}{|t|})^{M'(1-z)}\Eins_{E}T_{t}\Eins_{F} \;;\; z \in S\}$, where $S = \{z \in \C \;;\; Re(z) \in [0,1]\}$ and $M' \in \N$.
For $q \in [p,2]$, this gives $L^q$-$L^2$ off-diagonal bounds of order $\max(0,M'(1-\frac{\frac{1}{q}-\frac{1}{2}}{\frac{1}{p}-\frac{1}{2}}))$, which implies the result by choosing $M'$ large enough.  A similar proof gives (2), this time using the fact that the spaces $\dot W^{1,p}_{\ovB}(\R^n;\C^N)$ interpolate in $p$ by the complex method. \\

We refer to \cite[Lemma 3.3]{AuscherMemoirs}  for a proof of (3).\\

We now turn to (4).
Let $E,F\subset \R^{n}$ be two Borel sets, and $t \in \R^{*}$.
Set $\delta=\dist(E,F)$ and $G=\{x \in \R^n\,;\,\dist(x,F)<\frac{\delta}{2}\}$. Then $\dist(E,\overline{G})\geq \frac{\delta}{2}$ and $\dist(\R^n\setminus G,F) \geq \frac{\delta}{2}$. Observe that the assumptions on $V_t$ and $S_t$ in particular imply that $\sup_{t\in \R^{*}}\|V_t\|_{\calL(L^2)} <\infty$ and $\sup_{t\in \R^{*}}\||t|^{n(\frac{1}{p}-\frac{1}{2})}S_t\|_{\calL(L^p,L^2)}<\infty$ (taking $E=F=\R^{n}$ in the definition of off-diagonal bounds).
 We have the following for all $u \in L^p$:
\begin{align*}
\|\Eins_{E}V_{t}S_{t}\Eins_{F}u\|_{2} 
& \leq \|\Eins_{E}V_{t}\Eins_{G}S_{t}\Eins_{F}u\|_{2} + \|\Eins_{E}V_{t}\Eins_{\R^n \setminus G}S_{t}\Eins_{F}u\|_{2}\\
& \lesssim (1+\frac{\dist(E,\overline{G})}{|t|})^{-M} 
\|\Eins_GS_t\Eins_Fu\|_2 + \|\Eins_{\R^n \setminus G}S_t\Eins_F u\|_2
\\
& \lesssim |t|^{-n(\frac{1}{p}-\frac{1}{2})} ((1+\frac{\dist(E,\overline{G})}{|t|})^{-M} + (1+\frac{\dist(\R^n\setminus G,F)}{|t|})^{-M} )\|\Eins_{F}u\|_{p}\\
&\lesssim |t|^{-n(\frac{1}{p}-\frac{1}{2})}(1+\frac{\dist(E,F)}{|t|})^{-M} \|\Eins_{F}u\|_{p}.
\end{align*}
This proves (4)  for $\{V_t S_{t}\;;\; t \in \R^{*}\}$. For $\{T_t Z_{t}\;;\; t \in \R^{*}\}$, a ball $B$  of radius $|t|$,  $j \in \N$, and $u \in \dot{W}^{1,p}_{\ovB}$, we have, for all $N\in\N$, that
\begin{align*}
\|\Eins_{S_{j}(B)}T_{t}Z_{t}u\|_{2} &\leq \sum \limits _{k=1} ^{\infty} \|\Eins_{S_{j}(B)}T_{t}\Eins_{S_{k}(B)}Z_{t}u\|_{2}
\lesssim \sum \limits _{k=1} ^{\infty} 2^{-|j-k|(N+1)}\|\Eins_{S_{k}(B)}Z_{t}u\|_{2}\\
& \lesssim |t|^{-n(\frac{1}{p}-\frac{1}{2})}\sum \limits _{k=1} ^{\infty} 2^{-|j-k|(N+1)}2^{-k(N+1)}\|\nabla u\|_{p} \lesssim |t|^{-n(\frac{1}{p}-\frac{1}{2})}2^{-jN}\|\nabla u\|_{p}.
\end{align*}
(5) The extension $T_t: L^\infty(\R^n;\C^N)\to L^2_{\loc}(\R^n;\C^N)$ can be defined as
$$
	\Eins_Q (T_t u) = \lim_{\rho \to \infty} \sum_{\substack{R \in \Delta_{|t|} \\ \dist(Q,R) <\rho}} \Eins_Q (T_t (\Eins_R u)),
$$
where $u \in L^\infty(\R^n;\C^N)$, and $Q \in \Delta_{|t|}$ a dyadic cube in $\R^n$ (see the beginning of Section \ref{sec:low} for a definition of $\Delta_{|t|}$). 
It is shown in \cite[Corollary 5.3]{AKM} that the limit exists and the extension is well-defined. 
\end{proof}

The next lemma was shown in \cite[Lemma 7.3]{HvNP} (as in \cite[Lemma 3.6]{AMR}).

\begin{Lemma}
\label{lem:psiOD}
Suppose $M>0$, $0\leq\omega<\theta<\mu<\frac{\pi}{2}$. 
Let $D$ be an $\omega$-bisectorial operator in $L^2(\R^n;\C^N)$ such that
$\{z(zI-D)^{-1}\,;\,z\in\C\setminus S_\theta\}$ has $L^2$-$L^2$ off diagonal bounds of order $M$ in the sense that
\begin{equation*} 
	\| z(zI-D)^{-1}u\|_{L^2(E)} \leq C_M\left(1+|z|\dist(E,F)\right)^{-M} \|u\|_{2}
\end{equation*}
for all $z\in \C\setminus S_\theta$, all Borel subsets $E, F\subseteq\R^n$, and all $u \in L^2(\R^n;\C^N)$ with $\supp u \subseteq F$. 
If $\psi \in \Psi_{\beta} ^{\alpha} (S^o_{\mu})$ for   $\alpha>0$, $\beta>M$, then $\{\psi(tD)\;;\; t \in \R^{*}\}$ has $L^2$-$L^2$ off-diagonal bounds  of order $M$.
\end{Lemma}

\subsection{Tent spaces}
Recall that the tent space $T^{p,2}(\R^{n+1}_+)$, first introduced by Coifman, Meyer, and Stein in \cite{cms}, is the completion of $C_{c}^{\infty}(\R^{n+1}_{+})$ with respect to the norm
$$
\|F\|_{T^{p,2}}  = ( \int \limits _{\R^{n}} (\int \limits _{0} ^{\infty} \int \limits _{B(x,t)} |F(t,y)|^{2} \,\frac{dydt}{t^{n+1}})^{\frac{p}{2}} \,dx)^{\frac{1}{p}}
$$
for $p \in [1,\infty)$, and with respect to the norm
$$
	\|F\|_{T^{\infty,2}} = \sup_{(r,x) \in \R_+ \times \R^n} (r^{-n} \int_0^r \int_{B(x,r)} |F(t,y)|^2 \, \frac{dydt}{t})^{\frac12}
$$
for $p=\infty$. \\

The tent spaces interpolate by the complex method, in the sense that $[T^{p_{0},2},T^{p_{1},2}]_{\theta} = T^{p_{\theta},2}$ for $\theta \in [0,1]$ and 
$\frac{1}{p_{\theta}}=\frac{1-\theta}{p_{0}}+\frac{\theta}{p_{1}}$.
We recall a basic result about tent spaces, and another about operators acting on them.

\begin{Lemma}
\cite{a}
\label{lem:angle}
Let $p \in [1,\infty)$, $\alpha \geq 1$ and $T^{p,2}_{\alpha}(\R^{n+1}_+)$ denote the completion of $C_{c}^{\infty}(\R^{n+1}_{+})$ with respect to the norm
$$
\|F\|_{T^{p,2}_{\alpha}}  = ( \int \limits _{\R^{n}} (\int \limits _{0} ^{\infty} \int \limits _{B(x,\alpha t)} |F(t,y)|^{2} \,\frac{dydt}{t^{n+1}})^{\frac{p}{2}} \,dx )^{\frac{1}{p}}.
$$
Then $T^{p,2}_{\alpha}(\R^{n+1}_+) = T^{p,2}(\R^{n+1}_+)$ with the equivalence of norms
$$
\|F\|_{T^{p,2}} \leq \|F\|_{T^{p,2}_{\alpha}} \lesssim \alpha^{\frac{n}{\min\{p,2\}}}\|F\|_{T^{p,2}} \quad \forall F \in T^{p,2}(\R^{n+1}_+).
$$
\end{Lemma}

\begin{Lemma} \label{tent-boundedness}
\cite[Theorem 5.2]{HvNP}
Let $p \in (1,\infty)$. Let $\{T_t\}_{t>0}$ be a family of operators acting on $L^2(\R^n)$ with $L^2$-$L^2$ off-diagonal bounds of order $M>\frac{n}{min\{p,2\}}$. 
Then there exists $C>0$ such that for all $F \in T^{p,2}(\R^{n+1}_+)$
\[
	\|(t,x)\mapsto T_tF(t,\,.\,)(x)\|_{T^{p,2}} \leq C \|F\|_{T^{p,2}}.
\]
\end{Lemma}

\subsection{Hardy spaces associated with bisectorial operators}
We consider Hardy spaces associated with bisectorial operators. 
We refer to \cite{AMR,DY,HvNP,hm,HMMc,HLMMY} and the references therein for more details about such spaces, and just recall here the main definition and result.\\

Let $0 \leq \omega <\mu<\frac{\pi}{2}$, and $D$ be an  $\omega$-bisectorial operator in $L^{2}(\R^{n};\C^{N})$ 
such that
$\{(I+itD)^{-1} \;;\; t \in \R\setminus\{0\}\}$ has $L^{2}$-$L^{2}$ off-diagonal bounds of order $M>\frac{n}{2}$.
Assume further that $D$ has a bounded $H^{\infty}$ functional calculus with angle $\theta \in (\omega, \mu)$.
Given $u \in L^2(\R^n;\C^N)$ and $\psi \in \Psi(S_\mu^o)$, write
\[
	\calQ_\psi u(x,t):= \psi(tD)u(x), \qquad x \in \R^n,\; t>0.
\]

\begin{Def}
Let $p \in [1,\infty)$, let $\psi \in \Psi(S_\mu^o)$ be non-degenerate. 
The Hardy space $H^p_{D,\psi}(\R^n;\C^N)$ associated with $D$ and $\psi$ is the completion of the space 
\[
	\{ u \in \overline{\calR_2(D)} \,:\, \calQ_\psi u \in T^{p,2}(\R^{n+1}_{+};\C^N)\}
\]
with respect to the norm
\[
	\|u\|_{H^p_{D,\psi}} := \|\calQ_\psi u\|_{T^{p,2}}.
\]
\end{Def}

Let us also recall \cite[Theorem 7.10]{HvNP}:

\begin{Theorem} \label{thm:hardy}
Let $\eps>0$. Let $p \in (1,2]$ and $\psi,\tilde{\psi} \in \Psi_\eps^{\frac{n}{2}+\eps}(S_\mu^o)$, or $p \in [2,\infty)$ and $\psi,\tilde{\psi} \in \Psi_{\frac{n}{2}+\eps}^\eps(S_\mu^o)$, where $\mu>\omega$ and  both $\psi$ and $\tilde\psi$ are  non-degenerate.
Then
\begin{enumerate}
\item
$H^p_{D,\psi}(\R^n;\C^N) = H^p_{D,\tilde{\psi}}(\R^n;\C^N) =: H^p_{D}(\R^n;\C^N)$;
\item
For all $u \in H^p_{D}(\R^n;\C^N)$, and all $f \in \Psi(S^{o}_{\mu})$, we have 
$$
\|(t,x) \mapsto \psi(tD)f(D)u(x)\|_{T^{p,2}} \lesssim \|f\|_{\infty}\|u\|_{H^p_{D}}.
$$
In particular, $D$ has a bounded $H^\infty$ functional calculus on $H^p_{D}(\R^n;\C^N)$.
\end{enumerate}
\end{Theorem}

\subsection{Hodge-Dirac operators}

Throughout the paper, we work with the following class of Hodge-Dirac operators. It is a slight modification of the classes considered in \cite{AKM} and \cite{HMcP2}.

\begin{Def} A {\em Hodge-Dirac operator with constant coefficients} is
an operator of the form $\Pi = \Gamma + \Gamma^\ast$,
where $\Gamma=-i\sum_{j=1}^n\hat{\Gamma}_j\partial_j$ is a Fourier multiplier with symbol defined by
\begin{equation*}
\hat{\Gamma}=\hat{\Gamma} (\xi)=\sum_{j=1}^n \hat{\Gamma}_j\xi_j  \qquad  \forall \xi \in \R^{n}, 
\end{equation*}
with $\hat{\Gamma}_{j} \in \mathscr{L}(\C^{N})$,
the operator $\Gamma$ is nilpotent, i.e.
\(\hat{\Gamma}(\xi)^2=0\) for all \(\xi\in\R^n\), and there exists $\kappa >0$ such that
\begin{equation*}\tag{$\Pi$1}\label{Pi1}
  \kappa|\xi||w|\leq|\hat{\Pi}(\xi)w| \qquad
  \forall w\in \calR(\hat{\Pi}(\xi)),\; \forall \xi \in \R^{n}.
\end{equation*}
\end{Def}

We list some results about these operators.

\begin{Prop}
\label{prop:pi}
Suppose $p \in (1,\infty)$.\\
(1) The operator identity $\Pi=\Gamma + \Gamma^\ast$ holds in $L^p(\R^n;\C^N)$, in the sense
that $\calD_p(\Pi)= \calD_p(\Gamma) \cap \calD_p(\Gamma^\ast)$ and $\Pi u= \Gamma u + \Gamma^\ast u$ for all $u \in \calD_p(\Pi)$.\\
(2) There holds $\calN_p(\Pi)=\calN_p(\Gamma) \cap \calN_p(\Gamma^\ast)$.\\
(3) $\Pi$ Hodge decomposes $L^p(\R^n;\C^N)$ in the sense that
$$
L^p(\R^n;\C^N) = \calN_{p}(\Pi)\oplus \overline{\calR_{p}(\Gamma)} \oplus \overline{\calR_{p}(\Gamma^\ast)}\ ,
$$ or equivalently,
$
L^p = \calN_{p}(\Gamma)\oplus  \overline{\calR_{p}(\Gamma^\ast)} 
$ and $
L^p = \calN_{p}(\Gamma^\ast)\oplus \overline{\calR_{p}(\Gamma)}
$.\\
(4) $\calN_{p}(\Gamma)$, $\calN_{p}(\Gamma^\ast)$, $\overline{\calR_{p}(\Gamma)}$ and $\overline{\calR_{p}(\Gamma^\ast)}$ each form complex interpolation scales, $p \in (1,\infty)$.\\
(5) Hodge-Dirac operators with constant coefficients have a bounded $H^{\infty}$ functional calculus in $L^p(\R^n;\C^N)$.\\  
(6) There holds $\|\nabla\otimes u\|_{p} \lesssim \|\Pi u\|_{p}$ for all $u \in \calD_{p}(\Pi)\cap \overline{\calR_{p}(\Pi)}$.\\
(7) There exists a bounded potential map $S_\Gamma: \overline{\calR_{p}(\Gamma)} \to \dot W^{1,p}(\R^n;\C^N)$ such that $\Gamma S_\Gamma =I$ on $\overline{\calR_{p}(\Gamma)}$; and there exists a bounded potential map $S_{\Gamma^*}: \overline{\calR_{p}(\Gamma^*)} \to \dot W^{1,p}(\R^n;\C^N)$ such that $\Gamma^* S_{\Gamma^*} =I$ on $\overline{\calR_{p}(\Gamma^*)}$.

\end{Prop}

\begin{proof}
See \cite{HMcP2}, Lemma 5.3, Proposition 5.4. For (4), see \cite{HMc}. Part (5) is proven in \cite[Theorem 3.6]{HMcP2}. Part (6) is a consequence of \eqref{Pi1}, as shown in \cite[Proposition 5.2]{HMcP2}.

To prove part (7) for $\Gamma$, first note that, for all $u\in \overline{\calR_{p}(\Gamma)}$,  $u = \lim_{k\to\infty} \Gamma w_k$ where $w_k = k^2\Gamma^*(I+k^2\Pi^2) u$, so that  $\|\nabla\otimes w_k\|_{p} \lesssim \|u\|_{p}$  by (6), 

and further  $(w_k)_{k \in \N}$ is a Cauchy sequence in $\dot W^{1,p}(\R^n;\C^N)$. Define $S_\Gamma u = \lim_{k \to \infty} w_k$ in $\dot W^{1,p}(\R^n;\C^N)$, and we obtain $\Gamma S_\Gamma u = u$,  since $\Gamma$ is a bounded operator from $\dot W^{1,p}(\R^n;\C^N)$ to $L^p$. The same proof applies to $\Gamma^{*}$.  
\end{proof}

We now consider \emph{perturbed Hodge-Dirac operators}.

\begin{Def}\label{pHD} A {\em perturbed Hodge-Dirac operator} is an operator of the form
$$\Pi_{B}:= \Gamma + \Gamma_B^\ast := \Gamma + B_{1}\Gamma^\ast B_{2},$$
where  $\Pi = \Gamma + \Gamma^\ast$ is a Hodge-Dirac operator with constant coefficients, and
$B_{1},B_{2} $ are multiplication operators by $L^{\infty}(\R^{n};\mathscr{L}(\C^{N}))$ functions which satisfy
\begin{align*}
\Gamma^\ast B_{2}B_{1}\Gamma^\ast   &= 0 \quad \text{in the sense that} \quad \calR_2(B_{2}B_{1}\Gamma^\ast) \subset \calN_2(\Gamma^{\ast}); \\
\Gamma B_{1}B_{2}\Gamma   &= 0 \quad \text{in the sense that} \quad \calR_2(B_{1}B_{2}\Gamma) \subset \calN_2(\Gamma);
\end{align*}
\begin{align*}
	\Re(B_1\Gamma^\ast u,\Gamma^\ast u) &\geq \kappa_1 \|\Gamma^\ast u\|_2^2 , \qquad \forall u \in \calD_2(\Gamma^\ast)\qquad\text{and} \\
	\Re(B_2 \Gamma u,\Gamma u) &\geq \kappa_2 \|\Gamma u\|_2^2 , \qquad \forall u \in \calD_2(\Gamma)
\end{align*}
for some $\kappa_1,\kappa_2>0$. Let the angles of accretivity be
\begin{align*}
	\omega_1 :=\sup_{u \in \calR(\Gamma^\ast) \setminus \{0\}} |\arg(B_1u,u)| <\frac{\pi}{2},\\
	\omega_2 :=\sup_{u \in \calR(\Gamma) \setminus \{0\}} |\arg(B_2u,u)| <\frac{\pi}{2},
\end{align*}
and set $\omega:=\frac{1}{2}(\omega_1+\omega_2)$.
\end{Def}

Such operators satisfy the invertibility properties (denoting $\frac{1}{p'}=1-\frac{1}{p}$)
\begin{equation*}\tag{$\Pi_{B}(p)$}\label{PiB3}
\|u\|_{p} \leq C_p \|B_{1} u\|_{p} \quad \forall u \in \overline{\calR_{p}(\Gamma^\ast)}
\quad \text{and} \quad
\|v\|_{p'} \leq C_{p'} \|{B_{2}}^\ast v\|_{p'} \quad \forall v \in \overline{\calR_{p'}(\Gamma)}
\end{equation*}
when $p=2$. \\

In many cases they satisfy (\ref{PiB3}) for all $p\in(1,\infty)$, for example if $B_1$ and $B_2$ are invertible in $L^\infty$, though in general all we can say is that the set of $p$ for which (\ref{PiB3}) holds is open in $(1,\infty)$.  This follows on applying the extrapolation result of Kalton and Mitrea (\cite{KM}, Theorem 2.5) to the interpolation families $B_1: \overline{\calR_{p}(\Gamma^\ast)} \to L^p(\R^n)$ and ${B_2}^\ast: \overline{\calR_{p'}(\Gamma)}\to L^{p'}(\R^n)$. \\

As noted in \cite{HMcP2}, it is a consequence of (\ref{PiB3}) that $\Gamma^*_B$ is a closed operator in $L^p$ with adjoint $(\Gamma^*_B)^\ast={B_2}^*\Gamma{B_1}^*$ acting in $L^{p'}$, that $\overline{\calR_{p}(\Gamma^*_B)} = B_1\overline{\calR_{p}(\Gamma^*)}$, and that $\overline{\calR_{p'}({B_2}^*\Gamma {B_1}^*)} = {B_2}^*\overline{\calR_{p'}(\Gamma)}$. Moreover, if (\ref{PiB3}) 
holds for all $p$ in a subinterval of $(1,\infty)$, then the spaces $\overline{\calR_{p}(\Gamma^*_B)}$ interpolate for those $p$ also.     \\

\begin{Def} \label{def:Hodge-dec}
 A perturbed Hodge-Dirac operator $\Pi_B$ {\em Hodge decomposes} $L^p(\R^n;\C^N)$ for some $p \in (1,\infty)$, if {\em(\ref{PiB3})} holds and there is a splitting into complemented subspaces 
\begin{equation*}
	L^p(\R^n;\C^N) = \calN_{p}(\Pi_{B})\oplus \overline{\calR_{p}(\Pi_{B})} =  \calN_{p}(\Pi_{B})\oplus \overline{\calR_{p}(\Gamma)} \oplus \overline{\calR_{p}(\Gamma^\ast_B)}.
\end{equation*} 
\end{Def}
It is proved in \cite[Proposition 2.2]{AKM} that $\Pi_B$ Hodge decomposes $L^2(\R^n;\C^N)$.\\

In investigating the property of Hodge Decomposition,  let $\Pb_q$ denote the bounded projection of $L^q(\R^n;\C^N)$ onto $\overline{\calR_{q}(\Gamma^\ast)}$ with nullspace $\calN_{q}(\Gamma)$, and let $\Qb_q$ denote the bounded projection of $L^q(\R^n;\C^N)$ onto $\overline{\calR_q(\Gamma)}$  with nullspace $\calN_q(\Gamma^\ast)$ ($1<q<\infty$).  When $\Pi_{B}$ Hodge decomposes $L^p(\R^n;\C^N)$, we denote by $\mathbb{P}_{\overline{\calR_{p}(\Pi_{B})}}$ the projection of $L^p(\R^{n};\C^{N})$ onto $\overline{\calR_{p}(\Pi_{B})}$  with nullspace $\calN_{p}(\Pi_{B})$.

\begin{Prop}\label{equiv}  Let $\Pi_B$ be a perturbed Hodge-Dirac operator, and let $p\in (1,\infty)$. Then {\em (i)} $\Pi_B$ Hodge decomposes $L^p(\R^n;\C^N)$ if and only if both {\em (A)} and {\em (B)} hold, where
\begin{equation*}\tag{A} \|u\|_{p} \lesssim \|B_{1} u\|_{p} \quad \forall u \in \overline{\calR_{p}(\Gamma^\ast)}\quad\text{and}\quad L^p(\R^n;\C^N) = \calN_{p}(\Gamma)\oplus B_1\overline{\calR_{p}(\Gamma^\ast)}\ ;
\end{equation*}
\begin{equation*}\tag{B}
\|v\|_{p'} \lesssim \|{B_{2}}^\ast v\|_{p'} \quad \forall v \in \overline{\calR_{p'}(\Gamma)}
\quad\text{and}\quad L^{p'}(\R^n;\C^N) = \calN_{p'}(\Gamma^\ast)\oplus {B_2}^\ast\,\overline{\calR_{p'}(\Gamma)}\ .
\end{equation*}
{\em (ii)} Moreover  {\em(A)} is equivalent to {\em(A')}, and  {\em(B)} is equivalent to {\em(B')} where 
\begin{equation*}\tag{A'} \Pb_p B_1: \overline{\calR_{p}(\Gamma^\ast)}\to \overline{\calR_{p}(\Gamma^\ast)} \quad\text{is an isomophism}\ ;
\end{equation*}  
\begin{equation*}\tag{B'} \Qb_{p'} {B_2}^\ast: \overline{\calR_{p'}(\Gamma)}\to \overline{\calR_{p'}(\Gamma)} \quad\text{is an isomophism}\ .
\end{equation*}
\end{Prop}  

\begin{proof} (i)     Under the invertibility assumption (\ref{PiB3}), $\Pi_B$ Hodge decomposes $L^p(\R^n;\C^N)$ if and only if both $L^p(\R^n;\C^N) = \calN_{p}(\Gamma)\oplus \overline{\calR_{p}(\Gamma^\ast_B)}$ and $L^p(\R^n;\C^N) = \calN_{p}(\Gamma^\ast_B)\oplus \overline{\calR_{p}(\Gamma)}$ hold, i.e. if and only if 
$L^p(\R^n;\C^N) = \calN_{p}(\Gamma)\oplus B_1\overline{\calR_{p}(\Gamma^\ast)}$ and $L^{p'}(\R^n;\C^N) = \calN_{p'}(\Gamma^\ast)\oplus {B_2}^\ast\,\overline{\calR_{p'}(\Gamma)}$ \cite[Lemmas 6.1, 6.2]{HMcP2}. This gives the proof of (i). \\

(ii) (A) implies (A'): Let $u\in \overline{\calR_{p}(\Gamma^\ast)}$.  Then $\Pb_{p} B_1u = -(I-\Pb_{p})B_1u+B_1u$ with, by (A), $\|\Pb_{p} B_1u\|_p\approx \|(I-\Pb_p)B_1u\|_p+\|B_1u\|_p$, so that $\|u\|_p\lesssim\|B_1u\|_p\lesssim\|\Pb_{p} B_1u\|_p$. It remains to prove surjectivity. Let $v\in \overline{\calR_{p}(\Gamma^\ast)}$. By (A), there exist $w\in \calN_{p}(\Gamma)$ and $u\in \overline{\calR_{p}(\Gamma^\ast)}$ such that $v = w + B_1u$, and hence $v = \Pb_{p} v = \Pb_{p} B_1u$ as claimed.\\

(A') implies (A): First we have that if $u\in  \overline{\calR_{p}(\Gamma^\ast)}$, then $\|u\|_p \lesssim \|\Pb_p B_1 u\|_p\lesssim \|B_1u\|_p$. Next we show that $\calN_{p}(\Gamma)\cap B_1\overline{\calR_{p}(\Gamma^\ast)}=\{0\}$. Indeed if $u\in \calN_{p}(\Gamma)$, and $u=B_1v$ with $v\in \overline{\calR_{p}(\Gamma^\ast)}$, then $\Pb_p B_1 v = \Pb_p u = 0$, so by (A'), $v=0$ and thus $u=0$.  Now we show that every element $u\in L^p(\R^n;\C^N)$ can be decomposed as stated. Let $u\in L^p(\R^n;\C^N)$. Then 
\begin{align*} u&=(I-\Pb_p)u +\Pb_p u\\
&=(I-\Pb_p)u + \Pb_p B_1 v\qquad\text{for some }v\in \overline{\calR_{p}(\Gamma^\ast)}\quad\text{(by (A'))}\\
&=(I-\Pb_p)(u-B_1v) + B_1v\\
&\in \calN_{p}(\Gamma)\quad+\quad B_1\overline{\calR_{p}(\Gamma^\ast)}
\end{align*} with $\|B_1v\|_p\lesssim\|v\|_p\lesssim\|\Pb_p u\|_p\lesssim\|u\|_p$. This gives the claimed direct sum decomposition.\\

The proof that (B) is equivalent to (B') follows the same lines, with $p, \Gamma, B_1$ replaced by $p', \Gamma^\ast, {B_2}^\ast$.
\end{proof}

\begin{Prop} \label{prop:Hodge}
The set of all $p$ for which $\Pi_B$ Hodge decomposes $L^p(\R^n;\C^N)$, is an open interval $(p_H,p^H)$, where $1\leq p_H<2<p^H\leq\infty$. 
\end{Prop} 

\begin{proof} 
By the interpolation properties of $\overline{\calR_{p}(\Gamma^\ast)}$, the set of $p$ for which (A') holds, is an open interval which contains 2, and the same can be said about (B'). So the set of all $p$ for which $\Pi_B$ Hodge decomposes $L^p(\R^n;\C^N)$ is the intersection of these two intervals, and thus is itself an open interval which we denote
 by $(p_H,p^H)$, with $1\leq p_H<2<p^H\leq\infty$. 
\end{proof}

An investigation of $\Pi_B$ involves  the related operator $\underline{\Pi_B}=\Gamma^\ast + B_2 \Gamma B_1$, which is also a perturbed Hodge-Dirac operator with $(\Gamma,\Gamma^\ast,B_1,B_2)$ replaced by $(\Gamma^\ast,\Gamma,B_2,B_1)$, and for it we  need the invertibility properties
\begin{equation*}\tag{$\underline{\Pi_{B}}(p)$}\label{PiB4}
\|u\|_{p} \leq C_p \|B_{2} u\|_{p} \quad \forall u \in \overline{\calR_{p}(\Gamma)} 
\quad \text{and} \quad 
\|v\|_{p'} \leq C_{p'} \|{B_{1}}^\ast v\|_{p'} \quad \forall v \in \overline{\calR_{p'}(\Gamma^\ast)}\ .
\end{equation*}

The formulae connecting $\Pi_B$ and $\underline{\Pi_B}$ are, for $\theta \in (\omega,\frac{\pi}{2})$, $f \in H^\infty(S_\theta^o)$ and $u \in \calD_2(\Gamma^\ast)$,
\begin{align} \nonumber
	f(\underline{\Pi_B})(\Gamma^\ast u) &= B_2 f(\Pi_B)(B_1\Gamma^\ast u), & \text{when}\ f \ \text{is odd},\\
	B_1g(\underline{\Pi_B})(\Gamma^\ast u) &= g(\Pi_B)(B_1 \Gamma^\ast u), & \text{when}\ g \ \text{is even}.
	\label{eq:swap-bar}
\end{align}

\begin{Prop}
\label{prop:properties}
Suppose $\Pi_B$ is a perturbed Hodge-Dirac operator which Hodge decomposes $L^p(\R^n;\C^N)$ for all $p \in (p_H,p^H)$. Then:
\begin{enumerate}
\item  ${\Pi_B}^\ast = \Gamma^\ast + {B_2}^\ast\Gamma{B_1}^\ast$ is a perturbed Hodge-Dirac operator which Hodge decomposes $L^q(\R^n;\C^N)$ for  all  $q \in ((p^H)',(p_H)')$, i.e. {\em($\Pi_B(q)$)} holds and
$$L^{q}(\R^n;\C^N) = \calN_{q}({\Pi_{B}}^\ast)\oplus \overline{\calR_{q}(\Gamma^\ast)} \oplus \overline{\calR_{q}({B_2}^\ast \Gamma B_1^\ast)}\ .$$
\item The perturbed Hodge-Dirac operator  $\underline{\Pi_B}$ Hodge decomposes $L^p(\R^n;\C^N)$ for all  $p \in (p_H,p^H)$, i.e. {\em(\ref{PiB4})} holds and
 \begin{equation*}	L^p(\R^n;\C^N) = \calN_{p}(\underline{\Pi_{B}})\oplus \overline{\calR_{p}(\Gamma^\ast)} \oplus \overline{\calR_{p}(B_2\Gamma B_1)}\ .
\end{equation*}
 \item If, for some $p\in(p_H,p^H)$, $\Pi_{B}$ is $\omega$-bisectorial in $L^p(\R^n;\C^N)$, then $\underline{\Pi_{B}}$ is also $\omega$-bisectorial in $L^p(\R^n;\C^N)$.
\end{enumerate}
\end{Prop}

\begin{proof}(1) First note that the invertibility condition (\ref{PiB3}) for $\Pi_B$ is the same as the invertibility condition (${\Pi_B}^\ast(p')$) for ${\Pi_B}^\ast = \Gamma^\ast + {B_2}^\ast\Gamma{B_1}^\ast$. Using this, it is proved in \cite{HMcP2}, Lemma 6.3 that the Hodge decomposition for ${\Pi_B}^\ast$ in $L^{p'}(\R^n;\C^N)$ is equivalent to the Hodge decomposition for ${\Pi_B}$ in $L^{p}(\R^n;\C^N)$.\\
(2) On applying Proposition \ref{equiv}, $\underline{\Pi_B}$ Hodge decomposes $L^p(\R^n;\C^N)$ if and only if 
\begin{equation*}\tag{A''} \Qb_{p} {B_2}: \overline{\calR_{p}(\Gamma)}\to \overline{\calR_{p}(\Gamma)} \quad\text{is an isomophism}\ \text{and}
\end{equation*}
\begin{equation*}\tag{B''} \Pb_{p'} {B_1}^\ast: \overline{\calR_{p'}(\Gamma^\ast)}\to \overline{\calR_{p'}(\Gamma^\ast)} \quad\text{is an isomophism}\ .
\end{equation*}  
Using the Hodge decompositions for the unperturbed operators to identify the dual of $\overline{\calR_{p}(\Gamma^\ast)}$ with $\overline{\calR_{p'}(\Gamma^\ast)}$, we find by duality that (A') is equivalent to (B'') and (B') is equivalent to (A'').  This proves (2).  \\
(3) This is essentially proved in \cite{HMcP2}, Lemma 6.4.
\end{proof}
 
\begin{Remark} We are not saying that {\em(\ref{PiB3})} is equivalent to {\em(\ref{PiB4})} for general $p$.
\end{Remark}

We now define the operators
\begin{align*}
	R_t^B &:=(I+it\Pi_B)^{-1}, \quad t \in \R,\\
	P_t^B &:=(I+t^2{\Pi_B}^2)^{-1} = \frac12 (R_t^B + R_{-t}^B) = R_t^BR_{-t}^B, \quad t>0,\\
	Q_t^B &:= t\Pi_B(I+t^2{\Pi_B}^2)^{-1} = \frac{1}{2i}(-R_t^B + R_{-t}^B), \quad t>0.
\end{align*}
In the unperturbed case $B_1=B_2=I$, we write $R_t$, $P_t$ and $Q_t$ for $R_t^B$, $P_t^B$ and $Q_t^B$, respectively.
If we replace $\Pi_{B}$ by $\underline{\Pi_{B}}$, we replace $R_t^B$, $P_t^B$ and $Q_t^B$ by $\underline{R_t^B}$, $\underline{P_t^B}$ and $\underline{Q_t^B}$, respectively.
\\

We state some basic results for the unperturbed operator $\Pi$, noting that when we apply  \cite{HMcP2},
we do not make use of the probabilistic/dyadic methods developed there.

\begin{Prop}
\label{lem:hardyPi}

Let $M>\frac{n}{2}+1$.

\begin{enumerate}

\item
 For all $p \in (1,2]$, the family $\{s^{n(\frac{1}{p}-\frac{1}{2})}(R_s)^{M}\mathbb{P}_{\overline{\calR(\Pi)}} \;;\; s \in \R^{*}\}$ 
is uniformly bounded in $\calL(L^{p},L^{2})$, and the family
$\{(Q_s)^{M} \;;\; s \in \R^{*}\}$ 
has $L^p$-$L^2$ off-diagonal bounds of every order. 

\item 
For all $p \in (1,\infty)$, the family $\{P_s \;;\; s \in \R^{*}\}$ has $L^p$-$L^p$ off-diagonal bounds of 
every order.

\item
For all $p \in (1,\infty)$,
$$\|(s,x)\mapsto {Q_s}^{M} u(x)\|_{T^{p,2}} 
\approx \|u\|_{p} \quad \forall u \in \overline{\calR_p(\Pi)}.$$
\item
For all $p \in (\max\{2_{*},1\},2]$,
$$ \|(s,x)\mapsto Q_s u(x)\|_{T^{p,2}} \approx \|u\|_{p} \quad \forall u \in \overline{\calR_p(\Pi)}.$$
\end{enumerate}
\end{Prop}
\begin{proof}
(1)  
Let $s>0$. By \cite[Proposition 4.8]{HMcP2}, $(R_{s})^{M}$ is a Fourier multiplier with bounded symbol $\xi \mapsto m(s\xi)$.
We also have that $\Pi^{M-1}{R_{1}}^{M}$  is a Fourier multiplier with bounded symbol 
$\tilde{m}:\xi \mapsto \widehat{\Pi}(\xi)^{M-1}m(\xi)$. 
Since  $|\xi|^{M-1}|m(\xi)w| \lesssim |\tilde{m}(\xi)w|$ for every $\xi \in \R^{n}$ and every $w \in \calR(\widehat{\Pi}(\xi))$  (by \eqref{Pi1}),
we have that 
$$
\underset{\xi \in \R^{n}\backslash\{0\}}{\sup} |\xi|^{M-1}|m(\xi)w| \lesssim \|\tilde{m}\|_{\infty}.
$$
For $u \in \overline{\calR(\Pi)} \cap L^1$, this implies 
$$
\|s^{\frac{n}{2}}{R_s}^{M} u\|_{2} \lesssim \|s^{\frac{n}{2}}m(s\,.)\,\widehat{u}\|_{2} \lesssim \|\widehat{u}\|_{\infty} \lesssim \|u\|_{1}.
$$
Since $\mathbb{P}_{\overline{\calR(\Pi)}}$ is a Fourier multiplier of weak type 1-1 by \cite[Proposition 4.4]{HMcP2}, we have by interpolation that, for all $p \in (1,2]$, the family $\{s^{n(\frac{1}{p}-\frac{1}{2})}(R_s)^{M}\mathbb{P}_{\overline{\calR(\Pi)}} \;;\; s \in \R^{*}\}$ is uniformly bounded in $\calL(L^{p},L^{2})$.
This implies that $\{s^{n(\frac{1}{p}-\frac{1}{2})}(Q_s)^{M} \;;\; s \in \R^{*}\}$ is uniformly bounded in $\calL(L^{p},L^{2})$.
Using Lemma \ref{lem:ODcomp} to interpolate this uniform bound with the $L^2$-$L^2$ off-diagonal bounds for 
$\{(Q_s)^{M} \;;\; s \in \R^{*}\}$ gives the second part of (1).\\
(2) By Proposition \ref{prop:pi} (5), $\Pi$ is bisectorial in $L^p(\R^n;\C^N)$, $p \in (1,\infty)$. Then the proof of \cite[Proposition 5.2]{AKM}, showing off-diagonal estimates in $L^2$, carries over to $p \in (1,\infty)$.  \\
(3)  Let $p\in (1,2]$. We first notice that 
 $\{(R_t)^{\frac{M}{2}-2} \Gamma \;;\; t \in \R^{*}\}$ and $\{(R_t)^{\frac{M}{2}-2} \Gamma^{*} \;;\; t \in \R^{*}\}$ have $\dot{W}^{1,p}$-$L^2$ off-diagonal bounds of every order on balls.  This follows from (1) and Lemma \ref{lem:ODcomp}.
Using Theorem \ref{thm:conicalVSvertical}, Proposition \ref{prop:pi}, and Theorem \ref{thm:hardy}, we thus get 
$$\|(s,x)\mapsto {Q_s}^{M} u(x)\|_{T^{p,2}} 
\sim \|(s,x)\mapsto {Q_s}^{10n+1} u(x)\|_{T^{p,2}} \lesssim
\|(\int \limits _{0} ^{\infty} |Q_s u|^{2} \frac{ds}{s})^{\frac{1}{2}}\|_{L^p}
\lesssim \|u\|_{p},$$
for all  $p \in (1,2]$  and all $u \in \overline{\calR_p(\Pi)}$.  
Note that our use of Theorem \ref{thm:conicalVSvertical} for the unperturbed operator $\Pi$ only relies on the relevant $\dot{W}^{1,p}$-$L^2$ off-diagonal bounds. In particular, it does not rely on later parts of the paper that could be based on the proposition being proven.
The reverse estimate follows from  Proposition \ref{lem:reverse}, and the equivalence for all $p\in (1,\infty)$ follows by duality.
(4) For  $p \in (\max(1,2_{*}),\infty)$ and $u \in L^p \cap L^2$,
$$
\|(t,x)\mapsto Q_{t}u(x)\|_{T^{p,2}} \lesssim \|(t,x) \mapsto \int \limits _{0} ^{\infty}Q_{t}Q_{s}{Q_{s}}^{M}u(x)\,\frac{ds}{s}\|_{T^{p,2}}.
$$
Noting that $Q_{t}Q_{s} =  \begin{cases} \frac{s}{t}(I-P_{t})P_{s} \;\; \text{if} \; 0< s \leq t,\\
\frac{t}{s}(I-P_{s})P_{t} \;\; \text{if} \; 0<t \leq s,
\end{cases}
$
we consider the integral operator defined by
$$
T_{K}F(t,x) = \int \limits _{0} ^{\infty} \min(\frac{t}{s},\frac{s}{t})K(t,s)F(s,x) \frac{ds}{s} \qquad \forall t > 0 \quad \forall x \in \R^{n}, 
$$
for $F\in T^{2,2}$ and $K(t,s) = \begin{cases} (I-P_{t})P_{s} \;\; \text{if} \; 0< s \leq t,\\
(I-P_{s})P_{t} \;\; \text{if} \; 0< t \leq s.\end{cases}$
Since, for every $\varepsilon>0$, the integral operator defined by 
$
\tilde{T}_{K}F(t,x) = \int \limits _{0} ^{\infty} \min(\frac{t}{s},\frac{s}{t})^{\varepsilon}K(t,s)F(s,x) \frac{ds}{s}$, for $F \in T^{2,2}$ and all $t > 0$, $x \in \R^{n}$, 
is bounded on $T^{2,2}$ by Schur's lemma,
the result follows by Corollary \ref{Cor:SIO-tent} and (3). Note that Corollary \ref{Cor:SIO-tent}, and Section \ref{sec:tech1} in general, does not rely on the rest of the paper.

\end{proof}

We conclude the section by recalling the main result of  Axelsson, Keith and the second author in \cite{AKM}. Note that   perturbed Hodge-Dirac operators satisfy the assumptions of \cite{AKM} and \cite{HMcP2}. In particular,  $\|\nabla\otimes u\|_{2} \lesssim \|\Pi u\|_{2}$ for all $u \in \calD_{2}(\Pi)\cap \overline{\calR_{2}(\Pi)}$ as stated in Proposition \ref{prop:pi} (6).

\begin{Theorem}
\label{thm:p=2}
Suppose $\Pi_B$ is a perturbed Hodge-Dirac operator with angles of accretivity as specified in Definition \ref{pHD}.
Then:
\begin{enumerate}
\item $\Pi_B$ is an $\omega$-bisectorial operator in $L^{2}(\R^{n};\C^{N})$.
\item The family $\{R_{t}^{B} \;;\; t \in \R\}$ has $L^2$-$L^2$ off-diagonal bounds of every order.
\item $\Pi_{B}$ satisfies the quadratic estimate
\[
	\|(t,x)\mapsto Q_t^Bu(x)\|_{T^{2,2}} \approx \|u\|_{L^2}
\]
for all $u \in \overline{\calR_2(\Pi_B)} \subseteq L^2(\R^n;\C^N)$. 
\item For all $\mu>\omega$, $\Pi_{B}$ has a bounded $H^{\infty}$ functional calculus with angle $\mu$ in $L^2(\R^n;\C^N)$.
\end{enumerate}
\end{Theorem}

\section{Main results}
\label{sec:res}

Our main results are conical square function estimates on the range of $\Gamma$. Combining these estimates, and using the structure of  Hodge-Dirac operators, we obtain functional calculus results as corollaries.

\begin{Theorem} \label{thm:main}
Suppose $\Pi_B$ is a perturbed Hodge-Dirac operator. \\
Given $p \in (\max\{1,(p_H)_\ast\},\infty)$, we have
\begin{align*}
\|(t,x) \mapsto (Q_{t}^{B})^{M}u(x)\|_{T^{p,2}} &\leq C_p \|u\|_{p} \quad \forall u \in \overline{\calR_p(\Gamma)}\quad\text{and}\\
\|(t,x) \mapsto (\underline{Q_{t}^{B}})^{M}u(x)\|_{T^{p,2}} &\leq C_p \|u\|_{p} \quad \forall u \in \overline{\calR_p(\Gamma^\ast)} \, ,
\end{align*}
where $M \in \N$ if $p \geq 2$, and $M \in \N$ with $M>\frac{n}{2}$ if $p<2$. 
\end{Theorem}

The proof of Theorem \ref{thm:main} consists of two parts: a low-frequency estimate and a high-frequency estimate, stated below in Theorem \ref{thm:low} and Theorem \ref{thm:high}, respectively, and proven in the subsequent sections. We show that Theorem \ref{thm:main} is a consequence of Theorems \ref{thm:low} and \ref{thm:high} after their statements.\\

In the above theorem, when we consider a function of the form $(t,x) \mapsto (Q_{t}^{B})^{M}u(x)$ in a tent space $T^{p,2}$, we are considering $Q_{t}^{B}$ as a bounded operator on $L^2$ for each $t$. When $p \in (p_H,p^H)$, $Q_t^B$ does in fact extend to a bounded operator on $L^p$.\\

As a first consequence, we obtain equivalence of the Hardy space $H^p_{\Pi_B}(\R^n;\C^N)$ with the $L^p$ closure of $\calR_p(\Pi_B)$ whenever $p \in (p_H,p^H)$, and corresponding results restricted to the ranges of $\Gamma$ and $\Gamma_B^\ast$ for $p$ below $p_H$. We recall that \eqref{PiB3} always holds for $p \in (p_H,p^H)$.

\begin{Cor} \label{cor:hardy}
Suppose $\Pi_B$ is a perturbed Hodge-Dirac operator. Suppose that $\mu \in (\omega,\frac{\pi}{2})$, and that $\psi \in \Psi_\alpha^\beta(S_\mu^o)$ is non-degenerate with 
$$
	\text{either} \; p \in (1,2], \ \text{and} \ \alpha>0, \beta>\tfrac{n}{2}; \quad \text{or} \; p \in [2,\infty), \ \text{and} \ \alpha>\tfrac{n}{2}, \beta>0. 
$$
\begin{enumerate}
\item Let $p \in (\max\{1,(p_H)_\ast\},p^H)$. Then 
\begin{align*}
	\|(t,x) \mapsto \psi(t\Pi_B)u(x)\|_{T^{p,2}} & \approx \|u\|_{p}  \quad \forall u \in \overline{\calR_p(\Gamma)}\quad\text{and}\\
	\|(t,x) \mapsto \psi(t\underline{\Pi_B})u(x)\|_{T^{p,2}} & \approx \|u\|_{p}  \quad \forall u \in \overline{\calR_p(\Gamma^\ast)}.
\end{align*}
In particular, $\overline{\calR_p(\Gamma)} \subseteq H^p_{\Pi_B}(\R^n;\C^N)$ and $\overline{\calR_p(\Gamma^\ast)} \subseteq H^p_{\underline{\Pi_B}}(\R^n;\C^N)$.
\item Let $p \in (\max\{1,(p_H)_\ast\},p^H)$ and suppose that \eqref{PiB3} holds. Then 
\begin{align*}
	\|(t,x) \mapsto \psi(t\Pi_B)u(x)\|_{T^{p,2}} & \approx \|u\|_{p}  \quad \forall u \in \overline{\calR_p(\Gamma_B^\ast)}.
\end{align*}
In particular, $\overline{\calR_p(\Gamma_B^\ast)} \subseteq H^p_{\Pi_B}(\R^n;\C^N)$.
\item Let $p \in (p_H,p^H)$. Then 
\begin{align*}
	\|(t,x) \mapsto \psi(t\Pi_B)u(x)\|_{T^{p,2}} & \approx \|u\|_{p}  \quad \forall u \in \overline{\calR_p(\Pi_B)}.
\end{align*}
In particular, $H^p_{\Pi_B}(\R^n;\C^N) =\overline{\calR_p(\Pi_B)}$.
\end{enumerate}
\end{Cor}

\begin{Remark}
In Corollary \ref{cor:hardy} (2), one has in fact
$H^p_{\Pi_B}(\R^n;\C^N) = \overline{\calR_p(\Gamma)} \oplus \overline{\calR_p(\Gamma_B^\ast)}$, for $p \in (\max\{1,(p_H)_\ast\},p^H)$.
This follows from the fact that the Hodge projections preserve Hardy spaces, as can be seen by considering their actions on $H^{1}_{\Pi_{B}}$ molecules (as defined in \cite{AMcMo}).

\end{Remark}

\begin{Remark}
An inspection of our proof shows that we are actually proving that
$$
\|u\|_{H^{p}_{\Pi_{B}}} \approx \|u\|_{H^{p}_{\Pi}} \quad \forall u \in \overline{R_{2}(\Gamma)} \cap H^{p}_{\Pi}\ .
$$
When $p>1$, we then use that $\|u\|_{H^{p}_{\Pi}} \approx \|u\|_{L^{p}}$. The proof still works if $(p_{H})_{*} < 1$ and $p=1$. In this case we get that
$$
\|u\|_{H^{1}_{\Pi_{B}}} \approx \|u\|_{H^{1}_{\Pi}} \quad \forall u \in \overline{R_{2}(\Gamma)} \cap H^{1}_{\Pi}\ .
$$
As $\Pi$ is a Fourier multiplier one can then relate the $H^{1}_{\Pi}$ norm to the classical $H^{1}$ norm:
$$
\|u\|_{H^{1}_{\Pi}} \approx  \|u\|_{H^{1}(\R^n,\C^N)} \quad\forall u\in H^1_{\Pi}\ .
$$
 This can be done, for instance, by using the molecular theory presented in \cite{AMcMo}.
\end{Remark}

As a second consequence, we obtain functional calculus results for $\Pi_B$.

\begin{Cor} \label{cor:fc}
Suppose $\Pi_B$ is a perturbed Hodge-Dirac operator.
Suppose $\mu \in (\omega,\frac{\pi}{2})$.
\begin{enumerate}
\item Let $p \in (\max\{1,(p_H)_\ast\},p^H)$. Then for all $f \in \Psi(S_\mu^o)$,
\begin{align*}
\|f(\Pi_{B})u\|_{p} &\leq C_p \|f\|_{\infty}\|u\|_{p} \quad \forall u \in \overline{\calR_p(\Gamma)}.
\end{align*}
\item Let $p \in (\max\{1,(p_H)_\ast\},p^H)$ and suppose  \eqref{PiB3} holds.  Then for all $f \in \Psi(S_\mu^o)$,
\begin{align*}
\|f(\Pi_{B})u\|_{p} &\leq C_p \|f\|_{\infty}\|u\|_{p} \quad \forall u \in \overline{\calR_p(\Gamma_B^\ast)}.
\end{align*}
\item Let $p \in (p_H,p^H)$. Then $\Pi_B$ is $\omega$-bisectorial, and has a bounded $H^\infty$ functional calculus with angle $\mu$ in $L^p(\R^n;\C^N)$.
\end{enumerate}
\end{Cor}

For the proofs, we use the following result, that establishes the reverse square function estimates when $p<2$.

\begin{Prop}
\label{lem:reverse}
Suppose $\Pi_B$ is a perturbed Hodge-Dirac operator.
For all $p \in [2,\infty]$ and all $M \in \N$, we have
\begin{equation} \label{eq:reverse}
\|(t,x) \mapsto (Q_{t}^{B})^{M}u(x)\|_{T^{p,2}} \leq C_p \|u\|_{p} \quad \forall u \in L^2(\R^n;\C^N) \cap L^p(\R^n;\C^N).
\end{equation}
Consequently, for all $p \in (1,2]$ and all $M \in \N$, we have
$$
 \|u\|_{p}\leq C_p \|(t,x) \mapsto (Q_{t}^{B})^{M}u(x)\|_{T^{p,2}} \quad \forall u \in \overline{\calR_2(\Pi_B)}\cap L^p(\R^n;\C^N).
$$
\end{Prop}

\begin{proof}
The result for $p=2$ holds by Theorem \ref{thm:p=2}.
We show that $u \mapsto (Q_t^B)^{M}u$ maps $L^\infty$ to $T^{\infty,2}$. The claim for $p \in (2,\infty)$ then follows by interpolation. 
The argument goes back to Fefferman and Stein \cite{FeffermanStein}, and was used in a similar context in e.g. \cite[Section 3.2]{AHM}.
Fix a cube $Q$ in $\R^n$ and split $u \in L^{\infty}(\R^n;\C^N)$ into $u=u \Eins_{4Q} + u \Eins_{(4Q)^c}$.  
Recall $S_{j}(Q) = 2^{j+1}Q \backslash 2^{j}Q$ for all $j \geq 2$.
Theorem \ref{thm:p=2} gives
\[
	(\frac{1}{|Q|} \int_0^{l(Q)} \int_Q |(Q_t^B)^{M} \Eins_{4Q} u(x)|^2 \,\frac{dxdt}{t})^{\frac{1}{2}}
		\lesssim |Q|^{-\frac{1}{2}} \|\Eins_{4Q}u\|_{2} \lesssim \|u\|_{\infty}.
\] 
On the other hand, $L^2$-$L^2$ off-diagonal bounds for $(Q_t^B)^{M}$ of order $N'>\frac{n}{2}$ yield
\begin{align*}
	 & (\frac{1}{|Q|} \int_0^{l(Q)} \int_Q |(Q_t^B)^{M} \Eins_{(4Q)^c} u(x)|^2 \,\frac{dxdt}{t})^{\frac{1}{2}}\\
	  & \qquad \lesssim \sum_{j=2}^{\infty} (\frac{1}{|Q|} \int_0^{l(Q)} \int_Q |(Q_t^B)^{M} \Eins_{S_j(Q)} u(x)|^2 \,\frac{dxdt}{t})^{\frac{1}{2}}\\
	 & \qquad \lesssim \sum_{j=2}^{\infty} 2^{-jN'} (\frac{1}{|Q|} \int_0^{l(Q)} (\frac{t}{l(Q)})^{2N'} \|\Eins_{2^{j+1}Q}u\|_{2}^2 \,\frac{dt}{t})^{\frac{1}{2}} \lesssim \|u\|_{\infty}.
\end{align*}
Consider now $p \in (1,2)$. Let $u \in \overline{\calR_2(\Pi_B)}\cap L^p(\R^n;\C^N)$ and $v \in L^{p'}(\R^n;\C^N)\cap L^2(\R^n;\C^N)$. 
We apply the above result to $\Pi_B^\ast$ in $L^{p'}(\R^n;\C^N)$, noting that $\Pi_B^\ast$ Hodge decomposes $L^{q}$ for all $q \in ((p^H)',(p_H)')$ by Proposition \ref{prop:properties}.
By Calder\'on reproducing formula, tent space duality and the argument above, we have that
\begin{align*}
& |\int \limits _{\R^{n}} u(x).v(x)dx| \lesssim \int \limits _{\R^{n}} \int \limits _{0} ^{\infty} |(Q_t^B)^{M}u(x)||((Q_t^B)^\ast)^{M}v(x)| \frac{dt}{t}dx
\\& \qquad \lesssim  \|(t,x)\mapsto (Q_t^B)^{M} u(x)\|_{T^{p,2}} \|(t,x)\mapsto ((Q_t^B)^\ast)^{M} v(x)\|_{T^{p',2}}\\
& \qquad \lesssim  \|(t,x)\mapsto (Q_t^B)^{M} u(x)\|_{T^{p,2}}  \|v\|_{p'}.
\end{align*}
This gives the assertion.
\end{proof}

\begin{Remark}\label{CM}
Note that for the proof of \eqref{eq:reverse}, we only use that 
$((Q_t^B)^M)_{t>0}$ satisfies $L^2$-$L^2$ off-diagonal bounds of order $N'>\frac{n}{2}$, and defines a bounded mapping from $L^2$ to $T^{2,2}$. In particular, we do not use any assumptions on $\Pi_B$ in $L^p$ for $p \neq 2$. The proof gives a way to define a bounded extension from $L^p$ to $T^{p,2}$ of this mapping. In the case $p=\infty$, the above result shows that for every $u \in L^\infty(\R^n;\C^N)$, $|(Q_t^B)^M u(x)|^2\,\frac{dx dt}{t}$ is a Carleson measure. We will make use of this fact in Proposition \ref{Prop:Carleson-part} below. 
\end{Remark}

We next show that Corollaries \ref{cor:hardy} and \ref{cor:fc} follow from Theorem \ref{thm:main} and Proposition \ref{lem:reverse}.

\begin{proof}[Proof of Corollary \ref{cor:hardy}]
By Theorem \ref{thm:hardy}, it suffices to show
\begin{align*}
	\|(t,x) \mapsto (Q_t^B)^Mu(x)\|_{T^{p,2}} & \approx \|u\|_{p}  ,
\end{align*}
and the corresponding equivalence for $(\underline{Q_t^B})^M$ in case (1), for $M=10n$ and $u$ as given in (1), (2) or (3). First suppose $p \in (\max\{1,(p_H)_\ast\},2]$. Combining Theorem \ref{thm:main} and Proposition \ref{lem:reverse} gives the equivalence for $(Q_t^B)^M$ and all $u \in \overline{\calR_{2}(\Gamma)}\cap L^p(\R^n;\C^N)$. The same reasoning applies to $(\underline{Q_t^B})^M$ and $u  \in  \overline{\calR_{2}(\Gamma^{*})}\cap L^p(\R^n;\C^N)$. 
Now note that $\calN_p(\Gamma^\ast) \cap L^2(\R^n;\C^N) \subseteq \calN_2(\Gamma^\ast)$, and the same holds with $p$ and $2$ interchanged. 
Using the Hodge decomposition for the unperturbed operator $\Pi$, we therefore have 
$$
	L^p(\R^n;\C^N) \cap L^2(\R^n;\C^N)
	= [ \calN_p(\Gamma^\ast) \cap \calN_2(\Gamma^\ast)]
	\oplus [\overline{\calR_p(\Gamma)} \cap \overline{\calR_2(\Gamma)}].
$$
Since this space is dense in $L^p(\R^n;\C^N)$, the space $\overline{\calR_2(\Gamma)} \cap L^p(\R^n;\C^N)$ is dense in $\overline{\calR_p(\Gamma)}$. 
This gives the above equivalence on $\overline{\calR_p(\Gamma)}$, and, similarly, for $(\underline{Q_t^B})^M$ on $\overline{\calR_p(\Gamma^\ast)}$.
As stated before Definition \ref{def:Hodge-dec}, we have $\overline{\calR_p(\Gamma_B^\ast)}
=B_1\overline{\calR_p(\Gamma^\ast)}$ under \eqref{PiB3}.
Using that $M$ is even, the identity \eqref{eq:swap-bar}, $\|B_1\|_{\infty}<\infty$, and that \eqref{PiB3} holds by assumption, we therefore deduce from the above that, for $u= B_{1}\Gamma^{*} B_{2}v \in  \overline{\calR_{p}(\Gamma^{*}_{B})}$,
\begin{align*}
	&\|(t,x) \mapsto (Q_t^B)^M u(x)\|_{T^{p,2}} 	
	=\|(t,x) \mapsto (Q_t^B)^M B_1\Gamma^{*} B_2 v(x)\|_{T^{p,2}} \\
	& \qquad =\|(t,x) \mapsto B_1(\underline{Q_t^B})^M \Gamma^{*} B_2 v(x)\|_{T^{p,2}} 
	\lesssim \|(t,x) \mapsto (\underline{Q_t^B})^M \Gamma^{*} B_2 v(x)\|_{T^{p,2}} \\
	& \qquad \lesssim \|\Gamma^{*} B_2 v\|_{p} 
	\lesssim \|u\|_{p}.
\end{align*}
This gives (2).
In the case $p \in (p_H,2]$, $\Pi_B$ Hodge decomposes $L^p$. This yields the result on $\overline{\calR_p(\Pi_B)}$. 
The case $p \in [2,p^H)$ follows by duality, cf. the proof of Corollary \ref{cor:fc}.
\end{proof}

\begin{proof}[Proof of Corollary \ref{cor:fc}]
First suppose $p \in (\max\{1,(p_H)_\ast\},2]$. Let $M=10n$, and $\mu \in (\omega,\frac{\pi}{2})$. Let $f \in \Psi(S_\mu^o)$ and 
$u \in \overline{\calR_{p}(\Gamma)}$. Using Corollary \ref{cor:hardy} and Theorem \ref{thm:hardy}, we have that
\begin{align*}
\|f(\Pi_{B})u\|_{p} 
&\lesssim \|(t,x) \mapsto (Q_{t}^{B})^{M}f(\Pi_{B})u(x)\|_{T^{p,2}} \\
& \lesssim \|f\|_{\infty} \|(t,x) \mapsto (Q_{t}^{B})^{M}u(x)\|_{T^{p,2}}
\lesssim \|f\|_{\infty} \|u\|_{p}.
\end{align*}
The same reasoning applies to $u  \in  \overline{\calR_{p}(\Gamma^{*}_B)}$, assuming \eqref{PiB3}.
 Now let $p \in (p_H,2]$. Since $\Pi_{B}$ Hodge decomposes $L^p$, we have that, for all $f \in \Psi(S_\mu^o)$,
 $$
 \|f(\Pi_{B})u\|_{p} \lesssim \|f\|_{\infty} \|u\|_{p} \quad \forall u \in L^p(\R^n;\C^N).
 $$
This implies that $\Pi_{B}$ is $\omega$-bisectorial and has a bounded $H^{\infty}$ functional calculus in $L^p$.
 Finally, we consider the case $p \in [2,p^H)$. We apply the above result to $\Pi_B^\ast$, which Hodge decomposes $L^q$ for all $q \in ((p^H)',(p_H)')$ by Proposition \ref{prop:properties}. Hence, $\Pi_B^\ast$ has a bounded $H^\infty$ functional calculus in $L^q$ for all $q \in ((p^H)',2]$. By duality, $\Pi_B$ has a bounded $H^\infty$ functional calculus in $L^p$ for all $p \in [2,p^H)$.
\end{proof}

The first conical square function estimate is an $L^p$ version of the low frequency estimate in the main result of \cite{AKM}, Theorem 2.7 (and hence captures the harmonic analytic part of the proof of the Kato square root problem). The separation into low and high frequency is done via the operators ${P_{t}}^{\tilde{N}}$ and $I-{P_{t}}^{\tilde{N}}$ where $\tilde{N}$ is a large natural number. Throughout the paper we fix $\tilde{N}=10n$.

\begin{Theorem}
\label{thm:low}
Suppose $\Pi_B$ is a perturbed Hodge-Dirac operator. Suppose $M \in \N$ and $p \in (1,\infty)$. Then
$$
\|(t,x) \mapsto (Q_{t}^{B})^M {P_{t}}^{\tilde{N}}u(x)\|_{T^{p,2}} \leq C_p \|u\|_{p} \quad \forall u \in \overline{\calR_p(\Pi)}.
$$
\end{Theorem}

This result is proven in Section \ref{sec:low}.\\

The second conical square function estimate is an $L^p$ version of the high frequency estimate \cite[Proposition 4.8, part (i)]{AKM}.  Note that this operator theoretic part of the proof in the case $p=2$, is the part that does not necessarily hold for all $p \in (1,\infty)$.

\begin{Theorem}
\label{thm:high}
Suppose $\Pi_B$ is a perturbed Hodge-Dirac operator.
Suppose $M=10n$ and $p \in (\max\{1,(p_H)_{*}\},2]$. Then
$$
\|(t,x) \mapsto (Q_{t}^{B})^{M}(I-{P_{t}}^{\tilde{N}})u(x)\|_{T^{p,2}} \leq C_p \|u\|_{p} \quad \forall u \in \overline{\calR_{p}(\Gamma)}.
$$
\end{Theorem}

This result is proven in Sections \ref{sec:high1}, \ref{sec:high2} and \ref{sec:high}. \\

We now show how to prove Theorem \ref{thm:main} from Theorems \ref{thm:low} and \ref{thm:high}. Notice that the large (and somewhat arbitrary) value of $M$ appearing in Theorem \ref{thm:high} is appropriately reduced as part of this proof.

\begin{proof}[Proof of Theorem \ref{thm:main}]
For $p \in (2,\infty)$, the claim has been shown in Proposition \ref{lem:reverse}. From now on, suppose $p \in (\max\{1,(p_H)_\ast\},2]$. 
Without loss of generality, we can assume that $M=10n$. Indeed, the result for $M>\frac{n}{2}$ will then follow by Theorem \ref{thm:hardy}.
Combining Theorem \ref{thm:low} and Theorem \ref{thm:high}, we have that
$$
\|(t,x) \mapsto (Q_{t}^{B})^{M}u(x)\|_{T^{p,2}} \lesssim \|u\|_{p} \quad \forall u \in \overline{\calR_{p}(\Gamma)}.
$$
Applying the same results to $\underline{\Pi_{B}}$ gives
$$
\|(t,x) \mapsto (\underline{Q_{t}^{B}})^{M}u(x)\|_{T^{p,2}} \lesssim \|u\|_{p}\quad \forall  u \in \overline{\calR_{p}(\Gamma^\ast)}.$$
\end{proof}

We conclude this section by showing that in certain situations the results can be improved when restricted to subspaces of the form $L^p(\R^{n};W)$, where $W$ is a subspace of $\C^{N}$. The proof given depends on  Corollary \ref{cor:fullHFW},  as well as the preceding material.

\begin{Theorem}
\label{cor:restW}
Suppose $\Pi_B$ is a perturbed Hodge-Dirac operator.
Let $W$ be a subspace of $\C^{N}$ that is stable under $\widehat{\Gamma^{*}}(\xi)\widehat{\Gamma}(\xi)$ and $\widehat{\Gamma}(\xi)\widehat{\Gamma^{*}}(\xi)$ for all $\xi \in \R^{n}$. \\ 
(1) Let $p \in (\max\{1,r_{*}\},2]$ for some $r \in (1,2)$, and $M\in 2\N$ with $M \geq 10n$.
 Suppose further that $(\Pi_{B}(r))$ holds,
$\{t\Gamma^{*}_{B} P_{t}^{B}\mathbb{P}_{W} \;;\; t \in \R^{*}\}$ and $\{t\Gamma P_{t}^{B}\mathbb{P}_{W} \;;\; t \in \R^{*}\}$ are uniformly bounded in $\calL(L^{r})$, and that
 $\{(R_t^B)^{\frac{M}{2}-2} \Gamma \mathbb{P}_{W} \;;\; t \in \R^{*}\}$ and $\{(R_t^B)^{\frac{M}{2}-2} B_{1}\Gamma^{*} \mathbb{P}_{W}\;;\; t \in \R^{*}\}$ have $\dot{W}^{1,r}$-$L^2$ off-diagonal bounds of every order on balls, where $\mathbb{P}_{W}$ denotes the projection from $L^r(\R^{n};\C^{N})$ onto $L^r(\R^{n};W)$.

Then, for $\mu \in (\omega,\frac{\pi}{2})$, we have
\begin{equation}\tag{i}
 \|f(\Pi_{B})\Gamma u\|_{p} \lesssim \|f\|_{\infty}\|\Gamma u\|_{p} \quad \forall f \in H^\infty(S^o_{\mu}) \quad  \forall u \in \calD_{p}(\Gamma)\cap L^{2}(\R^{n};W).
\end{equation} 
Moreover, $\overline{\calR_p(\Gamma|_{L^p(\R^n;W)})} \subseteq H^p_{\Pi_B}(\R^n;\C^N)$ with
\begin{equation}\tag{ii}
	\|v\|_{H^p_{\Pi_B}} \approx \|v\|_{p} \qquad \forall v \in \overline{\calR_p(\Gamma|_{L^p(\R^n;W)})}.
\end{equation}
(2) If $L^{2}(\R^{n};W) \subset \overline{\calR_{2}(\Gamma^{*}_{B})}$, $(p_{H})_{*} > 1$, and $(\Pi_{B}(r))$ holds for all $r \in ((p_H)_{*},2]$, then the hypotheses and conclusions of (1) hold for all $p \in (\max\{1,(p_{H})_{**}\},2]$. 
 In particular, 
\begin{equation}\tag{iii} \|({\Pi_B}^2)^{1/2}u\|_p \lesssim \|\Gamma u\|_p \quad  \forall u \in \calD_{p}(\Gamma)\cap L^{2}(\R^{n};W).
\end{equation}
\end{Theorem}

\begin{proof}

 (1) (i) The hypotheses of Corollary \ref{cor:fullHFW} are satisfied by assumption.
We therefore obtain, for all $p \in (\max\{1,r_{*}\},2]$,
$$
\|(t,x) \mapsto (Q_{t}^{B})^{M}(I-{P_{t}}^{\tilde{N}})\Gamma v(x)\|_{T^{p,2}} \lesssim \|\Gamma v\|_{p} \qquad \forall v \in \calD_{p}(\Gamma)\cap L^{p}(\R^{n};W).$$
Combined with Theorem \ref{thm:low}, this gives
\begin{equation} \label{eq:rieszCharact}
\|(t,x) \mapsto (Q_{t}^{B})^{M}\Gamma v(x)\|_{T^{p,2}} \lesssim \|\Gamma v\|_{p} \qquad \forall v \in \calD_{p}(\Gamma)\cap L^{p}(\R^{n};W).
\end{equation}
Therefore we have, for all $f \in \Psi(S^o_{\mu})$, $v \in \calD_{p}(\Gamma)\cap L^{p}(\R^{n};W)$, that (i) holds:
\begin{equation*}
\begin{split}
\| f(\Pi_{B}) \Gamma v\|_{p} &\lesssim \|(t,x) \mapsto (Q_{t}^{B})^{M} f(\Pi_{B}) \Gamma v(x)\|_{T^{p,2}} 
\approx \|f(\Pi_{B}) \Gamma v\|_{H^{p}_{\Pi_{B}}}\\
&\lesssim \|\Gamma v\|_{H^{p}_{\Pi_{B}}} \approx \|(t,x) \mapsto (Q_{t}^{B})^{M}  \Gamma v(x)\|_{T^{p,2}} 
\lesssim \|\Gamma v\|_{p},
\end{split}
\end{equation*}
where we have used Proposition \ref{lem:reverse}, Theorem \ref{thm:hardy}, and \eqref{eq:rieszCharact}. The estimate holds for all $f\in H^\infty(S^o_\mu)$ on taking limits as usual.\\
 (ii) This follows from \eqref{eq:rieszCharact} and the reverse inequality shown in Proposition \ref{lem:reverse}.\\ 
(2) Let $q>p_{H}$ with $q_{*}>1$ and $r \in (q_{*},q]$. By  Lemma  \ref{lem:Lp-LqOD}, and the fact that $L^{2}(\R^{n};W) \subset \overline{\calR_{2}(\Gamma^{*}_{B})}$
by assumption, we have that
$$
\|t^{n(\frac{1}{q_*}-\frac{1}{2})}R_{t}^{B}u\|_{q} \lesssim \|u\|_{L^{q_*}(\R^{n};W)} \qquad \forall t>0 \quad \forall u \in L^{2}(\R^{n};W) \cap L^{q_*}(\R^{n};W).
$$
Iterating, and interpolating with $L^2$-$L^2$ off-diagonal bounds (see Lemma \ref{lem:ODcomp}), we get that   $\{(R_{t}^{B})^{\frac{M}{2}-2} \;;\; t>0\}$ has $L^{r}(\R^{n};W)$-$L^2$ off-diagonal bounds of every order and that $\{R_{t}^{B} \;;\; t>0\}$ has $L^{r}(\R^{n};W)$-$L^{\tilde{r}}$ off-diagonal bounds of every order for some $\tilde{r} \in (r,2]$. 
 The former implies that both $\{(R_t^B)^{\frac{M}{2}-2} \Gamma \mathbb{P}_{W} \;;\; t \in \R^{*}\}$ and $\{(R_t^B)^{\frac{M}{2}-2} B_{1}\Gamma^{*} \mathbb{P}_{W}\;;\; t \in \R^{*}\}$ have $\dot{W}^{1,p}$-$L^2$ off-diagonal bounds of every order on balls, while the latter implies that 
 $\{t\Gamma P_{t}^{B}\mathbb{P}_{W} \;;\; t\in \R^{*}\} = \{Q_{t}^{B}\mathbb{P}_{W} \;;\; t\in \R^{*}\}
 = \{\frac{1}{2i}(R_{-t}^{B}-R_{t}^{B})\mathbb{P}_{W} \;;\; t\in \R^{*}\}$ is uniformly bounded in $\calL(L^r)$, by 
Lemma \ref{lem:ODcomp}.  This yields the hypotheses and hence the conclusions of (1).
 To obtain (iii), apply (i) with $f(z) = \sgn(z)$:
$$\|({\Pi_B}^2)^{1/2}u\|_p =\|\sgn(\Pi_B)\Pi_Bu\|_p=\|\sgn(\Pi_B)\Gamma u\|_p\leq C_p\|\Gamma u\|_p \  \forall u \in \calD_{p}(\Gamma)\cap L^{2}(\R^{n};W),
$$
noting that $u \in \overline{\calR_{2}(\Gamma^{*}_{B})}$ by assumption. 
\end{proof}

\section{Consequences}
\label{sec:cons}

\subsection{Differential forms.} The motivating example for our formalism is perturbed differential forms, where $\C^N = \Lambda = \oplus_{k=0}^n
\Lambda^k=\wedge_\C\R^n$, 
the complex exterior algebra over $\R^n$,  and $\Gamma = d$, the exterior derivative, acting in $L^p(\R^n;\Lambda)=\oplus_{k=0}^n L^p(\R^n;\Lambda^k)$.
If the multiplication operators $B_1, B_2$ satisfy the conditions of Definition  \ref{pHD}, then $\Pi_B = d+B_1 d^\ast B_2$ is a perturbed Hodge-Dirac operator, and it is from here that it gets its name. The $L^p$ results stated in Section \ref{sec:res}  all apply to this operator.\\

Typically, but not necessarily,  the operators $B_j, j=1,2$  split as $B_j=B^0_j \oplus \dots \oplus B^n_j$, where $B^k_j \in L^\infty(\R^n;\calL(\Lambda^k))$, in which case Corollary \ref{cor:fc} has a converse in the following sense (cf. \cite[Theorem 5.14]{AMR} for an analogous result for Hodge-Dirac operators on Riemannian manifolds).

\begin{Prop} \label{prop:ConvHodge}
Suppose $\Pi_B = d+B_1 d^\ast B_2$ is a perturbed Hodge-Dirac operator as above, with 
$B_j, j=1,2$  splitting as $B_j=B^0_j \oplus \dots \oplus B^n_j$, where $B^k_j \in L^\infty(\R^n;\calL(\Lambda^k))$, $k=0,\ldots,n$.
Suppose that for some $p \in (1,\infty)$, \eqref{PiB3} holds and $\Pi_B$ is an $\omega$-bisectorial operator in $L^p(\R^n;\Lambda)$  
with a bounded $H^\infty$ functional calculus in $L^p(\R^n;\Lambda)$. Then $p\in(p_H, p^H)$. 
\end{Prop}

We do not know if this converse holds for all perturbed Hodge-Dirac operators. It does, however, hold for all examples given in this section.  
\begin{proof} 
We need to show that $\Pi_B$ Hodge decomposes $L^p(\R^n;\Lambda)$, i.e. $L^p(\R^n;\Lambda)= \calN_p(\Pi_B) \oplus \overline{\calR_p(\Gamma)} \oplus \overline{\calR_p(\Gamma_B^\ast)}$, where $\Gamma=d$ and $\Gamma^\ast_B = B_1d^\ast B_2$.
Since $\Pi_B$ is bisectorial in $L^p(\R^n;\Lambda)$, we know that 
$L^p(\R^n;\Lambda)= \calN_p(\Pi_B) \oplus \overline{\calR_p(\Pi_B)}$. Therefore, it suffices to show that 
\[
	\|\Gamma u\|_p + \|\Gamma_B^\ast u\|_p \approx \|\Pi_B u\|_p \quad \forall u \in \calD_p(\Pi_B)=\calD_p(\Gamma) \cap \calD_p(\Gamma_B^\ast).
\] 
For $k=0,\ldots,n$ and $u \in L^p(\R^n;\Lambda)$, denote by $u^{(k)} \in L^p(\R^n;\Lambda^k)$ the $k$-th component of $u$. 
Note that $\Gamma: L^p(\R^n;\Lambda^k) \to L^p(\R^n;\Lambda^{k+1})$, $\Gamma_B^\ast: L^p(\R^n;\Lambda^{k+1}) \to L^p(\R^n;\Lambda^k)$, $k=0,\ldots,n-1$, and ${\Pi_B}^2: L^p(\R^n;\Lambda^k) \to L^p(\R^n;\Lambda^k)$, $k=0,\ldots,n$. 
Using that $\sgn(\Pi_B)$,
where 
 $$\sgn(z) = \begin{cases} 1, \quad \text{if} \; \Re z>0, \\ -1, \quad \text{if} \; \Re z<0, \end{cases} \forall z \in S_{\mu}\setminus\{0\}\quad\text{and}\quad \sgn(0)=0,$$
is bounded in $L^p(\R^n;\Lambda)$ since $\Pi_{B}$ has a bounded $H^{\infty}$ calculus, we therefore get for $u \in \calD_p(\Pi_B)$:
\begin{align*}
	\|\Gamma u\|_p + \|\Gamma_B^\ast u\|_p
	&\approx \sum_{j=0}^n \|(\Gamma u)^{(j)}\|_p + \sum_j \|(\Gamma_B^\ast u)^{(j)}\|_p\\
	&\approx \sum_{k=0}^n ( \|\Gamma u^{(k)}\|_p + \|\Gamma_B^\ast u^{(k)}\|_p)
	\approx \sum_{k=0}^n \|\Pi_B u^{(k)}\|_p 
	\approx \sum_{k=0}^n \|({\Pi_B}^2)^{1/2} u^{(k)}\|_p\\
	& \approx \sum_{k=0}^n \|(({\Pi_B}^2)^{1/2} u)^{(k)}\|_p
	\approx \|({\Pi_B}^2)^{1/2}u\|_p 
	\approx \|\Pi_B u\|_p. 
\end{align*}
\end{proof}

\subsection{Second order elliptic operators}
Let $L$ denote  the uniformly elliptic second order  operator defined by 
$$  L f=-a\div A \nabla f=-a\sum_{j,k=1}^n \partial_j (A_{j,k} \partial_k f) $$
where $a\in L^\infty(\R^n)$  with $\Re(a(x)) \geq \kappa_1 >0$ a.e. and
$A\in L^\infty(\R^n;{\mathcal L}(\C^n))$  with $\Re(A(x)) \geq \kappa_2 I >0$ a.e.
 Associated with $L$ is the Hodge-Dirac operator 
$$\Pi_B = \Gamma+\Gamma^\ast_B=
\Gamma + B_1\Gamma^\ast B_2 =  \left[\begin{array}{cc}
0&-a\div  A  \\ \nabla &0 \end{array}\right]\quad\text{acting in}\quad L^2(\R^n;(\C^{1+n}))=\begin{array}{c}
L^2(\R^n) \\ \oplus \\ L^2(\R^n;\C^n) \end{array}
$$
where $$ \Gamma= \left[\begin{array}{cc}
0&0 \\  \nabla &0 \end{array}\right],\ \Gamma^\ast= \left[\begin{array}{cc}
0&-\div \\  0 &0 \end{array}\right],\  B_1 =\left[\begin{array}{cc}
a&0 \\ 0&0 \end{array}\right],\ B_2 =\left[\begin{array}{cc}
0&0 \\ 0&A \end{array}\right]\ ,
$$
so that
$${\Pi_B}^2
=\left[\begin{array}{cc}
 L &0 \\ 0 & \tilde L   \end{array}\right] \qquad(\text{where}\quad \tilde L=- \nabla a\div  A)\ .
 $$
As shown in \cite{AKM} (and recalled in Theorem \ref{thm:p=2}), $\Pi_B$ is an $\omega$-bisectorial operator with an $H^\infty$ functional calculus in $L^2$, so that in particular $\sgn(\Pi_B)$ is a bounded operator on $L^2(\R^n;\C^{1+n})$.

Using the expression
$$
\sgn(\Pi_B)= ({\Pi_B}^2)^{-1/2}\Pi_B
=\left[\begin{array}{cc}
 0 &-L^{-1/2}\,  a\div  A \\ \nabla L^{-1/2} &0   \end{array}\right]
 \ ,$$ on $\calD(\Pi_B)$,
and the
fact that $(\sgn(\Pi_B))^2 u = u$ for all $u \in \overline{\calR_2(\Pi_B)}=L^2(\R^n)\oplus \overline{\calR_2(\nabla)}$, we find  that $\|\nabla L^{-1/2}g\|_2 \approx \|g\|_2$ for all $g\in \calR(L^{1/2})$, i.e.  $\|\nabla f\|_2 \approx \|L^{1/2} f\|_2$ for all $f\in \calD(L^{1/2}) = W^{1,2}(\R^n)$, this being the Kato conjecture, previously solved in \cite{ahlmt} (when $a=1$).\\

Turning now to $L^p$, we see that by our hypotheses, (\ref{PiB3}) holds for all $p\in(1,\infty)$, and that
$$ \calN_p(\Pi_B)=\begin{array}{c}
\{0\} \\ \oplus \\ \calN_p(\div A) \end{array},\quad \overline{\calR_p(\Gamma)} =\begin{array}{c}
\{0\} \\ \oplus \\ \overline{\calR_p(\nabla)} \end{array},\quad \overline{\calR_p(\Gamma^\ast_B)} =\begin{array}{c}
L^p(\R^n) \\ \oplus \\ \{0\} \end{array}.
$$

So $p\in(p_H,p^H)$, i.e. the Hodge decomposition  $L^p(\R^n;(\C^{1+n})) =\calN_p(\Pi_B)\oplus\overline{\calR_p(\Gamma)}\oplus\overline{\calR_p(\Gamma^\ast_B)}$ holds, if and only if 
$L^p(\R^n;\C^n) = \calN_p(\div\,A)\oplus \overline{\calR_p(\nabla)}$.\\

Turning briefly to Hardy space theory, we have
$$H^2_{\Pi_B} = \overline{\calR_2(\Pi_B)} = \begin{array}{c}
L^2(\R^n) \\ \oplus \\ \overline{\calR_2(\nabla)} \end{array} =  \begin{array}{c}
H^2_L \\ \oplus \\ H^2_{\tilde L}  \end{array}\quad\text{and}\qquad
H^p_{\Pi_B} = H^p_{{\Pi_B}^2} = \begin{array}{c}
H^p_L \\ \oplus \\ H^p_{\tilde L}  \end{array}
$$
for all $p\in(1,\infty)$. We remark that $L$ has a bounded $H^\infty$ functional calculus in 
$H^p_L$, and that $\sgn(\Pi_B)$ is an isomorphism interchanging  $H^p_L$ and $H^p_{\tilde L}$.\\

We now state how the results of this section apply to  $\Pi_B$, and have as consequences for $L$ and its Riesz transform, results which are known, at least when $a=1$ (see \cite{AuscherMemoirs} and \cite[Section 5]{HMMc}).

\begin{Cor}
\label{cor:div-form}
Let  $L=-a\div A \nabla$ be a uniformly elliptic operator as above.
Then the following hold:
\begin{enumerate}
\item If $p_{H}<p<p^H$, then $\Pi_B$ is an $\omega$-bisectorial operator in $L^p(\R^n;\C^{1+n})$ with a bounded $H^\infty$ functional calculus.
\item
If $\max\{1,(p_{H})_{*}\}<p<p^H$, then $H^p_{\Pi_B}= \overline{\calR_p(\Pi_B)}$ and  
$\Pi_B$ is an $\omega$-bisectorial operator in $ \overline{\calR_p(\Pi_B)}$ with a bounded $H^\infty$ functional calculus, so that $L$ has a bounded $H^\infty$ functional calculus in $L^p(\R^n)$, and $\calD_p(L^{1/2}) = W^{1,p}(\R^{n})$ with
$\|L^{1/2} f\|_{p} \approx\|\nabla f\|_{p}  $.
\item
If $r \in (1,2]$, $\max\{1,\min\{r_{*},(p_{H})_{**}\}\}<p<p^H$, $M \in \N$ with $M \geq 10n$, and $\{(I+t^2L)^{-(\frac{M}{2}-1)}\,;\,t>0\}$
 has $L^r(\R^n)$-$L^2(\R^n)$ off-diagonal bounds of every order, 
then $H^p_{\tilde L}=\overline{\calR_p(\nabla)}$ and
$
  \|L^{1/2} f\|_{p} \lesssim \|\nabla f\|_{p}$ for all  $f \in W^{1,p}(\R^{n})$.
  Also $g\in H^p_L$ if and only if $\nabla L^{-1/2}g \in L^p(\R^n;\C^n)$, with $\|g\|_{H^p_L}\approx \|\nabla L^{-1/2}g\|_p$.
  \end{enumerate}
\end{Cor}

We remark that the hypotheses of (3) can also be stated in terms of off-diagonal bounds for the semigroup $(e^{-tL})_{t>0}$.

\begin{proof}
As described above, $\Pi_B$ is a perturbed Hodge-Dirac operator.
(1) follows from Corollary \ref{cor:fc} (3).
(2) follows from Corollary \ref{cor:hardy} (1) and (2), noting that in our situation, the decomposition $\overline{\calR_p(\Pi_B)} = \overline{\calR_p(\Gamma)} \oplus \overline{\calR_p(\Gamma_B^\ast)}$ holds for all $p \in (1,\infty)$. 
(3) Set $W=\C$. As stated before, $L^2(\R^n;W) \subseteq \overline{\calR_2(\Gamma^\ast_B)} =L^2(\R^n) \oplus \{0\}$. 
For $w \in W$ and $\xi \in \C^{n}$, we have that $\widehat{\Gamma^{*}}(\xi)\widehat{\Gamma}(\xi) w = (\sum \limits _{j=1} ^{n} |\xi_{j}|^{2})w$ and
$\widehat{\Gamma}(\xi)\widehat{\Gamma^{*}}(\xi) w = 0$, so $W$ is stable under $\widehat{\Gamma^{*}}(\xi)\widehat{\Gamma}(\xi)$ and
$\widehat{\Gamma}(\xi)\widehat{\Gamma^{*}}(\xi)$. 
If  $(p_{H})_{*}>1$,  we can therefore apply Theorem \ref{cor:restW}, which gives (3).
If $(p_{H})_{*}\leq1$, (3) follows from (2).
\end{proof}

\begin{Remark}
If $A=I$, one has $(p_H,p^H)=(1,\infty)$, so the estimates in Corollary \ref{cor:div-form} hold for all $p \in (1,\infty)$, in agreement with the results of \cite{McN} concerning $L=-a\Delta$.
\end{Remark}

\subsection{First order systems of the form $DA$}
 Results for operators of the form $DA$ or $AD$,  used in studying boundary value problems as in \cite{AAMc}, can be obtained in a similar way to those in this paper, building on the $L^2$ theory in \cite{AAMc2}. However they can also be obtained as consequences of the results for $\Pi_B$, as was shown in Section 3 of \cite{AKM} when $p=2$. Let us briefly summarise this in the $L^p$ case.\\

Let $D$ be a first order system which is self-adjoint in $L^2(\R^n;\C^N)$, and $A\in L^\infty(\R^n;{\mathcal L}(\C^N))$  with $\Re(ADu,Du) \geq \kappa {\|Du\|_2}^2 $ for all $u\in\calD_2(D)$.  Set
$$\Pi_B = \Gamma+\Gamma^\ast_B=
\Gamma + B_1\Gamma^\ast B_2 =  \left[\begin{array}{cc}
0&AD A  \\ D &0 \end{array}\right]\quad\text{acting in}\quad L^2(\R^n;\C^{2N})=\begin{array}{c}
L^2(\R^n;\C^N) \\ \oplus \\ L^2(\R^n;\C^N) \end{array}
$$
where $$ \Gamma= \left[\begin{array}{cc}
0&0 \\  D &0 \end{array}\right],\ \Gamma^\ast= \left[\begin{array}{cc}
0&D\\  0 &0 \end{array}\right],\  B_1 =\left[\begin{array}{cc}
A&0 \\ 0&0 \end{array}\right],\ B_2 =\left[\begin{array}{cc}
0&0 \\ 0&A \end{array}\right]\ .
$$
Then $\Pi_B$ is a Hodge-Dirac operator, and so, by \cite{AKM}, has a bounded $H^\infty$ functional calculus in $L^2(\R^n;\C^{2N})$. \\

Turning to $p\in(1,\infty)$, we find that \eqref{PiB3} holds if and only if 
\begin{align*}\|u\|_p &\lesssim \|Au\|_p\quad \forall u\in\overline{\calR_p(D)}\qquad\text{and}\\
\|v\|_{p'} &\lesssim \|A^\ast v\|_{p'}\quad \forall v\in\overline{\calR_{p'}(D)}\ ,
\end{align*} and that \eqref{PiB4} is the same.
Assuming this (in particular if $A$ is invertible in $L^\infty$), we find that 
$$ \calN_p(\Pi_B)=\begin{array}{c}
\calN_p(D)\\ \oplus \\ \calN_p(DA) \end{array},\quad \overline{\calR_p(\Gamma)} =\begin{array}{c}
\{0\} \\ \oplus \\ \overline{\calR_p(D)} \end{array},\quad \overline{\calR_p(\Gamma^\ast_B)} =\begin{array}{c}
\overline{\calR_p(AD)} \\ \oplus \\ \{0\} \end{array}.
$$
and hence that $\Pi_B$ Hodge decomposes $L^p(\R^n;\C^{2N})$, i.e. $p\in(p_H,p^H)$, if and only if 
\begin{equation} \label{eq:HodgeDA}
L^p(\R^n;\C^{N})= \calN_{p}(DA)\oplus\overline{\calR_p(D)}.
\end{equation}
This can be seen following the arguments in Proposition \ref{equiv}: Under \eqref{PiB3}, \eqref{eq:HodgeDA} holds if and only if $L^{p'}(\R^n;\C^{N})= \calN_{p'}(D)\oplus\overline{\calR_{p'}(A^\ast D)}$, i.e. if and only if  $L^{p}(\R^n;\C^{N})= \calN_{p}(D)\oplus\overline{\calR_{p}(AD)}$.\\

As in \cite[Proof of Theorem 3.1]{AKM} and \cite[Corollary 8.17]{HMcP2}, we compute that, when defined,
$$f(DA)u=\left[\begin{array}{cc}
0&I \end{array}\right]f(\Pi_B)\left[\begin{array}{c}
A\\ I \end{array}\right]u,
$$
so that results concerning $DA$ having a bounded $H^\infty$ functional calculus in $L^p(\R^n;\C^{N})$ can be obtained from our results for $\Pi_B$ in $L^p(\R^n;\C^{2N})$. Moreover  results concerning bounds on $f(DA)u$ when $u\in \overline{\calR_p(D)}$ 
can be obtained from our results on $f(\Pi_B)v$ when $v \in \overline{\calR_p(\Gamma)}$ and on $f(\Pi_B)w$ when $w \in \overline{\calR_p(\Gamma^\ast_B)}$.\\

We leave further details to the reader, as well as consideration  of $AD$.

\section{Low frequency estimates: The Carleson measure argument}
\label{sec:low}

In this section, we prove the low frequency estimate, Theorem \ref{thm:low}. 
By Lemma \ref{tent-boundedness} and Theorem \ref{thm:p=2}, we have
$$
\|(t,x) \mapsto (Q_{t}^{B})^{M} {P_{t}}^{\tilde{N}}u(x)\|_{T^{p,2}}
\lesssim  \|(t,x) \mapsto Q_{t}^{B}{P_{t}}^{\tilde{N}}u(x)\|_{T^{p,2}}
$$
so it suffices to prove
\begin{equation}\label{reduction} \|(t,x) \mapsto Q_{t}^{B} {P_{t}}^{\tilde{N}}u(x)\|_{T^{p,2}} \lesssim \|u\|_{p}
\quad \forall u \in \overline{\calR_p(\Pi)}.
\end{equation}

According to Theorem \ref{thm:p=2} and Lemma \ref{lem:ODcomp}, the operator $Q_t^B$ extends to an operator $Q_t^B: L^\infty(\R^n;\C^N) \to L^2_{\loc}(\R^n;\C^N)$ with
\begin{equation} \label{eq:L2loc-est}
	\|Q_t^B u\|_{L^2(B(x_0,t))} \lesssim t^{\frac{n}{2}} \|u\|_{\infty} \quad \forall u \in L^\infty(\R^n;\C^N), \; x_0 \in \R^n, \;t>0.
\end{equation}
We can therefore define
\begin{equation}
\label{def:gamma_t}
	\gamma_t(x)w :=(Q_t^B w)(x)
	\qquad \forall w \in \C^N, \; x \in \R^n,
\end{equation}
where, on the right-hand side, $w$ is considered as the constant function defined  by $w(x)=w$ for all $x\in \R^{n}$. Note that the definition of $\gamma_t$ is different from the one in \cite[Definition 5.1]{AKM}. \\

 In order to prove \eqref{reduction},  we use the splitting 
\[
	Q_t^B {P_t}^{\tilde{N}} u = [Q_t^B {P_t}^{\tilde{N}} u - \gamma_t A_t {P_t}^{\tilde{N}}u]
		+ \gamma_t A_t {P_t}^{\tilde{N}} u,
\]
and refer to $\gamma_t A_t {P_t}^{\tilde{N}}u$ as the principal part, and $[Q_t^B {P_t}^{\tilde{N}}u - \gamma_t A_t {P_t}^{\tilde{N}}u]$ as the principal part approximation.\\
 We use the following \emph{dyadic decomposition} of $\R^n$. 
Let $\Delta =\bigcup_{j=-\infty}^{\infty} \Delta_{2^j}$, where $\Delta_{2^j}:=\{2^j(k+(0,1]^n) \,:\, k \in \Z^n\}$. For a dyadic cube $Q \in \Delta_{2^j}$, denote by $l(Q) =2^j$ its sidelength, by $|Q| =2^{jn}$ its volume. We set $\Delta_t=\Delta_{2^j}$, if $2^{j-1}<t\leq 2^j$. 
The dyadic averaging operator $A_t: L^2(\R^n;\C^N) \to L^2(\R^n;\C^N)$ is defined by
\[
	A_t u(x):= \frac{1}{|Q_{x,t}|}\int_{Q_{x,t}} u(y)\,dy=:\langle u\rangle_{Q_{x,t}}
	\qquad \forall u \in L^2(\R^n;\C^N), \; x \in \R^n, \; t>0,
\]
where $Q_{x,t}$ is the unique dyadic cube in $\Delta_t$ that contains $x$. \\

Let us make the following simple observation: for all $\eps>0$, there exists a constant $C>0$ such that for all $t>0$
\[
	\sup_{Q \in \Delta_t} \sum_{R \in \Delta_t} \left(1+\frac{\dist(Q,R)}{t}\right)^{-(n+\eps)} \leq C.
\]

We first consider the principal part approximation, similar to \cite[Proposition 5.5]{AKM}.

\begin{Prop} \label{Prop:Principal-part-approx}
Suppose $\Pi_B$ is a perturbed Hodge-Dirac operator. 
Suppose $p \in (1,\infty)$. Then
\[
		\|(t,x) \mapsto Q_t^B {P_t}^{\tilde{N}} u(x) -\gamma_t(x) A_t {P_t}^{\tilde{N}} u(x)\|_{T^{p,2}} 
		\leq C_p \|u\|_{p} 
		\qquad \forall u \in \overline{\calR_p(\Pi)}.
\]
\end{Prop}

\begin{proof}
Fix $x \in \R^n$. For $t>0$, we cover the ball $B(x,t)$ by a finite number of cubes $Q \in \Delta_t$. According to Theorem \ref{thm:p=2}, $Q_t^B$ has $L^2$-$L^2$ off-diagonal bounds of every order $N'>0$. This, together with the Cauchy-Schwarz inequality and the Poincar\'e inequality (see \cite[Lemma 5.4]{AKM}), yields the following for $Q \in \Delta_{t}$:
\begin{align*}
	& (\int_0^\infty \int_Q |Q_t^B{P_{t}}^{\tilde{N}}u(y) -\gamma_t(y) A_t{P_{t}}^{\tilde{N}}u(y)|^2\,\frac{dydt}{t^{n+1}})^{\frac{1}{2}} \\
	& \qquad =(\int_0^\infty \int_Q |Q_t^B({P_{t}}^{\tilde{N}}u-\skp{{P_{t}}^{\tilde{N}} u}_Q)(y)|^2 \,\frac{dydt}{t^{n+1}})^{\frac{1}{2}} \\
	& \qquad \leq (\int_0^\infty (\sum_{R \in \Delta_t} \|Q_t^B\Eins_R({P_{t}}^{\tilde{N}} u - \skp{{P_{t}}^{\tilde{N}}u}_Q)\|_{L^2(Q)})^2 \,\frac{dt}{t^{n+1}})^{\frac{1}{2}} \\
	& \qquad \lesssim (\int_0^\infty ( \sum_{R \in \Delta_t} (1+\frac{\dist(Q,R)}{t})^{-N'} \|{P_{t}}^{\tilde{N}}u-\skp{{P_{t}}^{\tilde{N}}u}_Q\|_{L^2(R)} )^2 \,\frac{dt}{t^{n+1}} )^{\frac{1}{2}}\\
	& \qquad \lesssim (\int_0^\infty \sum_{R \in \Delta_t} (1+\frac{\dist(Q,R)}{t})^{-N'} \|{P_{t}}^{\tilde{N}}u - \skp{{P_{t}}^{\tilde{N}} u}_Q\|_{L^2(R)}^2 \,\frac{dt}{t^{n+1}} )^{\frac{1}{2}}\\
	& \qquad \lesssim (\int_0^{\infty} \int_{\R^n} (1+\frac{\dist(Q,y)}{t})^{-N'+2n} |t\nabla {P_{t}}^{\tilde{N}}u(y)|^2 \,\frac{dydt}{t^{n+1}})^{\frac{1}{2}} \\
	& \qquad \lesssim \sum_{j=0}^{\infty} (\int_0^\infty \int_{S_j(Q)} 2^{-j(N'-2n)} |t\nabla {P_{t}}^{\tilde{N}} u(y)|^2 \,\frac{dydt}{t^{n+1}})^{\frac{1}{2}}. 
\end{align*}
By change of angle in tent spaces, see Lemma \ref{lem:angle}, we thus get
\begin{align*}
 	 &\|(t,x) \mapsto Q_t^B{P_{t}}^{\tilde{N}}u(x) - \gamma_t(x)A_t{P_{t}}^{\tilde{N}}u(x)\|_{T^{p,2}}
 		 \lesssim \sum_{j=0}^{\infty} 2^{-\frac{j}{2}(N'-2n)} \|(t,x)\mapsto t\nabla {P_{t}}^{\tilde{N}}u(x)\|_{T^{p,2}_{2^j}} \\
 		& \qquad  \lesssim \sum_{j=0}^{\infty} 2^{-\frac{j}{2}(N'-2n)} 2^{j\frac{n}{\min\{p,2\}}} \|(t,x)\mapsto t\nabla {P_{t}}^{\tilde{N}}u(x)\|_{T^{p,2}} 
 		\lesssim \|(t,x)\mapsto t\nabla {P_{t}}^{\tilde{N}}u(x)\|_{T^{p,2}} ,
\end{align*}
choosing $N'> 2n+\frac{2n}{\min\{p,2\}}$. 
Since ${P_{t}}^{\tilde{N}}$ is a Fourier multiplier, we have that, for $u= \Pi v$ with $v \in \calD_2(\Pi)$, and all $j=1,...,n$:
$$
t\partial_{x_{j}}{P_{t}}^{\tilde{N}}u = \tilde{Q}_{t}(\partial_{x_{j}}v)
$$
with $\tilde{Q}_t= t\Pi {P_{t}}^{\tilde{N}}$.
Therefore, by Proposition \ref{lem:hardyPi} 
and Theorem \ref{thm:hardy} (for $p \leq 2$), or Proposition \ref{lem:reverse} (for $p\geq 2$),
along with Proposition \ref{prop:pi} (6), we have that
$$
\|(t,x)\mapsto t\nabla {P_{t}}^{\tilde{N}}u(x)\|_{T^{p,2}} \lesssim \underset{j=1,...,n}{\max}\|(t,x)\mapsto \tilde{Q}_t(\partial_{x_{j}}v)(x)\|_{T^{p,2}}
\lesssim \underset{j=1,...,n}{\max}\|\partial_{x_{j}}v\|_{p} \lesssim \|u\|_{p},
$$
which concludes the proof.
\end{proof}

Turning now to the estimate for the principal part, we first show  that $\{\gamma_tA_t\}_{t>0}$ defines a bounded operator on $T^{p,2}$ for all $p \in (1,\infty)$. This is an
analogue of \cite[Proposition 5.7]{AKM}.

\begin{Lemma} \label{tent-bdd-gamma}
Suppose $\Pi_B$ is a perturbed Hodge-Dirac operator. 
Suppose $p \in (1,\infty)$. Then
\[
	\|(t,x)\mapsto \gamma_t(x)A_tF(t,\,.\,)(x)\|_{T^{p,2}}
	\leq C_p \|F\|_{T^{p,2}} \qquad \forall F \in T^{p,2}(\R^{n+1}_+;\C^N).
\]
\end{Lemma}

\begin{proof}
First observe that, given $x \in \R^n$ and $t>0$,
\begin{align*}
	\|A_tF(t,\,.\,)\|_{L^{\infty}(B(x,t))}
	&=\sup_{y \in B(x,t)} |A_tF(t,y)|
	=\sup_{\substack{Q \in \Delta_t \\ B(x,t) \cap Q \neq \emptyset}} |Q|^{-1}|\int_Q F(t,z) \,dz|\\
	&\leq \sup_{\substack{Q \in \Delta_t \\ B(x,t) \cap Q \neq \emptyset}}|Q|^{-\frac{1}{2}} \|F(t,\,.\,)\|_{L^2(B(x,5t))}
	\lesssim t^{-\frac{n}{2}} \|F(t,\,.\,)\|_{L^2(B(x,5t))}.
\end{align*}
According to \eqref{eq:L2loc-est}, we have on the other hand
$
	\|\gamma_t\|_{L^2(B(x,t))} \lesssim t^{\frac{n}{2}},
$
and consequently
\begin{align*}
	 (\int_0^\infty \int_{B(x,t)} |\gamma_t(y) \cdot A_tF(t,y)|^2 \,\frac{dydt}{t^{n+1}} )^{\frac{1}{2}} 
	 \lesssim (\int_0^\infty \int_{B(x,5t)} |F(t,y)|^2 \,\frac{dydt}{t^{n+1}} )^{\frac{1}{2}}.
\end{align*}
Taking the $L^p$ norm with respect to $x \in \R^n$ then yields the assertion.
\end{proof}

The corresponding estimate for the principal part $\gamma_t(x) A_t {P_{t}}^{\tilde{N}}$ relies on the following factorisation result  for tent spaces:

\begin{Theorem}[\cite{CohnVerbitsky}, Theorem 1.1] \label{Thm:factorisation-tent}
Let $p,q \in (1,\infty)$. If $F \in T^{p,\infty}(\R^{n+1}_+;\C^N)$ and $G \in T^{\infty,q}(\R^{n+1}_+;\C^N)$, then $FG \in T^{p,q}(\R^{n+1}_+;\C^N)$ and
\[
	\|F \cdot G\|_{T^{p,q}} \leq C \|F\|_{T^{p,\infty}} \|G\|_{T^{\infty,q}},
\]
with a constant $C$ which is independent of $F$ and $G$.
\end{Theorem}

This plays the role of the $L^p$ vertical square function version of Carleson's inequality proven in \cite[Lemma 8.1]{HMcP1}. Note that this conical version is substantially simpler than its vertical counterpart.\\

We also use the following conical maximal function estimate for operators with $L^q$-$L^q$ off-diagonal bounds.

\begin{Lemma} \label{lemma:nontang-max-est}
Let $q \in [1,2]$ and $p \in (1,\infty)$ with $q<p$. 
Let $\{T_t\}_{t>0}$ be a family of operators in   $\calL(L^{q}(R^{n};\C^{N}))$,
such that for all $t>0$, $T_{t}$ has
$L^q$-$L^q$ off-diagonal bounds of order $N'>\frac{n}{q}$.  Then
\[
	\|(t,x)\mapsto A_t T_t u(x)\|_{T^{p,\infty}} \leq C_p \|u\|_{p}
	 \quad \forall u \in L^q(\R^n;\C^N) \cap L^p(\R^n;\C^N). 
\]
\end{Lemma}

\begin{proof}
Let $u \in  L^q(\R^n;\C^N)  \cap L^p(\R^n;\C^N)$. Using H\"older's inequality and $L^q$-$L^q$ off-diagonal bounds for $T_t$, we obtain, given $x \in \R^n$, the pointwise estimate

\begin{align*}
	\sup_{(y,t) \in \Gamma(x)} |A_t T_t u(y)|
		& \lesssim \sup_{(y,t) \in \Gamma(x)} (t^{-n} \int_{Q_{y,t}} |T_t u(z)| \,dz)\\
		& \leq \sup_{(y,t) \in \Gamma(x)} \sum_{j=0}^\infty (t^{-n} \int_{Q_{y,t}} |T_t\Eins_{S_j(Q_{y,t})} u(z)|^q\,dz)^{1/q}  \\ 
				& 
 \lesssim \sup_{(y,t) \in \Gamma(x)} \sum_{j=0}^\infty t^{-\frac{n}{q}} 2^{-jN'} \|u\|_{L^q(2^j Q_{y,t})} \\
		& \lesssim \sup_{(y,t) \in \Gamma(x)} \sum_{j=0}^\infty 2^{-j(N'-\frac{n}{q})} (2^jt)^{-\frac{n}{q}} \|u\|_{L^q(2^j Q_{y,t})}\\ 
		& \lesssim \underset{R>0}{\sup} (\frac{1}{|B(x,R)|} \int \limits _{B(x,R)} |u(z)|^q dz)^{\frac{1}{q}} =: \calM_q u(x).
\end{align*}

Since $q<p$, the boundedness of the Hardy-Littlewood maximal operator in $L^{\frac{p}{q}}$ implies that the maximal operator $\calM_q$ is bounded in $L^p(\R^n;\C^N)$.
 Thus, 
\[
	\|(t,x) \mapsto A_tT_tu(x)\|_{T^{p,\infty}} 
				\lesssim \|\calM_q u\|_{p} 
		\lesssim \|u\|_{p}.
\]
\end{proof}

The estimate for the principal part is a direct consequence of the results above, together with the Carleson measure estimate for $|\gamma_t(x)|^{2}\frac{dtdx}{t}$.

\begin{Prop} \label{Prop:Carleson-part}
Suppose $\Pi_B$ is a perturbed Hodge-Dirac operator. 
Let $(t,x) \mapsto \gamma_t(x)$ be defined as in \eqref{def:gamma_t}. Suppose $p \in (1,\infty)$. Then
\[
	\|(t,x)\mapsto\gamma_t(x)A_t{P_{t}}^{\tilde{N}}u(x)\|_{T^{p,2}} \leq C_p \|u\|_{p} \qquad \forall u \in \overline{\calR_p(\Pi)}.
\]	
\end{Prop}

\begin{proof}
Since $A_t^2=A_t$, Theorem \ref{Thm:factorisation-tent} yields 
\begin{align*}
	\|(t,x)\mapsto\gamma_t(x)A_t{P_{t}}^{\tilde{N}}u(x)\|_{T^{p,2}}
	\lesssim \|(t,x) \mapsto A_t{P_{t}}^{\tilde{N}} u(x)\|_{T^{p,\infty}} \cdot \|(t,x) \mapsto \gamma_t(x)\|_{T^{\infty,2}}.
\end{align*}
The boundedness of the last factor is shown in Proposition \ref{lem:reverse} and noted in Remark \ref{CM}, as a consequence of the $L^2$ theory for $\Pi_B$ established in \cite{AKM}, cf. Theorem \ref{thm:p=2}. The first factor is bounded by a constant times $\|u\|_{p}$ as an application of Lemma \ref{lemma:nontang-max-est}:  take $T_t:={P_{t}}^{\tilde{N}}$ and notice that, 
 for all $t>0$, and every $q \in (1,2]$,
${P_{t}}^{\tilde{N}}$ satisfies  $L^q$-$L^q$ off-diagonal bounds of every order  by Proposition \ref{lem:hardyPi}.  This completes the proof of \eqref{reduction} and hence of Theorem \ref{thm:low}.
\end{proof}

\section{High frequency estimates for $p \in (2_{*},2]$}
\label{sec:high1}

In this section, we give a proof of Theorem \ref{thm:high} for the case $2_\ast < p_H <2$. In particular, this gives a proof for $n \in \{1,2\}$, a case we have to exclude in Section \ref{sec:high2} below for technical reasons. The proof is similar to the corresponding proof in $L^2$ in \cite{AKM}, and is less technically involved than the case $p_H \leq 2_\ast$ considered in the next sections.

\begin{Prop} \label{lem:simpleHF}
Suppose $\Pi_B$ is a perturbed Hodge-Dirac operator.
Suppose $M \in \N$ and $p \in (2_\ast,2]$. Then
\begin{align*}
\|(t,x) \mapsto (Q_{t}^{B})^{M}(I-{P_{t}}^{\tilde{N}})u(x)\|_{T^{p,2}} &\leq C_p \|u\|_{p} \quad \forall u \in \overline{\calR_{2}(\Gamma)} \cap L^{p}(\R^n;\C^{N}).
\end{align*} 
\end{Prop}

\begin{proof}[Proof of Proposition \ref{lem:simpleHF}]
Let $p \in (2_\ast,2]$, $M \in \N$,
and $u \in \calR_{2}(\Gamma) \cap L^{p}(\R^n;\C^{N})$. Lemma \ref{tent-boundedness} and Lemma \ref{Lemma:off-diag-est} below yield
\begin{align*}
	\|(t,x) \mapsto (Q_{t}^{B})^{M}(I-{P_{t}}^{\tilde{N}})u(x)\|_{T^{p,2}} 
	&= \|(t,x) \mapsto Q_{t}^{B}t\Gamma (\sum \limits _{k=0} ^{\tilde{N}-1} {P_{t}}^{k})Q_t u(x)\|_{T^{p,2}} \\
	&\lesssim \|(t,x) \mapsto Q_t u(x)\|_{T^{p,2}}.
\end{align*}
The assertion then follows from Proposition \ref{lem:hardyPi}.
\end{proof}

We use the following lemma in the proof of Proposition \ref{lem:simpleHF} above.
The result and its proof are a slight modification of \cite[Proposition 5.2]{AKM}.

\begin{Lemma} \label{Lemma:off-diag-est}
The families $\{t\Gamma_B^\ast Q_t^B \;;\; t \in \R\}$ 
and  $\{t\Gamma Q_t^B \;;\; t \in \R\}$ have $L^2$-$L^2$ off-diagonal bounds of every order.
\end{Lemma}

\begin{proof}
We prove the result for $\{t\Gamma_B^\ast Q_t^B \;;\; t \in \R\}$. The result for $\{t\Gamma Q_t^B \;;\; t \in \R\}$ then follows, given that for all $t\in \R$,
$$
t\Gamma Q_t^B=(I-P_{t}^{B})-t\Gamma_B^\ast  Q_t^B,
$$
and $\{P_t^B \;;\; t \in \R\}$ has $L^2$-$L^2$ off-diagonal bounds of every order by Theorem \ref{thm:p=2}.
By Theorem \ref{thm:p=2}, we also have that the family $\{t\Gamma_B^\ast Q_t^B \;;\; t \in \R\}=\{P_{\overline{R(\Gamma_B^\ast)}}(I-P_t^B) \;;\; t \in \R\}$ is uniformly bounded in $L^2$.
Let $E,F \subset \R^{n}$ be two Borel sets, $u \in L^2(\R^n;\C^N)$, and $t \in \R$. 
As in \cite[Proposition 5.2]{AKM}, let $\eta$ be a Lipschitz function supported in 
$\tilde{E} = \{ x\in \R^{n} \;;\; \dist(x,E)<\frac{1}{2}\dist(x,F)\}$, constantly equal to $1$ on $E$, and such that
$\|\nabla \eta\|_{\infty} \leq \frac{4}{\dist(E,F)}$. We have the following:
\begin{align*}
	\|t\Gamma_B^\ast Q_t^B u\|_{L^2(E)} 
	\leq \|\eta t\Gamma_B^\ast Q_t^B  u\|_{2} 
	\leq \|[\eta I,t\Gamma_B^\ast]Q_t^Bu\|_{2} 
+ \|t\Gamma_B^\ast \eta Q_t^Bu\|_{2} .
\end{align*}
To estimate the first term, we use that $[\eta I,t\Gamma_B^\ast]=tB_1[\eta I,\Gamma^\ast]B_2$ is a multiplication operator
with norm bounded by $t\|\nabla \eta\|_{\infty}$, together with the off-diagonal bounds for $Q_t^B$.
For the second term, observe that, since $\Pi_{B}$ Hodge decomposes $L^2$ according to Proposition \ref{prop:Hodge}, we have that 
\[
	\|t\Gamma_B^\ast \eta Q_t^Bu\|_{2} 
		\lesssim \|t\Pi_B \eta Q_t^B u\|_{2} 
		\leq \|[\eta I,t\Pi_B]Q_t^Bu\|_{2} 
			+ \|\eta t\Pi_B Q_t^Bu\|_{2} .
\]
Here, we use that the commutator in the first part of the sum is again a multiplication operator. For the second part, we use that $t\Pi_BQ_t^B =I-P_t^B$, which satisfies $L^2$-$L^2$ off-diagonal bounds.
\end{proof}

\section{$L^p$-$L^2$ off-diagonal bounds}
\label{sec:high2}

We assume $n \geq 3$ throughout this section.\\

 In this section, we show how to deduce $L^{p_\ast}$-$L^p$ bounds from $L^p$ bisectoriality via a Sobolev inequality.
  We then use this result to establish appropriate $\dot{W}^{1,p}-L^{2}$ off-diagonal bounds on balls, when $p_{*}>p_{H}$.

\begin{Lemma} \label{lem:Lp-LqOD}
Suppose $n \geq 3$.
Suppose $\Pi_B$ is a perturbed Hodge-Dirac operator. Suppose $p \in (1,2]$ with $p_\ast> 1$, and assume that $p>p_{H}$ and that $\Pi_{B}$ is bisectorial in $L^p(\R^n;\C^N)$. Then   \begin{align} \label{eq:uniformLq}
	\sup_{t \in \R} \|tR_{t}^{B}u\|_{p} &\lesssim \|u\|_{{p_{*}}}\quad\forall u\in \overline{\calR_{p_{*}}(\Gamma)}.
\end{align}
Moreover, if $(\Pi_{B}(p_{*}))$ holds, then \begin{align}	
	\label{eq:uniformLq2}
	\sup_{t \in \R}  \|tR_{t}^{B}u\|_{p} &\lesssim \|u\|_{{p_{*}}}\quad \forall u\in \overline{\calR_{p_{*}}(\Gamma^*_B)}.
\end{align}
\end{Lemma}

\begin{proof}

To prove \eqref{eq:uniformLq}, we use the potential map $S_\Gamma: \overline{\calR_{p_*}(\Gamma)} \to \dot W^{1,p_*}(\R^n;\C^N) \hookrightarrow L^p(\R^n;\C^N)$
 defined in Proposition \ref{prop:pi}(7). Then,  for all $u\in \overline{\calR_{p_*}(\Gamma)}$, 
\begin{align*} \|tR_{t}^{B}u\|_{p} &= \|t(P_{t}^{B}-iQ_t^B)\Gamma S_\Gamma u\|_{p} 
\leq \|t\Gamma P_{t}^{B} S_\Gamma u\|_{p} + \|t\Gamma^*_BQ_{t}^{B} S_\Gamma u\|_{p}\\ 
& \lesssim \|Q_{t}^{B} S_\Gamma u\|_{p} + \|(I-P_{t}^{B}) S_\Gamma u\|_{p} \quad\text{(since } p>p_{H})\\
&\lesssim \|S_\Gamma u\|_p \quad\text{(using bisectoriality)}\\
&\lesssim \|S_\Gamma u\|_{\dot W^{1,p_*}} \lesssim \|\Gamma S_\Gamma u\|_{p_*} = \|u\|_{p_*}
\end{align*} as claimed.

By $(\Pi_{B}(p_{*}))$, \eqref{eq:uniformLq2} follows from
\begin{equation*}	
	\sup_{t \in \R}  \|tR_{t}^{B}B_1v\|_{p} \lesssim \|v\|_{{p_{*}}}  \quad \forall v\in \overline{\calR_{p_{*}}(\Gamma^*)}.
\end{equation*}
 This is proven in the same way as \eqref{eq:uniformLq}, using $\underline{\Pi_B}$ instead of $\Pi_{B}$, and remarking that 
$$tR_{t}^{B}B_1v = t(P_{t}^{B}-iQ_{t}^{B})B_1\Gamma^* S_{\Gamma^*}v
=  tB_1\underline{P_{t}^{B}} \Gamma^* S_{\Gamma^*}v
-itB_2\underline{Q_{t}^{B}} \Gamma^* S_{\Gamma^*}v \ \text{(by \eqref{eq:swap-bar})}.
$$
\end{proof}

We use the following induction argument in which  $\begin{cases} p^{*(k)}=(p^{*(k-1)})^{*} \quad \forall k \in \N, \\ p^{*(0)}=p, \end{cases}$ and $M_{s}(p)$ is the smallest natural number such that
$p^{*(M_{s}(p))} \geq 2$. A simple induction argument gives  $p^{*(M)}=\frac{np}{n-pM}$ for all $M \in \N$, so that $M_{s}(p)\ge n(\frac1p-\frac12)$.

\begin{Prop} \label{prop:induction}
Suppose $n \geq 3$.
Suppose $\Pi_B$ is a perturbed Hodge-Dirac operator. Suppose $p \in (p_H,2]$ with $p_\ast> 1$. Assume that $\Pi_B$ is bisectorial in $L^p(\R^n;\C^N)$. Assume further that  for all $M \in \N$ such that $M\geq M_{s}(p)$ and all $r \in (p,2]$ (with $r=2$ if $p=2$), 
 $\{|t|^{n(\frac1r-\frac12)}(R_t^B)^M \mathbb{P}_{\overline{\calR_{r}(\Pi_{B})}} \;;\; t \in \R^{*}\}$  is bounded from $L^r(\R^n;\C^N)$ to $L^2(\R^n;\C^N)$ uniformly in $t$. 
\begin{enumerate} 
\item Given $q \in (p_\ast,2]$ and $M \in \N$ such that $M\geq M_{s}(p)$, we have $$  
	\sup_{t \in \R} \||t|^{n(\frac{1}{q}-\frac{1}{2})}(
	R_{t}^{B})^M u\|_{2} \lesssim \|u\|_{{q}}\quad \forall u \in \overline{\calR_{q}(\Gamma)}.$$ 	
Moreover, assuming  $(\Pi_{B}(p_{*}))$ holds when $q<p_{H}$, we have  
$$
	\sup_{t\in \R}  \||t|^{n(\frac{1}{q}-\frac{1}{2})}(R_{t}^{B})^M u\|_{2} \lesssim \|u\|_{{q}}\quad \forall u \in  \overline{\calR_{2}(\Gamma_B^\ast)} \cap L^{q}(\R^{n};\C^{N}).
$$
\item Given $q \in (\max\{p_H,p_\ast\},2]$ and $M \in \N$ such that $M\geq M_{s}(p)$, we have 
$$
	\sup_{t \in \R} \||t|^{n(\frac{1}{q}-\frac{1}{2})}(R_{t}^{B})^M \mathbb{P}_{\overline{\calR_{q}(\Pi_{B})}}u\|_{2} \lesssim \|u\|_{{q}}\quad \forall u \in L^q(\R^n;\C^N).
$$
Moreover $\{(R_t^B)^M \Gamma \;;\; t \in \R^{*}\}$ and $\{(R_t^B)^M B_{1}\Gamma^{*} \;;\; t \in \R^{*}\}$ have $\dot{W}^{1,q}$-$L^2$ off-diagonal bounds of every order on balls.
\end{enumerate}
\end{Prop}

\begin{proof}
(1) 
By assumption, we have for all $r \in (p,2]$ and all $M \in \N$ such that $M\geq M_{s}(p)$ 
$$
	\sup_{t \in \R} \| |t|^{n(\frac{1}{r}-\frac{1}{2})}(R_{t}^{B})^{M} \mathbb{P}_{\overline{\calR_{r}(\Pi_{B})}}u\|_{2} \lesssim \|u\|_{{r}}\quad \forall u \in L^r(\R^n;\C^N).
$$
Combining this with Lemma \ref{lem:Lp-LqOD} gives the assertion for $q=r_\ast$.\\
(2) For $q>p_H$, $\Pi_B$ Hodge decomposes $L^q(\R^n;\C^N)$ by assumption. We therefore get the first estimate in (2) as a direct consequence of (1). \\
Since $\|\Gamma w\|_{q}+\|B_{1}\Gamma^{*}w\|_{q} \lesssim \|\nabla w\|_{q}$ for all balls $B$ and all $w \in \dot{W}^{1,q}_{\ovB}$, with constants independent of the ball $B$, we have that 
$\underset{t \in \R^{*}}{\sup}\  \underset{B=B(x,|t|)}{\sup} \,\||t|^{n(\frac{1}{q}-\frac{1}{2})}(R_t^B)^M \Gamma\|_{\calL(\dot{W}^{1,q}_{\ovB},L^2)}<\infty$ 
and that 
$\underset{t \in \R^{*}}{\sup}\ \underset{B=B(x,|t|)}{\sup} \,\||t|^{n(\frac{1}{q}-\frac{1}{2})}(R_t^B)^M B_{1}\Gamma^{*}\|_{\calL(\dot{W}^{1,q}_{\ovB},L^2)}<\infty$. 
Moreover $\{(R_t^B)^M \Gamma \;;\; t \in \R^{*}\}$ and $\{(R_t^B)^M B_{1}\Gamma^{*} \;;\; t \in \R^{*}\}$ have $\dot{W}^{1,2}$-$L^2$ off-diagonal bounds of every order on balls, since
$\{(R_t^B)^M \;;\; t \in \R^{*}\}$ has $L^2$-$L^2$ off-diagonal bounds of every order. The result then follows from Lemma \ref{lem:ODcomp}(2). 

\end{proof}

\section{Estimating conical square functions by vertical square functions}
\label{sec:tech2}

While vertical and conical square functions look similar, the conical square functions are applied quite differently here compared with the way the vertical square functions are used in \cite{HMcP1}.
Nevertheless, as is the case classically (see e.g.  \cite{Steinbook}), there are relationships between conical and vertical square functions, as Auscher, Hofmann, and Martell have already pointed out in \cite{AHM}.  Here we prove a new comparison theorem that exploits  $\dot{W}^{1,p}$-$L^2$ off-diagonal bounds on balls.  The proof is based on some unpublished work of Auscher, Duong and the second author, where a similar result was obtained for operators with pointwise Gaussian bounds.

\begin{Theorem}
\label{thm:conicalVSvertical}  Let $p \in (1,2]$, and  $M \in 2\N$  with $M\geq 10n$. 
Let $\Pi_{B}$ be a perturbed Hodge Dirac operator, and assume that $(\Pi_{B}(p))$ holds.
Assume that $\{(R_t^B)^{\frac{M}{2}-2} \Gamma \;;\; t \in \R^{*}\}$ and $\{(R_t^B)^{\frac{M}{2}-2} B_{1}\Gamma^{*} \;;\; t \in \R^{*}\}$ have $\dot{W}^{1,p}$-$L^2$ off-diagonal bounds of every order on balls.
Then 
\begin{equation} \label{vertcon}
\|(t,x)\mapsto (Q_t^B)^{M}G(t,.)(x)\|_{T^{p,2}} 
\leq C_p \|(\int \limits _{0} ^{\infty} |G(t,.)|^{2} \frac{dt}{t})^{\frac{1}{2}}\|_{p}, 
\end{equation}

for all $G \in  L^p(\R^n; L^{2}(\R_{+},\frac{dt}{t};\C^{N}))$ such that either $G(t,.) \in \overline{\calR_{p}(\Gamma)}$ or  $G(t,.) \in \overline{\calR_{p}(\Gamma^{*}_{B})}$ for almost every $t>0$. If, in addition, $p>p_{H}$, then the result holds for all $G \in  L^p(\R^n; L^{2}(\R_{+},\frac{dt}{t};\C^{N}))$.
\end{Theorem}

\begin{proof}
We use a variant of the Blunck-Kunstmann extrapolation method established by Auscher in \cite[Theorem 1.1]{AuscherMemoirs}, combined with Auscher's Calder\'on-Zygmund decomposition for Sobolev spaces \cite[Lemma 4.12]{AuscherMemoirs}. As pointed out in \cite{AuscherMemoirs} both results hold in Hilbert space valued $L^p$ spaces. 
Let us therefore consider $H_{1} = L^{2}(\R_{+},\frac{dt}{t};\C^{N})$, $H_{2} = L^{2}(\R_{+}\times \R^{n},\frac{dt dx}{t^{n+1}};\C^{N})$, and an operator $T$ defined by
$$
T(G):x \mapsto [(t,y)\mapsto \Eins_{B(y,t)}(x)(Q_t^B)^{M} G(t,.)(y)],
$$
for $G \in L^{2}(\R^{n};H_{1})$.  Since $\Pi_B$ is bisectorial in $L^2$, $T$ is a bounded operator from $L^{2}(\R^{n};H_{1})$ to $L^{2}(\R^{n};H_{2})$.  
Indeed $\|T(G)\|_{L^{2}(H_{2})} \sim (\int \limits _{0} ^{\infty} \|((Q_t^B)^{M}G(t,.)\|_{2} ^{2} \frac{dt}{t})^{\frac{1}{2}} \lesssim \|G\|_{L^{2}(H_{1})}$.
 
 We are going to show the weak type estimate: For all $\alpha>0$,
\begin{equation} \label{weakpp}
	|\{x \in \R^n:\,\|TG(x)\|_{H_2} >\alpha\}| \lesssim \alpha^{-p}\|G\|_{p}^{p}.
\end{equation}
The strong type estimate will then follow by interpolation for every $q \in (p,2]$.

In proving bounds on $L^p(\R^n; H_1)$, we use the fact that the lifting 
$S\mapsto \tilde S$ from $\calL(L^p(\R^n;\C^N))$ to $\calL(L^p(\R^n;H_1))$ is bounded, where $(\tilde S G)(t,.) = S(G(t,. ))$
 for almost every $t>0$. This is a classical consequence of Khintchine-Kahane's inequalities (see \cite[Section 2]{kuw}) for discrete square functions, that extends to continuous square functions using a decomposition in an orthonormal basis of $H_{1}$. When $p>p_{H}$, we can
apply this to the Hodge projections, and obtain the Hodge decomposition  in $L^p(\R^n;H_1)$: 
$$ G = \tilde{\mathbb{P}}_{N_{p}(\Pi_{B})}G + \tilde{\mathbb{P}}_{\overline{\calR_{p}(\Gamma)}}G + \tilde{\mathbb{P}}_{\overline{\calR_{p}(\Gamma^{*}_{B})}}G =:G_0+G_1+G_2
$$
with $\|G\|_p\approx \|G_0\|_p+\|G_1\|_p + \|G_2\|_p$.

Since $T(G_{0})=0$ for all $G_{0}$ such that $G_{0}(t,.) \in N_{p}(\Pi_{B})$ for almost every $t>0$, we only have to prove the result for $G_{1}(t,.) \in \overline{\calR_{p}(\Gamma)}$ or  $G_{2}(t,.) \in \overline{\calR_{p}(\Gamma^{*}_{B})}$ for almost every $t>0$. 
In the first case, one can use the lifted potential map $\tilde {S}_\Gamma:\overline{\calR_p(\Gamma)}\to W^{1,p}(\R^n;H_1)$ to write $G_1=\Gamma \tilde{S}_\Gamma G_1 = \Gamma f$ where $f= \tilde{S}_\Gamma G_1 \in W^{1,p}(\R^n;H_1)$ with $\|\nabla \otimes f\|_p\approx\|G_1\|_p\leq \|G\|_p$.

We thus only have  to prove  the following  weak type estimate: For all $\alpha>0$,
\begin{equation} \label{weakpp1}
	|\{x \in \R^n:\,\|TG_1(x)\|_{H_2} >\alpha\}| \lesssim \alpha^{-p}\|\nabla\otimes f\|_{p}^{p}.
\end{equation}
 for $G_{1}=\Gamma f$ and $f\in \dot W^{1,p}(\R^n; H_1)$.

According to \cite[Lemma 4.12]{AuscherMemoirs}, given $\alpha>0$, there exists  a collection of cubes $(Q_j)$ and functions $g,b_j$ such that  $f=g+\sum_j b_j$ with $\supp b_j \subseteq Q_j$ and
\begin{equation} \label{CZdec}
	G_1=\Gamma f=\Gamma g + \sum_j \Gamma b_j,
\end{equation}
satisfying
\begin{align} \label{CZeq1}
	&g \in \dot{W}^{1,2}(H_1) \cap \dot{W}^{1,p}(H_1), \qquad 
	\|\nabla g\|_2 \lesssim \alpha^{2-p}\|\nabla f\|_{p}^{p};\\
	\label{CZeq2}
	& b_j \in W^{1,p}_0(Q_j,H_1), \qquad \|\nabla b_j\|_{p} \lesssim \alpha |Q_j|^{1/{p}};\\
	\label{CZeq3}
	&\sum_j |Q_j| \lesssim \alpha^{-p} \|\nabla f\|_{p}^{p};\\
	&\sum_j \Eins_{Q_j} \leq N.
\end{align}

Applying this decomposition, we have 
\begin{align*} |\{x &\in \R^n:\,\|TG_1(x)\|_{H_2} >\alpha\}| \lesssim  |\{x \in \R^n:\,\|T\Gamma g(x)\|_{H_2} >\alpha/3\}| +\\
&|\{x \in \R^n:\,\|\sum_j T(I-A_{r_j})\Gamma b_j(x)\|_{H_2} >\frac{\alpha}{3}\}| +
|\{x \in \R^n:\,\|\sum_j T(A_{r_j})\Gamma b_j(x)\|_{H_2} >\frac{\alpha}{3}\}|.
\end{align*} 

By \eqref{CZeq1} and the $L^2$ estimate, we have
$$
	|\{x \in \R^n:\,\|T\Gamma g(x)\|_{H_2} >\alpha/3\}| 
	\lesssim \alpha^{-2}\|\Gamma g\|_{2}^{2}
	\lesssim \alpha^{-p}\|\nabla f\|_{p}^{p}
$$
so we turn our attention to the functions $b_j$.

 Define 
\begin{align*}
	\phi_{t}(\Pi_B) = I-(I-(P_t^B)^{M})^M \qquad \forall t>0
\end{align*}
and set  $A_{r_j}=\phi_{r_j}(\Pi_B)$ where 
$r_j$ is the diameter of $Q_j$.

We will establish, in Lemma \ref{prop:CZoff-diag} below, the following two off-diagonal bound conditions, similar to those of \cite[Theorem 1.1]{AuscherMemoirs}. Let $B\subset \R^{n}$ be a ball of radius $r>0$, $k \in \N$, and 
$f\in \dot W^{1,p}_{\ovB}(\R^{n};H_{1})$. 
The two conditions are, in our setting, 
\begin{align} \label{extrapol-cond1}
	(\frac{1}{|2^{k+1}B|}  \int \limits _{S_{k}(B)} \|\phi_{r}(\Pi_B)\Gamma f(.,y)\|_{H_{1}} ^{2} dy)^{\frac{1}{2}} 
	& \lesssim 
	g(j) (|B|^{-1}\int \limits _{B} \|\nabla f(.,y)\|_{H_{1}}^{p}dy)^{\frac{1}{p}}, 
	\intertext{for $k \geq 1$, and} \label{extrapol-cond2}
(\frac{1}{|2^{k+1}B|}  \int \limits _{S_{k}(B)} \|T(I-\phi_{r}(\Pi_B))\Gamma f(.,y)\|_{H_{2}} ^{2} dy)^{\frac{1}{2}}
&\lesssim g(k) (|B|^{-1}\int \limits _{B} \|\nabla f(.,y)\|_{H_{1}}^{p}dy)^{\frac{1}{p}}
\end{align}
for $k \geq 2$, with $g(k)$ satisfying $\sum_{j \in \N} g(k) 2^{nk}<\infty$.\\

Let us assume these for the moment.
Then
\begin{align*}
	&|\{x \in \R^n:\,\|\sum_j T(I-A_{r_j})\Gamma b_j(x)\|_{H_2} >\alpha/3\}|\\
	& \qquad  \lesssim \sum_j |Q_j|
	+ \alpha^{-2} \|\sum_j \Eins_{(4Q_j)^c}T(I-A_{r_j})\Gamma b_j\|_2^2,
\end{align*}
and the first term is bounded by  $\alpha^{-p} \|\nabla f\|_{p}^{p} $  according to \eqref{CZeq3}.
To bound the second term, let $u \in L^{2}(H_2)$ with $\|u\|_{2}=1$. Use H\"older's inequality in the first step, \eqref{extrapol-cond2} in the second step, $\sum_k g(k) 2^{kn}<\infty$ and \eqref{CZeq2} in the third step to obtain 
\begin{align*}
	&\int_{\R^n}|\skp{u(x),\sum_j \Eins_{(4Q_j)^c}(x) T(I-A_{r_j})\Gamma b_j(x)}|\,dx\\
	&\leq \sum_j \sum_{k=2}^\infty  
	\left(\int_{S_k(Q_j)} \|u(x)\|_{H_2}^{2}\,dx\right)^{1/2} 
	\left(\int_{S_k(Q_j)} \|T(I-A_{r_j})\Gamma b_j(x)\|_{H_2}^2\,dx\right)^{1/2}\\
	&\lesssim \sum_j \sum_{k=2}^\infty |2^kQ_j|^{1/2}
	\inf_{y \in Q_j} \calM_{2}(\|u\|_{H_2})(y) g(k) |2^kQ_j|^{1/2} |Q_j|^{-1/{p}} \|\Gamma b_j\|_{p}\\
	&\lesssim \alpha  \sum_j \int_{Q_j} \calM_{2}(\|u\|_{H_2})(y)\,dy 
	\lesssim \alpha  \int_{\bigcup_j Q_j} \calM_{2}(\|u\|_{H_2})(y)\,dy.
\end{align*}
By Kolmogorov's lemma and \eqref{CZeq3}, the last line is bounded by
\begin{align*}
	\alpha |\bigcup_j Q_j|^{1/2} \|u\|_{2}
	\lesssim \alpha (\sum_j |Q_j|)^{1/2} \lesssim \alpha^{1-p/2} \|\nabla f\|_{p}^{p/2},
\end{align*}
which gives the desired estimate.\\
It remains to consider the last part
\begin{align*}
	|\{x \in \R^n:\,\|\sum_j TA_{r_j}\Gamma b_j(x)\|_{H_2} >\alpha/3\}|
	\lesssim \alpha^{-2} \|\sum_j TA_{r_j}\Gamma b_j\|_{L^2(H_2)}^2.
\end{align*}
We again dualise with $u \in L^{2}(H_2)$, $\|u\|_{2}=1$, now using  $L^{p}$-$L^2$ off-diagonal bounds for $A_{r_j}$ on $R_p(\Gamma)$. With similar arguments as before, but now including the on-diagonal term, using \eqref{extrapol-cond1} and using that $T \in \calL(L^2(H_1),L^2(H_2))$ by assumption,  one obtains
\begin{align*}
	|\skp{u,\sum_j TA_{r_j}\Gamma b_j}| 
	&\leq \sum_j \sum_{k=0}^\infty \left(\int_{S_k(Q_j)}\|Tu(x)\|_{H_1}^{2}\,dx\right)^{1/2} \left(\int_{S_k(Q_j)}\|A_{r_j}\Gamma b_j(x)\|_{H_1}^2\,dx\right)^{1/2}\\
	&\lesssim \alpha  \calM_{2}(\|Tu\|_{H_1})(y)\,dy
	\lesssim \alpha \|Tu\|_{2} (\sum_j |Q_j|)^{1/2}
	\lesssim  \alpha^{1-p/2} \|\nabla f\|_{p}^{p/2}.
\end{align*}
This establishes \eqref{weakpp1} for  $G_1 = \Gamma f$.

 In the case where $G(t,.) \in \overline{\calR_{p}(\Gamma^{*}_{B})}$ for almost every $t>0$, we 
replace $(Q_t^B)^Mu$ for $u\in\overline{\calR_p(\Gamma^*_B)}$,  by $(\underline{Q_t^B})^Mv$  for $v\in\overline{\calR_p(\Gamma^*)}$, as in the proof of Corollary \ref{cor:hardy}, and then proceed as above. 
\end{proof}

\begin{Lemma} \label{prop:CZoff-diag}
Suppose $p \in (p_H,2]$ and  $M \in 2\N$  with $M>n+4$.\\  Assume that $\{(R_t^B)^{\frac{M}{2}-1} \Gamma \;;\; t \in \R^{*}\}$ and $\{(R_t^B)^{\frac{M}{2}-1} B_{1}\Gamma^{*} \;;\; t \in \R^{*}\}$ have $\dot{W}^{1,p}$-$L^2$ off-diagonal bounds of every order on balls.
 Then \eqref{extrapol-cond1} and \eqref{extrapol-cond2} hold
 for all balls $B\subset \R^{n}$  of radius $r>0$, and all $f\in \dot{W}^{1,p}_{\ovB}(\R^n;H_1)$.
\end{Lemma}

\begin{proof}

We write $\phi(z)=\tilde{\phi}(z)(1+z^{2})^{-1}$ for $z \in S_{\theta}$ and $\theta \in (0,\frac{\pi}{2})$, and notice that $\phi \in H^{\infty}(S_{\theta})$.
Moreover $\phi$ is a sum and product of functions of the form $z \mapsto (1\pm iz)^{-1}$. Therefore $\{\phi(r\Pi_{B}) \;;\; r>0\}$ has $L^2$-$L^2$ off-diagonal bounds of every order. Combining this with the $\dot{W}^{1,p}$-$L^2$ off-diagonal bounds assumption and using Lemma \ref{lem:ODcomp}, we have that
$\{\phi(r\Pi_{B})\Gamma \;;\; r>0\}$ has $\dot{W}^{1,p}$-$L^2$ off-diagonal bounds of every order on balls. This gives the following:
\begin{equation*}
\begin{split}
(\frac{1}{|2^{j+1}B|} & \int \limits _{S_{j}(B)} \|\phi_{r}(\Pi_B)\Gamma f(.,y)\|_{H_{1}} ^{2} dy)^{\frac{1}{2}}
\approx 2^{-j\frac{n}{2}}r^{-\frac{n}{2}}(\int \limits _{0} ^{\infty} \int \limits _{S_{j}(B)} |\phi_{r}(\Pi_B)\Gamma f(t,y)|^{2} \frac{dydt}{t})^{\frac{1}{2}}\\
&\lesssim 2^{-j(M+\frac{n}{2})} r^{-\frac{n}{p}}(\int \limits _{0} ^{\infty} \|\nabla f(t,y)\|^{2} _{p} \frac{dt}{t})^{\frac{1}{2}}
\lesssim 
 2^{-j(M+\frac{n}{2})} |B|^{-\frac{1}{p}} \|\int \limits _{B} |\nabla f(t,y)|^{p} dy\|_{\frac{2}{p}} ^{\frac{1}{p}}\\ 
 & \lesssim
  2^{-j(M+\frac{n}{2})} (|B|^{-1}\int \limits _{B} (\int \limits _{0} ^{\infty} |\nabla f(t,y)|^{2} \frac{dt}{t})^{\frac{p}{2}}dy)^{\frac{1}{p}}
=  2^{-j(M+\frac{n}{2})} (|B|^{-1}\int \limits _{B} \|\nabla f(.,y)\|_{H_{1}}^{p}dy)^{\frac{1}{p}}.
\end{split}
\end{equation*}

This establishes \eqref{extrapol-cond1}. 

The proof will thus be complete once we have established \eqref{extrapol-cond2}.
To do so, we first use the straightforward integration lemma \cite[Lemma 1]{cms} (an application of Fubini's theorem), and obtain that for $j \geq 2$ 
\begin{equation*}
\begin{split}
(\frac{1}{|2^{j+1}B|}  &\int \limits _{S_{j}(B)} \|T(I-\phi_{r}(\Pi_B))\Gamma f(.,y)\|_{H_{2}} ^{2} dy)^{\frac{1}{2}}\\
&\lesssim
(\frac{1}{|2^{j+1}B|}  \int \limits_0^\infty \int \limits _{\mathcal{R}(S_{j}(B))} |(Q_t^B)^M(I-\phi_{r}(\Pi_B))\Gamma f(t,y)|^2 \frac{dydt}{t})^{\frac{1}{2}}\\
& \leq
(\frac{1}{|2^{j+1}B|}  \int \limits _{0} ^{\infty} \int \limits _{(2^{j-1}B)^{c}} |(Q_t^B)^M(I-\phi_{r}(\Pi_B))\Gamma f(t,y)|^{2} \frac{dydt}{t})^{\frac{1}{2}}\\
&\quad + (\frac{1}{|2^{j+1}B|}  \int \limits _{2^{j-2}r}^{\infty} \int \limits _{2^{j-1}B} |(Q_t^B)^M(I-\phi_{r}(\Pi_B))\Gamma f(t,y)|^{2} \frac{dydt}{t})^{\frac{1}{2}}
=:I_1 + I_{2},
\end{split}
\end{equation*}
where $\mathcal{R}(S_{j}(B)) = \bigcup \limits _{x \in S_{j}(B)} \{(t,y) \in \R_{+}\times \R^{n} \;;\; |y-x|\leq t\}$.
Let us first estimate $I_{2}$.
Notice that 
\begin{align*}
	(Q_t^B)^M(I-\phi_{r}(\Pi_B)) = (\frac{r}{t})^{M}(t^2\Pi_B^2)^M(I+t^2\Pi_B^2)^{-M} \tilde{\phi}(r\Pi_B)
	\qquad \forall t,r>0,
\end{align*}
for $\tilde{\phi}(z)=\left(\frac{(1+z^2)^M-1}{z(1+z^2)^M}\right)^M$,
and that $\{(t^2\Pi_B^2)^M(I+t^2\Pi_B^2)^{-M} \;;\; t>0\}$ has $L^2$-$L^2$ off-diagonal bounds. In order to show that 
 $\{\tilde{\phi}(r\Pi_B) \;;\; r>0\}$ has $\dot{W}^{1,p}$-$L^2$ off-diagonal bounds of order  $N'>\frac{n}{2}$ on balls, 
write $\tilde{\phi}(z)=(1+z^2)^{\frac{M}{2}-1} \tilde{\phi}(z) (1+z^2)^{-(\frac{M}{2}-1)}$. By Lemma \ref{lem:psiOD}, $\{(1+r^2\Pi_B^2)^{\frac{M}{2}-1} \tilde{\phi}(r\Pi_B)\;;\; r>0\}$ has $L^2$-$L^2$ off-diagonal estimates of every order. On the other hand,  $\{(I+r^2\Pi_B^2)^{-(\frac{M}{2}-1)}\Gamma\;;\;r \in \R^{*}\}$   
has $\dot{W}^{1,p}$-$L^2$ off-diagonal bounds of every order on balls.  Combining the two families of operators gives the statement (see Lemma \ref{lem:ODcomp}).  
Thus $\{(\frac{r}{t})^{n(\frac{1}{p}-\frac{1}{2})}(t^2\Pi_B^2)^M(I+t^2\Pi_B^2)^{-M} \tilde{\phi}(r\Pi_B)\Gamma \;;\; t>0\}$  has   $\dot{W}^{1,p}$-$L^2$ off-diagonal bounds of order $N'>\frac{n}{2}$ on balls, again by Lemma \ref{lem:ODcomp}, whenever $0<r <t$. 
In particular, $\|\psi_t(\Pi_B)(I-\phi_{r}(\Pi_B))\|_{\calL(L^p,L^2)} \lesssim (\frac{r}{t})^{\tilde{M}}
t^{-n(\frac{1}{p}-\frac{1}{2})}$, for $\tilde{M}:=M-n(\frac{1}{p}-\frac{1}{2})$.
This gives 
\begin{equation*}
\begin{split}
I_{2} &= (\frac{1}{|2^{j+1}B|}  \int \limits _{2^{j-2}r} ^{\infty} \int \limits _{2^{j-1}B} |\psi_{t}(\Pi_B)(I-\phi_{r}(\Pi_B))\Gamma f(t,y)|^{2} \frac{dydt}{t})^{\frac{1}{2}}\\
&\lesssim (\frac{1}{|2^{j+1}B|}  \int \limits _{2^{j-2}r} ^{\infty} ((\frac{r}{t})^{\tilde{M}}
t^{-n(\frac{1}{p}-\frac{1}{2})} \|\nabla f(t,.)\|_{p})^{2} \frac{dt}{t})^{\frac{1}{2}}\\
&\lesssim  
(\int \limits _{2^{j-2}r} ^{\infty} ((2^{j}r)^{-\frac{n}{2}}(\frac{r}{t})^{\tilde{M}}
t^{-n(\frac{1}{p}-\frac{1}{2})} \|\nabla f(t,.)\|_{p})^{2} \frac{dt}{t})^{\frac{1}{2}}.
\end{split}
\end{equation*}

We can estimate the above by 
\begin{align*}
 2^{-j(\tilde{M}+\frac{n}{p})}r^{-\frac{n}{p}}(\int \limits _{2^{j}r} ^{\infty}
\|\nabla f(t,.)\|_{p}^{2} \frac{dt}{t})^{\frac{1}{2}}
&\lesssim 2^{-j(\tilde{M}+\frac{n}{p})}r^{-\frac{n}{p}}
\| \int \limits_{B} |\nabla f(.,y)|^{p}dy\|_{\frac{2}{p}} ^{\frac{1}{p}}\\
&\lesssim  2^{-j(\tilde{M}+\frac{n}{p})}(\frac{1}{|B|} \int \limits _{B} \|\nabla f(.,y)\|_{H_{1}} ^{p} dy)^{\frac{1}{p}}.
\end{align*}
We now estimate 
\begin{align*}
I_1 &\leq J_{1}+J_{2} :=(\frac{1}{|2^{j+1}B|}  \int \limits _{0} ^{r} \int \limits _{(2^{j-1}B)^{c}} |(Q_t^B)^M(I-\phi_{r}(\Pi_B))\Gamma f(t,y)|^{2} \frac{dydt}{t})^{\frac{1}{2}}
\\ & \qquad +
(\frac{1}{|2^{j+1}B|}  \int \limits _{r} ^{\infty} \int \limits _{(2^{j-1}B)^{c}} |(Q_t^B)^M(I-\phi_{r}(\Pi_B))\Gamma f(t,y)|^{2} \frac{dydt}{t})^{\frac{1}{2}}.
\end{align*}
For $J_{1}$, we use that for $0<t<r$, 

$\{(Q_t^B)^M \Gamma f\;;\; t>0\}$ and $\{(Q_t^B)^M\phi_r(\Pi_B)\Gamma \;;\; r>0\}$ have $\dot{W}^{1,p}$-$L^2$ off-diagonal bounds of every order on balls
(since $(Q_{t}^B)^M = (I-P_{t} ^{B})^{\frac{M}{2}} (P_{t}^{B})^{\frac{M}{2}}$ and $\{P_{t}^{B} \;;\; t>0\}$ has $L^2$-$L^2$ off-diagonal bounds of every order). 
Thus,
\begin{equation*}
\begin{split}
(&\frac{1}{|2^{j+1}B|}  \int \limits _{0} ^{r}  \int \limits _{(2^{j-1}B)^{c}} |(Q_t^B)^M(I-\phi_{r}(\Pi_B))\Gamma f(t,y)|^{2} \frac{dydt}{t})^{\frac{1}{2}}
\\ &\leq
(\frac{1}{|2^{j+1}B|}  \int \limits _{0} ^{r} \int \limits _{(2^{j-1}B)^{c}} |(Q_t^B)^M \Gamma f(t,y)|^{2} \frac{dydt}{t})^{\frac{1}{2}} +
(\frac{1}{|2^{j+1}B|}  \int \limits _{0} ^{r} \int \limits _{(2^{j-1}B)^{c}} |(Q_t^B)^M\phi_{r}(\Pi_B) \Gamma f(t,y)|^{2} \frac{dydt}{t})^{\frac{1}{2}} \\
& \lesssim
(\frac{1}{|2^{j+1}B|}  \int \limits _{0} ^{r} (t^{-\frac{n}{p}+\frac{n}{2}} (1+\frac{2^{j}r}{t})^{-N'}\|\nabla f(t,.)\|_{p})^{2} \frac{dt}{t})^{\frac{1}{2}}
+
(\frac{1}{|2^{j+1}B|}  \int \limits _{0} ^{r} (r^{-\frac{n}{p}+\frac{n}{2}} 2^{-jN'}\|\nabla f(t,.)\|_{p})^{2} \frac{dt}{t})^{\frac{1}{2}} \\
& \lesssim 2^{-j(\frac{n}{2}+N')}r^{-\frac{n}{p}}
(\int \limits _{0} ^{r} (\frac{t}{r})^{N'-\frac{n}{p}+\frac{n}{2}} \|\nabla f(t,.)\|_{p})^{2} \frac{dt}{t})^{\frac{1}{2}}
+2^{-j(\frac{n}{2}+N')}r^{-\frac{n}{p}}
(\int \limits _{0} ^{r}  \|\nabla f(t,.)\|_{p})^{2} \frac{dt}{t})^{\frac{1}{2}}\\
& \lesssim
2^{-j(\frac{n}{2}+N')}(\frac{1}{|B|} \int \limits _{B} \|\nabla f(.,y)\|_{H_{1}} ^{p} dy)^{\frac{1}{p}}.
\end{split}
\end{equation*}
Turning to $J_{2}$, we now use that $\{(\frac{r}{t})^{n(\frac{1}{p}-\frac{1}{2})}(t^2\Pi_B^2)^M(I+t^2\Pi_B^2)^{-M} \tilde{\phi}(r\Pi_B)\Gamma \;;\; t>0\}$ has $\dot{W}^{1,p}$-$L^2$ off-diagonal bounds of order $N'>\frac{n}{2}$, which gives 
\begin{equation*}
\begin{split}
(&\frac{1}{|2^{j+1}B|}  \int \limits _{r} ^{\infty} \int \limits _{(2^{j-1}B)^{c}} |(Q_t^B)^M(I-\phi_{r}(\Pi_B))\Gamma f(t,y)|^{2} \frac{dydt}{t})^{\frac{1}{2}} \\
&\lesssim 
(\int \limits _{r} ^{2^{j}r} ((2^{j}r)^{-\frac{n}{2}}(\frac{r}{t})^{\tilde{M}}
t^{-n(\frac{1}{p}-\frac{1}{2})} (\frac{2^{j}r}{t})^{-N'}\|\nabla f(t,.)\|_{p})^{2} \frac{dt}{t})^{\frac{1}{2}}
\\& \qquad+(\int \limits _{2^{j}r} ^{\infty} ((2^{j}r)^{-\frac{n}{2}}(\frac{r}{t})^{\tilde{M}}
t^{-n(\frac{1}{p}-\frac{1}{2})} \|\nabla f(t,.)\|_{p})^{2} \frac{dt}{t})^{\frac{1}{2}}\\
& \lesssim 
2^{-j\frac{n}{p}}r^{-\frac{n}{p}}
(\int \limits _{r} ^{2^{j}r}  ((\frac{r}{t})^{\tilde{M}} (\frac{2^{j}r}{t})^{-(N'-n(\frac{1}{p}-\frac{1}{2}))}\|\nabla f(t,.)\|_{p})^{2} \frac{dt}{t})^{\frac{1}{2}}
+ 2^{-j(\tilde{M}+\frac{n}{p})}(\frac{1}{|B|} \int \limits _{B} \|\nabla f(.,y)\|_{H_{1}} ^{p} dy)^{\frac{1}{p}} \\
&\lesssim (2^{-j(N'+\frac{n}{2})} + 2^{-j(\tilde{M}+\frac{n}{p})})(\frac{1}{|B|} \int \limits _{B} \|\nabla f(.,y)\|_{H_{1}} ^{p} dy)^{\frac{1}{p}}.
\end{split}
\end{equation*}

Combining all the estimates gives
$$
(\frac{1}{|2^{j+1}B|}  \int \limits _{S_{j}(B)} \|T(I-\phi_{r}(\Pi_{B}))\Gamma f(.,y)\|_{H_{2}} ^{2} dy)^{\frac{1}{2}}
\lesssim j2^{-j\min\{M+\frac{n}{2},N'+\frac{n}{2}\}} (|B|^{-1}\int \limits _{B} \|\nabla f(.,y)\|_{H_{1}}^{p}dy)^{\frac{1}{p}},
$$
which shows \eqref{extrapol-cond2}, given that $M>\frac{n}{2}$ and $N'>\frac{n}{2}$. 
\end{proof}

\section{High frequency estimates for \texorpdfstring{$p \in (\max\{1,(p_{H})_{*}\},2]$}{p (max{1,pH*},2]}}
\label{sec:high}

In this section, we finally prove Theorem \ref{thm:high} in the case $p \in (\max\{1,(p_H)_\ast\},2]$.\\

The idea of the proof is to show, using Corollary \ref{Cor:SIO-tent}, that the integral operator defined by
\[
		T_K(F)(t,\,.\,) 
		:= \int_0^\infty K(t,s) F(s,\,.\,)\, \frac{ds}{s} 
\]
with $K(t,s):=(Q_t^B)^{M} (I-{P_{t}}^{\tilde{N}}){Q_s}^{2{\tilde{N}}}$ and $M$ sufficiently large, 
extends to a bounded operator on tent spaces. The square function estimate of Theorem \ref{thm:high} is then reduced to the square function estimate for the unperturbed operator shown in Proposition \ref{lem:hardyPi}.

\begin{Theorem}
\label{thm:fullHF}
Let $M$ be even and such that $M\geq10n$, and let $p \in (1,2]$.
Suppose $\Pi_B$ is a perturbed Hodge-Dirac operator,  such that $(\Pi_{B}(p))$ holds.
 Assume that $\{t\Gamma P_{t}^B \;;\; t \in \R^{*}\}$ and $\{t\Gamma^{*}_{B} P_{t}^B \;;\; t \in \R^{*}\}$ are uniformly bounded in $\calL(L^{p})$ and that  
 $\{(R_t^B)^{\frac{M}{2}-2} \Gamma \;;\; t \in \R^{*}\}$ and $\{(R_t^B)^{\frac{M}{2}-2} B_{1}\Gamma^{*} \;;\; t \in \R^{*}\}$ have $\dot{W}^{1,p}$-$L^2$ off-diagonal bounds of every order on balls.

 Then, for all $q \in (\max\{1,p_{*}\},2]$,
$$
\|(t,x) \mapsto (Q_{t}^{B})^{M}(I-{P_{t}}^{\tilde{N}})u(x)\|_{T^{q,2}} \leq C_q \|u\|_{{q}} \quad \forall u \in \overline{\calR_{q}(\Gamma)}.
$$\end{Theorem}

\begin{proof}
 Let $p \in (1,2]$. Let $u \in \calR_{p}(\Pi_{B})$.
By Theorem \ref{thm:p=2}, Lemma \ref{tent-boundedness}, and using the reproducing formula $u = C \int \limits _{0} ^{\infty} {Q_{s}}^{2\tilde{N}}u\frac{ds}{s}$ for some constant $C$, we have that:
\begin{align*}
\|(t,x) \mapsto \int \limits _{t} ^{\infty} (Q_{t}^{B})^{M}(I-{P_{t}}^{\tilde{N}}){Q_{s}}^{2\tilde{N}}u(x)\frac{ds}{s}\|_{T^{p,2}} 
& \lesssim
\|(t,x) \mapsto \int \limits _{t} ^{\infty} (I-{P_{t}}^{\tilde{N}}){Q_{s}}^{2\tilde{N}}u(x)\frac{ds}{s}\|_{T^{p,2}} \\
& = 
\|(t,x) \mapsto \psi(t\Pi)u(x)\|_{T^{p,2}},
\end{align*}
for $\psi(z) = (1-(1+z^{2})^{-\tilde{N}})\int \limits _{1} ^{\infty} (\frac{zs}{1+(zs)^{2}})^{2\tilde{N}} \frac{ds}{s}$. 
Therefore 
\begin{align*}
\|(t,x) \mapsto \int \limits _{t} ^{\infty} (Q_{t}^{B})^{M}(I-{P_{t}}^{\tilde{N}}){Q_{s}}^{2\tilde{N}}u(x)\frac{ds}{s}\|_{T^{p,2}} \lesssim
\|u\|_{{p}},
\end{align*}
by Theorem \ref{thm:hardy} and Proposition \ref{lem:hardyPi},  since $\psi \in \Psi_{2} ^{\tilde{N}}$. 
Now let $q \in (\max\{1,p_\ast\},2]$.
For $u \in \calR_{2}(\Gamma)\cap L^{q}(\R^{n};\C^{N})$, we have that 
\begin{align*}
\|(t,x) & \mapsto \int \limits _{0} ^{t} (Q_{t}^{B})^{M}(I-{P_{t}}^{\tilde{N}}){Q_{s}}^{2\tilde{N}}u(x)\frac{ds}{s}\|_{T^{q,2}} \\
&= \|(t,x) \mapsto \int \limits _{0} ^{t} (\frac{s}{t})t\Gamma(Q_{t}^{B})^{M}
(I-{P_{t}}^{\tilde{N}})P_{s}{Q_{s}}^{{\tilde{N}}-1}
{Q_{s}}^{\tilde{N}}u(x)\frac{ds}{s}\|_{T^{q,2}}\\
&\lesssim
\|(t,x) \mapsto \int \limits _{0} ^{t} (\frac{s}{t})(Q_{t}^{B})^{M-2}
(I-{P_{t}}^{\tilde{N}})P_{s}{Q_{s}}^{{\tilde{N}}-1}
{Q_{s}}^{\tilde{N}}u(x)\frac{ds}{s}\|_{T^{q,2}},
\end{align*}
where in the last step we have used
Lemma \ref{Lemma:off-diag-est} and Lemma \ref{tent-boundedness}.
We now consider the integral operator $T_{K}$ with kernel
$$
K(t,s) = \Eins_{(0,\infty)}(t-s)(Q_{t}^{B})^{M-2}
(I-{P_{t}}^{\tilde{N}})P_{s}{Q_{s}}^{{\tilde{N}}-1}.
$$
Using the results of Section \ref{sec:tech1}, we aim to show that $T_{K_{1}^{+}}$ extends to a bounded operator on $T^{q,2}$.
The result then follows from Proposition \ref{lem:hardyPi}.\\
From our assumption, we have that 
$\{(P_t^B)^{\frac{M}{2}-2}\Gamma \;;\; t \in \R^{*}\}$ and $\{(P_t^B)^{\frac{M}{2}-2}B_{1}\Gamma^{*} \;;\; t \in \R^{*}\}$
 have $\dot{W}^{1,p}$-$L^{2}$ off-diagonal bounds of every order on balls. 
Since $(Q_{t}^{B})^{M-4} = (I-P_{t}^{B}) ^{\frac{M}{2}-2}.(P_{t}^{B}) ^{\frac{M}{2}-2}$ and $\{P_{t}^{B} \;;\; t \in \R\}$ has $L^2$-$L^2$ off-diagonal bounds of every order, we have that 
$\{(Q_t^B)^{M-4}\Gamma \;;\; t \in \R^{*}\}$ and $\{(Q_t^B)^{M-4}B_{1}\Gamma^{*} \;;\; t \in \R^{*}\}$ have $\dot{W}^{1,p}$-$L^{2}$ off-diagonal bounds of every order on balls. 
This gives, for $u \in L^{2}$,
\begin{align*}
\| & (Q_{t}^{B})^{M-2} 
(I-{P_{t}}^{\tilde{N}})P_{s}{Q_{s}}^{{\tilde{N}}-1}u\|_{2} \\
&\leq \|(Q_{t}^{B})^{M-3}\Gamma S_{\Gamma}t\Gamma P_{t}^{B}
(I-{P_{t}}^{\tilde{N}})P_{s}{Q_{s}}^{{\tilde{N}}-1}u\|_{2}
+
\|(Q_{t}^{B})^{M-3}B_{1}\Gamma^{*} S_{\Gamma^{*}}t\Gamma^{*}B_{2}P_{t}^{B}
(I-{P_{t}}^{\tilde{N}})P_{s}{Q_{s}}^{{\tilde{N}}-1}u\|_{2}\\
&\lesssim \|t\Gamma P_{t}^{B}(I-{P_{t}}^{\tilde{N}})P_{s}{Q_{s}}^{{\tilde{N}}-1}u\|_{p}
+
\|t\Gamma^{*}B_{2}P_{t}^{B}(I-{P_{t}}^{\tilde{N}})P_{s}{Q_{s}}^{{\tilde{N}}-1}u\|_{p} \\
&\lesssim
\|u\|_{p} + \|tB_{1}\Gamma^{*}B_{2}P_{t}^{B}(I-{P_{t}}^{\tilde{N}})P_{s}{Q_{s}}^{{\tilde{N}}-1}u\|_{p} \lesssim \|u\|_{p},
\end{align*}
where we have used the bisectoriality of the unperturbed operator in $L^p$ (see Proposition \ref{lem:hardyPi}), the assumption that
$\{t\Gamma^{*}_{B} P_{t}^B \;;\; t \in \R^{*}\}$ and $\{t\Gamma P_{t}^B \;;\; t \in \R^{*}\}$ are uniformly bounded in $\calL(L^{p})$, and the properties of the potential maps (see Proposition \ref{prop:pi}).

Using Lemma \ref{lem:ODcomp}, we thus get that $K$ satisfies (\ref{kernel-cond-Lq}) with $\max\{t,s\}=t$ for all $r \in (p,2]$. 
To conclude the proof using Corollary \ref{Cor:SIO-tent}, we thus only have to show that
$$ \underset{\gamma \in \R}{\sup} \|T_{K^{+}_{\varepsilon+i\gamma}}\|_{\calL(T^{r,2})}<\infty \quad \forall \varepsilon>0 \quad \forall r \in (p,2].$$
To do so we use Lemma \ref{tent-boundedness}, Lemma \ref{lem:psiOD}, and Theorem \ref{thm:conicalVSvertical},
and obtain the following, for $\varepsilon>0$, $r \in (p,2]$, $F\in T^{r,2}$,
and $\gamma \in \R$ (with implicit constants independent of $F$ and $\gamma$): 
\begin{equation*}
\begin{split}
\|T_{K_{\varepsilon+i\gamma} ^{+}}F\|_{T^{r,2}} &=
 \|(t,x) \mapsto \int \limits _{0} ^{t} 
(\frac{s}{t})^{\varepsilon+i\gamma}  (Q_{t}^{B})^{M-2}
(I-{P_{t}}^{\tilde{N}})P_{s}{Q_{s}}^{{\tilde{N}}-1}
 F(s,x)\frac{ds}{s}\|_{T^{r,2}}
\\ &\lesssim \|(t,x) \mapsto \int \limits _{0} ^{t} 
(\frac{s}{t})^{\varepsilon+i\gamma}  
(I-{P_{t}}^{\tilde{N}})P_{s}{Q_{s}}^{{\tilde{N}}-1}
 F(s,x)\frac{ds}{s}\|_{L^{r}(\R^{n};L^{2}((0,\infty),\frac{dt}{t}))}
 \\ & = \|T_{\widetilde{K} ^{+} _{\varepsilon+i\gamma}}F\|_{L^{r}(\R^{n};L^{2}((0,\infty),\frac{dt}{t}))},
\end{split}
\end{equation*}
where $\widetilde{K}(t,s) =  (I-{P_{t}}^{\tilde{N}})P_{s}{Q_{s}}^{\tilde{N}-1}$.
Since the unperturbed operator $\Pi$ has a bounded $H^{\infty}$ functional calculus in $L^r$, the family $\{\tilde{K}(t,s) \;;\; t,s>0\}$ is R-bounded in $L^r$ by \cite[Theorem 5.3]{kw}. Therefore, Lemma \ref{lem:vertSchur} gives
$$
\|T_{K_{\varepsilon+i\gamma} ^{+}}F\|_{T^{r,2}} \lesssim \|F\|_{L^{r}(\R^{n};L^{2}((0,\infty),\frac{dt}{t}))}.
$$
We conclude the proof using \cite[Proposition 2.1]{AHM} to get that $\|F\|_{L^{r}(\R^{n};L^{2}((0,\infty),\frac{dt}{t}))} \lesssim \|F\|_{T^{r,2}}$.
\end{proof}

\begin{Cor}
\label{cor:fullHF}
Suppose $\Pi_B$ is a perturbed Hodge-Dirac operator. 
Suppose $p \in (\max\{1,(p_{H})_{*}\},2]$ and $M\in \N$ with  $M\geq 10n$.  Then
$$
\|(t,x) \mapsto (Q_{t}^{B})^{M}(I-{P_{t}}^{\tilde{N}})u(x)\|_{T^{p,2}} \leq C_p \|u\|_{p} \quad \forall u \in \overline{\calR_{2}(\Gamma)}\cap L^{p}(\R^{n};\C^{N}).
$$\end{Cor}

\begin{proof}
For $p=2$, $\Pi_B$ is bisectorial in $L^2(\R^n;\C^N)$ by Theorem \ref{thm:p=2}, and   
 $\{(R_t^B)^{\frac{M}{2}-2} \Gamma \;;\; t \in \R^{*}\}$ and $\{(R_t^B)^{\frac{M}{2}-2} B_{1}\Gamma^{*} \;;\; t \in \R^{*}\}$ have $\dot{W}^{1,2}$-$L^2$ off-diagonal bounds of every order on balls, since $\{R_t^B  \;;\; t \in \R^{*}\}$ has $L^{2}$-$L^2$ off-diagonal bounds of every order.
 Therefore for all $q \in (2_{*},2]$,
$$
\|(t,x) \mapsto (Q_{t}^{B})^{M}(I-{P_{t}}^{\tilde{N}})u(x)\|_{T^{q,2}} \leq C_q \|u\|_{{q}} \quad \forall u \in \overline{\calR_{q}(\Gamma)},$$
by Theorem \ref{thm:fullHF}.
Combined with Theorem \ref{thm:low}, this gives
$$
\|(t,x) \mapsto (Q_{t}^{B})^{M}u(x)\|_{T^{q,2}} \leq C_q \|u\|_{{q}} \quad \forall u \in \overline{\calR_{q}(\Gamma)}.$$
Since the same holds for $\underline{\Pi_{B}}$ instead of $\Pi_{B}$, we have, as in the proof of Corollary \ref{cor:hardy}, that
$H^{r}_{\Pi_{B}} = \overline{\calR_{r}(\Pi_{B})}$ for all $r \in (\max(p_{H},2_{*}),2]$. In particular, $\Pi_{B}$ is bisectorial in $L^r$.
Moreover, applying Proposition \ref{prop:induction}, we have that  $\{(R_t^B)^{\frac{M}{2}-2} \Gamma \;;\; t \in \R^{*}\}$ and $\{(R_t^B)^{\frac{M}{2}-2} B_{1}\Gamma^{*} \;;\; t \in \R^{*}\}$ have $\dot{W}^{r,2}$-$L^2$ off-diagonal bounds of every order on balls.
The assumptions of Theorem \ref{thm:fullHF} and Proposition \ref{prop:induction} are now satisfied in $L^r$. 
 Note that $(\Pi_{B}(r))$ holds as long as $r>p_{H}$ by Proposition \ref{equiv}.
We can repeat the argument finitely many times until we reach a value of $p$ such that $p_{*}<p_{H}$.
\end{proof}

If we restrict the off-diagonal bound assumptions to certain subspaces, the following restricted version of the theorem remains valid.

\begin{Cor}
\label{cor:fullHFW}
 Let $p \in (1,2)$, $M\in \N$  be even and such that $M \geq 10n$, and let
 $\Pi_B$ be a perturbed Hodge-Dirac operator such that $(\Pi_{B}(p))$ holds.
Let $W$ be a subspace of $\C^{N}$ that is stable under $\widehat{\Gamma^{*}}(\xi)\widehat{\Gamma}(\xi)$ and $\widehat{\Gamma}(\xi)\widehat{\Gamma^{*}}(\xi)$ for all $\xi \in \R^{n}$.
Assume that  $\{t\Gamma P_{t}^B \mathbb{P}_{W}\;;\; t \in \R^{*}\}$
$\{t\Gamma^{*}_{B} P_{t}^B \mathbb{P}_{W}\;;\; t \in \R^{*}\}$ are uniformly bounded in $\calL(L^{p})$, and that  
 $\{(R_t^B)^{\frac{M}{2}-2} \Gamma \mathbb{P}_{W} \;;\; t \in \R^{*}\}$ and $\{(R_t^B)^{\frac{M}{2}-2} B_{1}\Gamma^{*} \mathbb{P}_{W}\;;\; t \in \R^{*}\}$ have $\dot{W}^{1,p}$-$L^2$ off-diagonal bounds of every order on balls, where $\mathbb{P}_{W}$ denotes the projection from $L^p(\R^{n};\C^{N})$ onto $L^p(\R^{n};W)$.

Then, for all $q \in (\max\{1,p_{*}\},2]$,
$$
\|(t,x) \mapsto (Q_{t}^{B})^{M}(I-{P_{t}}^{\tilde{N}})\Gamma v(x)\|_{T^{q,2}} \leq C_q \|\Gamma v\|_{{q}} \quad \forall v \in \calD_{q}(\Gamma)\cap L^{q}(\R^{n};W).$$
\end{Cor}

\begin{proof}
Let $p \in (1,2)$, $v \in \mathcal{S}(\R^{n};W)$, and $s>0$.
Notice that, for all $\xi \in \R^{n}$, 
$$
s\widehat{\Pi}(\xi)(I+s^{2}\widehat{\Pi^{2}}(\xi))^{-1}\widehat{\Gamma}(\xi)\widehat{v}(\xi)=
s\widehat{\Gamma^{*}}(\xi)\widehat{\Gamma}(\xi)(I+s^{2}\widehat{\Gamma^{*}}(\xi)\widehat{\Gamma}(\xi))^{-1}\widehat{v}(\xi) \in W,
$$
since $\widehat{\Gamma}(\xi)$ is nilpotent, and $W$ is stable under $\widehat{\Gamma^{*}}(\xi)\widehat{\Gamma}(\xi)$.
We thus have that $Q_{s}\Gamma v$ belongs to $L^{p}(\R^{n};W)$. The same reasoning also gives that
$(I-{P_{t}}^{\tilde{N}})P_{s}Q_{s}^{2{\tilde{N}}-1}\Gamma v \in L^{p}(\R^{n};W)$ for all $t,s>0$. Therefore we have that, for all $t>0$ and $x \in \R^{n}$,
\begin{align*}
& \int \limits _{0} ^{t} (\frac{s}{t})(Q_{t}^{B})^{M-2}
(I-{P_{t}}^{\tilde{N}})P_{s}{Q_{s}}^{{\tilde{N}}-1}
{Q_{s}}^{\tilde{N}}\Gamma v(x)\frac{ds}{s} \\
& \quad = \int \limits _{0} ^{t} (\frac{s}{t})(Q_{t}^{B})^{M-2}\mathbb{P}_{W}
(I-{P_{t}}^{\tilde{N}})P_{s}{Q_{s}}^{{\tilde{N}}-1}
{Q_{s}}^{\tilde{N}}\Gamma v(x)\frac{ds}{s}.
\end{align*}
This allows us to use the proof of Theorem \ref{thm:fullHF}, replacing the kernel $K$ by
$\widetilde{K}(t,s) = \Eins_{(0,\infty)}(t-s)(Q_{t}^{B})^{M-2}\mathbb{P}_{W}
(I-{P_{t}}^{\tilde{N}})P_{s}{Q_{s}}^{{\tilde{N}}-1}$, for all $t,s>0$.\end{proof}

\section{Appendix: Schur estimates in tent spaces}
\label{sec:tech1}

We need boundedness criteria for integral operators of the form
$$
T_{K}F(t,x) = \int \limits _{0} ^{\infty} K(t,s)F(s,.)(x)\frac{ds}{s} 
\qquad \forall F \in T^{p,2}(\R^{n+1}_+;\C^N),
$$
where $\{K(t,s) \;;\; t,s>0\}$ is a uniformly bounded family of bounded linear operators acting on $L^2(\R^n;\C^N)$.
We are interested here in Schur type estimates, i.e. estimates for integral operators with kernels satisfying  size conditions of the form $\|K(t,s)\| \lesssim \min(\frac{t}{s},\frac{s}{t})^{\alpha}$ for some $\alpha>0$. 
The proofs are similar to those developed in \cite{akmp} to treat singular integral operators with kernels 
satisfying  size conditions of the form $\|K(t,s)\| \lesssim |t-s|^{-1}$. The appropriate off-diagonal bound assumptions are as follows.\\

Let $p \in [1,2]$.
Let $\{K(t,s), \; s,t>0\}$  be a uniformly bounded family of bounded linear operators acting on $L^2(\R^n;\C^N)$ that satisfies $L^p$-$L^2$ off-diagonal bounds of the following form: there exists $C>0$, $N'>0$, such that for all Borel sets $E,F \subseteq \R^n$ and all $s,t>0$
\begin{equation} 
\label{kernel-cond-Lq}
	\|\Eins_EK(t,s)\Eins_F\|_{L^p\to L^2} 
	\leq C  \max\{t,s\}^{-n(\frac{1}{p}-\frac{1}{2})} \left(1+\frac{\dist(E,F)}{\max\{t,s\}}\right)^{-N'}.
\end{equation}

Given a kernel $K$, we also consider
\begin{align*}
K^{+}_{z}(t,s) &= \Eins_{(0,\infty)}(t-s)(\frac{s}{t})^{z}K(t,s),\\
K^{-}_{z}(t,s) &= \Eins_{(0,\infty)}(s-t)(\frac{t}{s})^{z}K(t,s), \qquad \forall t,s \in (0,\infty), \quad \forall z \in \C.
\end{align*}

We then obtain the following result on $T^{1,2}(\R^{n+1}_+;\C^N)$, which is a refined version of the arguments in \cite[Theorem 4.9]{AMR}.

\begin{Prop} \label{prop:SIO-tent}
Suppose $K$ satisfies (\ref{kernel-cond-Lq}) for some $N'>\frac{n}{2}$ and $p\in [1,2]$. Then the following holds.\\
(1) Given $\alpha \in (0,\infty)$, we have $T_{K^{-}_{\alpha}} \in \calL(T^{1,2}(\R^{n+1}_+;\C^N))$. \\
(2) Given $\beta \in (\frac{n}{p'},\infty)$ and $\gamma \in \R$, we have $T_{K^{+}_{\beta+i\gamma}} \in \calL(T^{1,2}(\R^{n+1}_+;\C^N))$ with 
$$\underset{\gamma \in \R}{\sup} \, \|T_{K^{+}_{\beta+i\gamma}}\|_{\calL(T^{1,2})}<\infty.$$
\end{Prop}

\begin{Cor} 
\label{Cor:SIO-tent}
Suppose $K$ satisfies (\ref{kernel-cond-Lq}) for some $N'>\frac{n}{2}$ and $p\in [1,2]$. 
Suppose 
$q \in [1,p]$, $\alpha \in (0,\infty)$, and $\beta \in (n(\frac{1}{q}-\frac{1}{p}),\infty)$.
If $\underset{\gamma \in \R}{\sup} \, \|T_{K^{+}_{i\gamma}}\|_{\calL(T^{p,2})}<\infty$,
then $T_{K^{-}_{\alpha}+K^{+}_{\beta}} \in \calL(T^{q,2}(\R^{n+1}_+;\C^N))$.
\end{Cor}

\begin{proof}[Proof of Corollary \ref{Cor:SIO-tent}]
This follows from Proposition \ref{prop:SIO-tent} by applying Stein's interpolation \cite[Theorem 1]{Stein} to the analytic family of operators
$\{T_{K^{-}_{\frac{n}{p'}z}} \;;\; Re(z) \in [0,1]\}$. We choose the spaces $T^{p,2}$ and $T^{1,2}$ as endpoints, and set $\frac{1}{q}=\frac{1-\theta}{p}+\theta$ for $\theta \in [0,1]$. This gives $\theta=p'(\frac{1}{q}-\frac{1}{p})$ and thus the condition $\beta>\frac{n}{p'}\theta =n(\frac{1}{q}-\frac{1}{p})$.
\end{proof}

We now turn to the proof of Proposition \ref{prop:SIO-tent}, which follows the one of \cite[Theorem 4.9]{AMR}.

\begin{proof}[Proof of Proposition \ref{prop:SIO-tent}]
Let $\alpha>0$, $\beta> \frac{n}{p'}$.
It suffices to show that
\[
	\|T_{K^{-}_{\alpha}+K^{+}_{\beta+i\gamma}}F\|_{T^{1,2}} \leq C
\]
uniformly for all atoms $F$ in $T^{1,2}$ and all $\gamma \in \R$.\\
Let $F$ be a $T^{1,2}$ atom associated with a ball $B \subseteq \R^n$ of radius $r>0$. Then
\[
	\iint_{T(B)} |F(s,x)|^2\,\frac{dxds}{s} \leq |B|^{-1},
\]
where $T(B)=(0,r) \times B$.
Set $\tilde{K} = K^{-}_{\alpha}+K^{+}_{\beta+i\gamma}$, $\tilde{F}:=T_{\tilde{K}}(F)$, and $\tilde{F}_1:=\tilde{F}\Eins_{T(4B)}, \; \tilde{F}_k:=\tilde{F}\Eins_{T(2^{k+1}B)\setminus T(2^kB)}, \; k \geq 2$.
We show that there exists $\delta>0$, independent of $\gamma$, such that
\[
	\iint |\tilde{F}_k(t,x)|^2 \,\frac{dxdt}{t} \lesssim 2^{-k\delta} |2^{k+1}B|^{-1}.
\]

Let $k=1$. Observe that for every $\eps>0$,
$
	\int_0^\infty \min\left(\frac{s}{t},\frac{t}{s}\right)^\eps \,\frac{ds}{s} \leq C,
$
uniformly in $t>0$. Using Minkowski's inequality and the assumption on $K$, we obtain
\begin{align*}
	\iint_{T(4B)} |\tilde{F}(t,x)|^2\,\frac{dxdt}{t}
	&\leq \int_0^{l(4B)} \int_{4B} (\int_0^\infty \min(\frac{s}{t},\frac{t}{s})^{\min(\alpha,\beta)}|K(t,s)F(s,\,.\,)(x)|\,\frac{ds}{s})^2\,\frac{dxdt}{t}\\
	&\leq \int_0^{l(4B)} (\int_0^\infty \min(\frac{s}{t},\frac{t}{s})^{\min(\alpha,\beta)} \|K(t,s)F(s,\,.\,)\|_{L^2(4B)} \,\frac{ds}{s})^2 \,\frac{dt}{t} \\
	& \lesssim \int_0^\infty (\int_0^\infty \min(\frac{s}{t},\frac{t}{s})^{\min(\alpha,\beta) }\|F(s,\,.\,)\|_{L^2(B)} \,\frac{ds}{s})^2 \,\frac{dt}{t} \\
	& \lesssim \int_0^\infty \|F(s,\,.\,)\|^2_{L^2(B)} \,\frac{ds}{s} \leq |B|^{-1}.
\end{align*}

For $k \geq 2$, we cover $T(2^{k+1}B)\setminus T(2^kB)$ with the  two parts $(0,2^{k-1}r) \times 2^{k+1}B\setminus 2^{k-1}B$ and $(2^{k-1}r,2^{k+1}r) \times 2^{k+1}B$.
Via Minkowski's inequality, we have
\begin{align*}
	&(\iint_{T(2^{k+1}B)\setminus T(2^kB)} |\tilde{F}_k(t,x)|^2 \,\frac{dxdt}{t})^{\frac{1}{2}} \\
	& \qquad \leq \int_0^r (\int_0^{2^{k-1}r} \min(\frac{s}{t},\frac{t}{s})^{\min(\alpha,\beta)}\|K(t,s)F(s,\,.\,)\|_{L^2(2^{k+1}B\setminus 2^{k-1}B)}^2 \,\frac{dt}{t})^{\frac{1}{2}} \,\frac{ds}{s} \\
	&\qquad  \quad + \int_0^r (\int_{2^{k-1}r}^{2^{k+1}r} (\frac{s}{t})^{\beta}\|K(t,s)F(s,\,.\,)\|_{L^2(2^{k+1}B)}^2 \,\frac{dt}{t})^{\frac{1}{2}} \,\frac{ds}{s} \\
	& \qquad =:I_1 + I_2.
\end{align*}

For $I_2$, the fact that $s<t$ and the assumed $L^p$-$L^2$ boundedness of $K(t,s)$ yield 
\begin{align*}
	I_2
	&\leq \int_0^r (\int_{2^{k-1}r}^{2^{k+1}r} (\frac{s}{t})^{\beta} \|K(t,s)F(s,\,.\,)\|_{L^2(\R^n)}^2 \,\frac{dt}{t})^{\frac{1}{2}} \,\frac{ds}{s} \\
	& \lesssim \int_0^r (\int_{2^{k-1}r}^{2^{k+1}r} ( (\frac{s}{t})^{\beta} t^{-n(\frac{1}{p}-\frac{1}{2})} \|F(s,\,.\,)\|_{L^p(B)})^2 \,\frac{dt}{t})^{\frac{1}{2}} \,\frac{ds}{s} \\
	& \lesssim \int_0^r (\frac{s}{2^k r})^\beta (2^kr)^{-n(\frac{1}{p}-\frac{1}{2})} r^{n(\frac{1}{p}-\frac{1}{2})} \|F(s,\,.\,)\|_{L^2(B)} \,\frac{ds}{s} \\
	& \lesssim 2^{-k(\beta+n(\frac{1}{p}-\frac{1}{2}))} (\int_0^r (\frac{s}{r})^{2\beta} \,\frac{ds}{s})^{\frac{1}{2}} (\int_0^r \|F(s,\,.\,)\|_{L^2(B)}^2 \,\frac{ds}{s})^{\frac{1}{2}} \\
	&\lesssim 2^{-k(\beta+n(\frac{1}{p}-\frac{1}{2})-\frac{n}{2})} |2^kB|^{-\frac{1}{2}}. 
\end{align*}
Since, by assumption, $\beta>n(1-\frac{1}{p})$, this yields the desired estimate for $I_2$.\\
We split the term $I_1$ into the two parts
\begin{align*}
	I_1
	&\leq \int_0^r (\int_0^s (\frac{t}{s})^{\alpha}\|K(t,s)F(s,\,.\,)\|_{L^2(2^{k+1}B \setminus 2^{k-1}B)}^2 \,\frac{dt}{t})^{\frac{1}{2}} \,\frac{ds}{s} \\
	& \quad + \int_0^r (\int_s^{2^{k-1}r}(\frac{s}{t})^{\beta}\|K(t,s)F(s,\,.\,)\|^{2} _{L^2(2^{k+1}B \setminus 2^{k-1}B)} \,\frac{dt}{t})^{\frac{1}{2}} \,\frac{ds}{s} \\
	&=: I_{1,1} + I_{1,2}.
\end{align*}
For $I_{1,1}$, we have $t<s$. The assumed $L^p$-$L^2$ off-diagonal bounds and the fact that $N'>n(\frac{1}{p}-\frac{1}{2})$ yield
\begin{align*}
	I_{1,1}
	&\lesssim \int_0^r (\int_0^s ((\frac{t}{s})^\alpha s^{-n(\frac{1}{p}-\frac{1}{2})} (1+\frac{\dist(B,2^{k-1}B)}{s})^{-N'} \|F(s,\,.\,)\|_{L^p(B)} )^2 \,\frac{dt}{t})^{\frac{1}{2}} \,\frac{ds}{s} \\
	& \lesssim \int_0^r \|F(s,\,.\,)\|_{L^2(B)} (\frac{s}{r})^{-n(\frac{1}{p}-\frac{1}{2})} (\frac{s}{2^kr})^{N'} (\int_0^s (\frac{t}{s})^{2\alpha} \,\frac{dt}{t})^{\frac{1}{2}} \,\frac{ds}{s} \\
	& \lesssim 2^{-kN'} (\int_0^r \|F(s,\,.\,)\|_{L^2(B)}^2 \,\frac{ds}{s})^{\frac{1}{2}} (\int_0^r (\frac{s}{r})^{2N'-2n(\frac{1}{p}-\frac{1}{2})} \,\frac{ds}{s})^{\frac{1}{2}} \\
	& \lesssim 2^{-kN'} |B|^{-\frac{1}{2}}
	\lesssim 2^{-k(N'-\frac{n}{2})} |2^kB|^{-\frac{1}{2}}.
\end{align*}
Since $N'>\frac{n}{2}$, this yields the assertion for $I_{1,1}$.\\
For $I_{1,2}$, we have $s<t$. According to our assumptions, there exists $\tilde{N}>0$ with $\frac{n}{2}<\tilde{N}<\max(N',\beta+n(\frac{1}{p}-\frac{1}{2}))$. Using the $L^p$-$L^2$ off-diagonal bounds, we get
\begin{align*}
	I_{1,2}
	& \lesssim \int_0^r (\int_s^{2^{k-1}r} ((\frac{s}{t})^\beta t^{-n(\frac{1}{p}-\frac{1}{2})} (1+\frac{2^kr}{t})^{-N'} \|F(s,\,.\,)\|_{L^p(B)})^2 \,\frac{dt}{t})^{\frac{1}{2}} \,\frac{ds}{s} \\
	& \lesssim 2^{-kN'} \int_0^r \|F(s,\,.\,)\|_{L^2(B)} (\int_s^r (\frac{s}{t})^{2\beta} (\frac{t}{r})^{2N'-2n(\frac{1}{p}-\frac{1}{2})} \,\frac{dt}{t})^{\frac{1}{2}} \,\frac{ds}{s} \\
	& \quad + 2^{-k\tilde{N}} \int_0^r \|F(s,\,.\,)\|_{L^2(B)} (\int_r^{2^{k-1}r} (\frac{s}{t})^{2\beta} (\frac{t}{r})^{2\tilde{N} -2n(\frac{1}{p}-\frac{1}{2})} \,\frac{dt}{t})^{\frac{1}{2}} \,\frac{ds}{s} \\
	& \lesssim 2^{-k\tilde{N}} \int_0^r (\frac{s}{r})^{\beta} \|F(s,\,.\,)\|_{L^2(B)} \,\frac{ds}{s} \\
	& \lesssim 2^{-k\tilde{N}} (\int_0^r (\frac{s}{r})^{2\beta} \,\frac{ds}{s})^{\frac{1}{2}} (\int_0^r \|F(s,\,.\,)\|_{L^2(B)}^2 \,\frac{ds}{s})^{\frac{1}{2}} 
	\lesssim 2^{-k(\tilde{N}-\frac{n}{2})} |2^kB|^{-\frac{1}{2}},
\end{align*}
where again we use the assumptions $N'>\frac{n}{2}$ and $\beta>n(1-\frac{1}{p})$.
\end{proof}

We conclude this section by pointing out that such estimates are much simpler in the context of vertical, rather than conical, square functions. In particular we have the following lemma (see \cite[Section 5]{jan-survey} for the relevant information regarding R-boundedness).
\begin{Lemma}
\label{lem:vertSchur}
Suppose $p \in (1,\infty)$ and $\varepsilon>0$.
If $\{K(t,s) \;;\; t,s>0\}$ is a R-bounded family of bounded operators on $L^p(\R^{n})$, then $T_{K^{+}_{\varepsilon}+K^{-}_{\varepsilon}} \in \mathcal{L}(L^{p}(\R^{n};L^{2}(\R_{+};\frac{dt}{t})))$.
\end{Lemma}

\begin{proof}
Let $F \in L^{p}(\R^{n};L^{2}(\R_{+};\frac{dt}{t}))$.
By Kalton-Weis' $\gamma$-multiplier theorem (see \cite[Theorem 5.2]{jan-survey}),
 we have the following:
 \begin{equation*}
\begin{split}
\|T_{K^{+}_{\varepsilon}}F\|_{L^{p}(\R^{n};L^{2}(\R_{+};\frac{dt}{t}))}
&\leq \int \limits _{0} ^{1} s^{\varepsilon} \|(x,t) \mapsto K(t,ts)F(x,ts)\|_{L^{p}(\R^{n};L^{2}(\R_{+};\frac{dt}{t}))} \frac{ds}{s}\\
&\lesssim \int \limits _{0} ^{1} s^{\varepsilon} \|(x,t) \mapsto F(x,ts)\|_{L^{p}(\R^{n};L^{2}(\R_{+};\frac{dt}{t}))} \frac{ds}{s}
\\&= \int \limits _{0} ^{1} s^{\varepsilon} \|(x,t) \mapsto F(x,t)\|_{L^{p}(\R^{n};L^{2}(\R_{+};\frac{dt}{t}))} \frac{ds}{s} \lesssim \|F\|_{L^{p}(\R^{n};L^{2}(\R_{+};\frac{dt}{t}))}.
\end{split}
\end{equation*}
The same reasoning applies to $T_{K^{-}_{\varepsilon}}$.
\end{proof}

\vspace{3ex}

{\flushleft{\sc Dorothee Frey}}\\
Australian National University, Mathematical Sciences Institute, John Dedman Building,
Acton ACT 0200, Australia.\\
{\tt dorothee.frey@anu.edu.au}\\

{\flushleft{\sc Alan McIntosh}}\\
Australian National University, Mathematical Sciences Institute, John Dedman Building,
Acton ACT 0200, Australia.\\
{\tt alan.mcintosh@anu.edu.au}\\

{\flushleft{\sc Pierre Portal}}\\
Australian National University, Mathematical Sciences Institute, John Dedman Building,
Acton ACT 0200, Australia.\\
{\tt pierre.portal@anu.edu.au}\\

\end{document}